\documentclass[12pt,a4paper]{amsart}

\usepackage{tikz-cd}
\usepackage[margin=1in]{geometry}  % set the margins to 1in on all sides
\usepackage{graphicx}              % to include figures
\usepackage{amsmath,euscript,comment}               % great math stuff
\usepackage{amsfonts}              % for blackboard bold, etc
\usepackage{amssymb,,amscd,epsf,amsthm, epsfig}                % better theorem environments
\usepackage[toc,page]{appendix}
\usepackage{amsmath,calligra,mathrsfs}
\usepackage{float}
\usepackage[symbol]{footmisc}
\usepackage{tikz-cd}

\newtheorem{theorem}{Theorem}[section]
\newtheorem{lemma}[theorem]{Lemma}
\newtheorem{proposition}[theorem]{Proposition}
\newtheorem{corollary}[theorem]{Corollary}

\newtheorem{lemma-definition}[theorem]{Lemma-Definition}
\newtheorem{Claim}[theorem]{Claim}

\theoremstyle{definition}
\newtheorem{definition}[theorem]{Definition}
\newtheorem{example}[theorem]{Example}

\newtheorem{remark}[theorem]{Remark}

\numberwithin{theorem}{section}

\newcommand{\Z}{\mathbb{Z}}
\newcommand{\LL}{\mathbb{L}}

\newcommand{\calC}{\mathcal{C}}
\newcommand{\calD}{\mathcal{D}}

\DeclareMathOperator{\HOM}{\mathscr{H}\text{\kern -3pt {\calligra\large om}}\,}

\newcommand{\DD}{\EuScript D}

\newcommand{\KK}{\EuScript K}

\newcommand{\ac}{\mathop{ac}\nolimits}
\newcommand{\A}{\mathop{\overline{A}}\nolimits}
\newcommand{\sC}{\mathop{s^{-1}\overline{C}}\nolimits}
\newcommand{\sA}{\mathop{s\overline{A}}\nolimits}

\newcommand{\proj}{\mathop{\mathrm{proj}}\nolimits}
\newcommand{\Hom}{\mathop{\mathrm{Hom}}\nolimits}

\newcommand{\Barr}{\mathop{\mathrm{Bar}}\nolimits}
\newcommand{\modu}{\mathop{\mathrm{mod}}\nolimits}
\newcommand{\Modu}{\mathop{\mathrm{Mod}}\nolimits}

\newcommand{\Map}{\mathop{\mathrm{Map}}\nolimits}
\newcommand{\id}{\mathop{\mathrm{id}}\nolimits}
\newcommand{\inj}{\mathop{\mathrm{inj}}\nolimits}
\newcommand{\op}{\mathop{\mathrm{op}}\nolimits}

\newcommand{\Ker}{\mathop{\mathrm{Ker}}\nolimits}

\newcommand{\Tot}{\mathop{\mathrm{Tot}}\nolimits}

\newcommand{\sg}{\mathop{\mathrm{sg}}\nolimits}

\newcommand{\HH}{\mathop{\mathrm{HH}}\nolimits}
\newcommand{\THH}{\mathop{\mathrm{TH}}\nolimits}

\newcommand{\Perf}{\mathop{\mathrm{Perf}}\nolimits}

\input xy
\xyoption{all}

\usepackage{hyperref}

\begin{document}

\title{Singular Hochschild cohomology and algebraic string operations}
%\author{Manuel Rivera and Zhengfang Wang}
%\author[1]{Manuel Rivera \thanks{manuel,rivera@imj-prg.fr}}
%\author[2]{Zhengfang Wang\thanks{zhengfang.wang@imj-prg.fr}}
%\curraddr{CNRS, UMR 7586, Institut de Math\'{e}matiques de Jussieu-Paris Rive Gauche, Universit\'{e} Pierre et Marie Curie, Case 247, Bureau 15-25/501, 4 place Jussieu, F-75005, Paris, France}
%\email{manuel.rivera@imj-prg.fr}
%\footnote{a}
%\thanks{The first author is currently supported by the ERC via the Starting Grant ERC StG-259118 managed by the CNRS}
%\affil[1]{}
%\affil[2]{}
\author{Manuel Rivera}
\author{Zhengfang Wang}
\address{Manuel Rivera\\ Department of Mathematics \\ University of Miami, 1365 Memorial Drive, Coral Gables, FL 33146}
\address{Departamento de Matem\'aticas, Cinvestav, Av. Instituto Polit\'ecnico Nacional 2508, Col. San Pedro Zacatenco, M\'exico, D.F. CP 07360, M\'exico}

\email{manuelr@math.miami.edu}

\address{Zhengfang Wang\footnote{Corresponding author}\\ Beijing International Center for Mathematical Research (BICMR), Peking University, No. 5 yiheyuan Road Haidian District, 100871 Beijing, China}

\email{wangzhengfang@bicmr.pku.edu.cn}

\maketitle

\begin{abstract}
Given a differential graded (dg) symmetric Frobenius algebra $A$ we construct an unbounded complex $\mathcal{D}^{*}(A,A)$, called the Tate-Hochschild complex, which arises as a totalization of a double complex having Hochschild chains as negative columns and Hochschild cochains as non-negative columns. We prove that the complex $\mathcal{D}^*(A,A)$ computes the singular Hochschild cohomology of $A$. We construct a cyclic (or Calabi-Yau) $A$-infinity algebra structure, which extends the classical Hochschild cup and cap products, and an $L$-infinity algebra structure extending the classical Gerstenhaber bracket, on $\mathcal{D}^*(A,A)$. Moreover, we prove that the cohomology algebra $H^*(\mathcal{D}^*(A,A))$ is a Batalin-Vilkovisky (BV) algebra with  BV operator extending Connes' boundary operator. Finally, we show that if two dg algebras are quasi-isomorphic then their singular Hochschild cohomologies are isomorphic and we use this invariance result to relate the Tate-Hochschild complex to string topology.  

\end{abstract}

\section{Introduction}

For any differential graded (dg) associative algebra $A$ over a ring $\mathbb{K}$ such that $A$ is projective as a $\mathbb{K}$-module,  the Hochschild $i$-th cohomology group $\HH^i(A,A)$ is defined as the group of morphisms from $A$ to $s^iA$ in $\DD(A \otimes_{\mathbb{K}} A^{\op})$, the derived category of dg $A$-$A$-bimodules, where $s^iA$ is the dg $A$-$A$-bimodule defined by $(s^iA)^j=A^{i+j}$ with left and right actions induced by the multiplication of $A$. $\HH^*(A,A)$ is a graded algebra with the Yoneda product. The \textit{singular} Hochschild $i$-th cohomology group $\HH_{\sg}^i(A, A)$ is defined as the group of morphisms from $A$ to $s^iA$ in the singularity category $\DD_{\sg}(A\otimes_{\mathbb{K}} A^{\op})$. The singularity category $\DD_{\sg}(A\otimes_{\mathbb{K}}A^{\op})$ is the Verdier quotient  (of triangulated categories) of $\DD^b(A \otimes A^{\op}) $, the bounded derived category of finitely presented dg $A$-$A$-bimodules, by the full sub-category $\Perf(A\otimes A^{\op})$ of $\DD^b(A \otimes A^{\op})$ whose objects are perfect dg $A$-$A$-bimodules \cite{KoSo}. The singularity category was defined by Buchweitz \cite{Buc} and later rediscovered by Orlov \cite{Orl}. $\HH_{\sg}^*(A, A)$ is a graded algebra as well. In this article we discuss how the algebraic structure of the singular Hochschild cohomology of a dg symmetric Frobenius algebra relates to algebraic operations on Hochschild complexes which model constructions originated in string topology.
\\

Let $A$ be a dg associative algebra over a field $\mathbb{K}$ together with a symmetric Frobenius pairing $\langle\cdot,\cdot\rangle: A \otimes A \to \mathbb{K}$ of degree $k\geq 0$, meaning a symmetric pairing that induces a degree $k$ isomorphism $A \to A^{\vee}:=\text{Hom}_{\mathbb{K}}(A,\mathbb{K})$ of dg $A$-$A$-bimodules. For any such $A$ we may construct a cochain complex, called the \textit{Tate-Hochschild complex}, whose cohomology is isomorphic to the singular Hochschild cohomology of $A$. The Tate-Hochschild complex of $A$ is a totalization of an unbounded double cochain complex denoted by $\mathcal{D}^{*,*}(A,A)$ which may be constructed as follows: place the Hochschild chain complex of $A$ on the side of the columns with negative degree (i.e. $\mathcal{D}^{i,*}(A,A):= C_{-i-1,*}(A,A)$ for $i <0$), the Hochschild cochain complex of $A$ on the side of the columns with non-negative degree (i.e. $\mathcal{D}^{i,*}(A,A):= C^{i,*}(A,A)$ for $i\geq 0$), and use the Frobenius structure on $A$ to define a linear map $\gamma: \mathcal{D}^{-1,*} (A,A) \to \mathcal{D}^{0,*}(A,A)$, connecting the Hochschild chain complex and the Hochschild cochain complex, which extends the differentials defined on each side. Here the $C_{-m,p}(A,A)$ denotes the $\mathbb{K}$-module of Hochschild chains with monomial length $m$ and total (homological) degree $p$, similarly for $C^{m,p}(A,A)$ for Hochschild cochains. More precisely, the connecting map $\gamma:
\calD^{-1, *}(A, A)\rightarrow \mathcal{D}^{0, *}(A, A)$ is the degree $-k$ map  given by the composition 
\begin{equation}\label{equation-cardy}
\gamma: A\xrightarrow{\Delta} A\otimes A \xrightarrow{T} A\otimes A \xrightarrow{\mu} A,
\end{equation}
where $\mu: A \otimes A \to A$ is the product of the algebra $A$, $\Delta$ is the coproduct of $A$ associated to the Frobenius pairing, namely
$$\Delta: A \xrightarrow{\cong} A^{\vee} \xrightarrow{ \mu^{\vee}} A^{\vee}\otimes A^{\vee} \xrightarrow{\cong} A\otimes A,$$
and $T$ is the braiding isomorphism
$$T(x\otimes y):=(-1)^{|x||y|}y\otimes x.$$
In order to get a total differential on $\mathcal{D}^{*}(A,A):= \text{Tot}(\calD^{*,*}(A,A))$ of degree $+1$ we shift each of the negative columns by $1-k$, where the totalization here means the direct sum totalization in the Hochschild chains direction and direct product totalization in the Hochschild cochains direction. A short calculation yields that $\calD^{*,*} (A,A)$ is in fact a double complex. 
\\

We describe explicitly a chain level lift of the graded associative algebra structure of $\HH_{\sg}^*(A, A)$ to an $A_{\infty}$-algebra structure on $\mathcal{D}^*(A,A)$. The $m_2$ operation of this $A_{\infty}$-algebra is a product $\star: \calD^*(A, A)\otimes \calD^*(A, A)\rightarrow \calD^*(A, A)$ which extends the Hochschild cup product, defined for a pair of Hochschild cochains, and the Hochschild cap product, defined for a Hochschild cochain and a Hochschild chain satisfying certain degree constraints. We construct the $A_{\infty}$-algebra structure by first lifting the associative algebra structure of $\HH_{\sg}^*(A, A)$  to a dg associative algebra on a chain complex $\calC^*_{\sg}(A,A)$, chain homotopy equivalent to $\mathcal{D}^{*}(A,A)$ through a homotopy retraction, and then transferring such structure to a quasi-isomorphic $A_{\infty}$-algebra structure on $\mathcal{D}^{*}(A,A)$. Furthermore, we show the resulting $A_{\infty}$-algebra structure is cyclically compatible with a pairing on $\mathcal{D}^{*,*}(A,A)$ induced by the Frobenius pairing of $A$. In other words, we construct a cyclic (or Calabi-Yau) $A_{\infty}$-algebra structure on the totalization of the double complex $\mathcal{D}^{*,*}(A,A)= s^{1-k}C_{*,*}(A,A) \oplus C^{*,*}(A,A)$ which lifts the graded associative algebra structure of $\HH_{\sg}^*(A, A)$, where $s^{1-k}$ denotes the degree shift by $1-k$. We also describe an $L_{\infty}$-algebra structure on $\mathcal{D}^{*,*}(A,A)$ extending the dg Lie algebra structure on Hochschild cochains with the Gerstenhaber bracket. Moreover, we prove that the product $\star$ on $\HH_{\sg}^*(A, A) \cong H^*(\calD^*(A,A))$ is part of a BV-algebra structure. The BV operator may be defined on $\calD^*(A, A)$ and it extends Connes' boundary operator $B$ on $C_{*, *}(A, A)$. We pose the following Deligne type conjecture: this BV-algebra can be lifted to an action of a dg operad, weakly equivalent to chains on the framed little disks operad, on $\mathcal{D}^{*}(A,A)$. In particular, the construction of such an action would extend the solution of the cyclic Deligne conjecture described by Kaufmann \cite{Kau} and independently by Tradler-Zeinalian \cite{TrZe}.  \\

It follows from a well known theorem of Jones that for any simply-connected space $X$ there is an isomorphism $\HH_*(C^*(X;\mathbb{Q}), C^*(X;\mathbb{Q})) \cong H^*(LX;\mathbb{Q})$ where $LX$ is the free loop space on $X$. Moreover, for a $k$-dimensional simply-connected closed manifold $M$ we have an isomorphism of graded algebras $\HH^*(C^*(M;\mathbb{Q}), C^*(M;\mathbb{Q})) \cong s^kH_*(LM;\mathbb{Q})$, where the left hand side is the Hochschild cohomology of the rational singular cochains on $M$ with cup product and the right hand side is the singular homology of the free loop space of $M$ with rational coefficients (shifted by $k$) with the Chas-Sullivan loop product \cite{FeThVi}. The Chas-Sullivan loop product is an intersection type operation which combines the intersection product of the manifold $M$ and concatenation of loops. The loop product led to the construction of more general \textit{string topology operations}  on chains of loops in a manifold by using the Poincar\'e duality of the underlying manifold and concatenating or deconcatenating loops according to different compatible patterns. String topology operations have counterparts in the Hochschild chain complex of a dg symmetric Frobenius algebra \cite{TrZe, WaWe}. 
\\

Since Hochschild homology is invariant under quasi-isomorphisms,  it follows there is an isomorphism of graded vector spaces $\HH_*(A,A) \cong H^*(LM)$ for any differential graded algebra $A$ quasi-isomorphic to the rational singular cochains $C^*(M;\mathbb{Q})$. The singular cohomology with rational coefficients $H^*(M;\mathbb{Q})$ is a graded symmetric Frobenius  algebra equipped with the Poincar\'e duality pairing and the cup product, however, this pairing does not give a Frobenius structure at the level of singular cochains $C^*(M;\mathbb{Q})$. By the main result in \cite{LaSt}, for any simply-connected manifold $M$, one may construct a commutative differential graded symmetric Frobenius algebra which is quasi-isomorphic, as a dg associative algebra, to $C^*(M;\mathbb{Q})$ with the property that the induced isomorphism $H^*(A) \cong H^*(C^*(M;\mathbb{Q}))$ preserves the Frobenius algebra structures. Using a Frobenius structure on a differential graded algebra $A$ one may construct operations on the Hochschild chain complex $C_{*,*}(A,A)$ analogue to string topology operations \cite{WaWe, TrZe, Abb}. One of the main points we would like to stress in this paper is the following surprising observation: when the product $\star: \calD^*(A, A)\otimes \calD^*(A, A)\rightarrow \calD^*(A, A)$ is applied to a pair of Hochschild chains we obtain an operation which was previously described in \cite{Abb} and \cite{Kla},  in the case when $A$ is commutative, and is believed to be intimately related to the Goresky-Hingston product on $H^*(LM,M)$, a string topology operation of degree $k-1$ constructed on the cohomology of the free loop space of a $k$-dimensional manifold relative to constant loops. 
\\

More precisely, for any symmetric Frobenius differential graded algebra $A$ with pairing of degree $k$, Abbaspour describes in \cite{Abb} an associative product $*: C_{*,*}(A,A) ^{\otimes 2}  \to C_{*,*}(A,A)$ of degree $k-1$ which arises as a chain homotopy between two chain maps $*_0, *_1: C_{*,*}(A,A)^{\otimes 2} \to C_{*,*}(A,A)$ each of degree $k$. The operation $*$  is sometimes called a \textit{secondary} operation since it is a chain homotopy between two \textit{primary} operations; note, in particular, that $*$ is not a chain map. The construction of $*$ resembles a chain level version of a geometric string topology product on $C^*(LM;\mathbb{Q})$ (described in \cite{GoHi} at the level of relative cohomology $H^*(LM,M;\mathbb{Q})$) which is dual to an operation which, at the level of chains, considers self intersections in a chain of loops and splits loops at these intersection points. There are several ways of obtaining a chain map from $*$. In the case $A$ is commutative Abbaspour suggests taking a subcomplex of $C^{*,*}(A,A)$ and Klamt suggests modifying the product $*$ in order to obtain a chain map.  We do not assume commutativity for $A$ and propose extending the chain complex $C_{*,*}(A,A)$ by extending the differential through the map $\gamma$ defined in (\ref{equation-cardy}).
% induced by the composition
%$$
%A \xrightarrow{ \gamma } A^{\vee} \xrightarrow{\mu^{\vee}} (A \otimes A)^{\vee} \cong A^{\vee} \otimes A^{\vee}  \xrightarrow{\gamma^{-1} \otimes \gamma^{-1}} A \otimes A \xrightarrow{\mu} A,
%$$
%where $\gamma$ is the isomorphism induced by the Frobenius pairing $< , >: A \otimes A \to \mathbb{K}$ and $\mu$ the product of $A$, and then extending $*$ to a product on $C_{*,*}(A,A) \oplus A[-1]$. 
We may keep extending the resulting complex by the Hochschild cochains differential to obtain an unbounded double complex and we may extend the product $*$ as well by Hochschild cup and cap products. The resulting complex, after appropriate degree shifts, is precisely $\calD^{*,*}(A,A)$ and the new product is $\star$. In conclusion, one may interpret $\star: \calD^{*,*}(A,A) \otimes \calD^{*,*}(A,A) \to \calD^{*,*}(A,A) $ as combination of the cup product in Hochschild cochains, which is an algebraic analogue of the Chas-Sullivan loop product, and the $*$-product in Hochschild chains, an algebraic analogue of the Goresky-Hingston product. We finish this article by describing the singular Hochschild cohomology of the dga of singular cochains on a simply-connected manifold $M$ in terms of the homology and cohomology of $LM$. We do this by first proving an invariance result for singular Hochschild cohomology under dga quasi-isomorphisms and then using a dg symmetric Frobenius algebra model for $M$ as provided by the main result of \cite{LaSt}. We use these results to calculate explicitly the singular Hochschild cohomology and its BV-algebra structure for the differential graded algebra $C^*(S^n, \mathbb{Q})$ of singular cochains on spheres $S^n$ for $n>1$. 
\\

There is a close relationship between the singular Hochschild cohomology of the dga of singular cochains on a closed simply-connected manifold $M$ and the \textit{Rabinowitz-Floer homology} of the unit cotangent bundle of $M$ with its canonical symplectic structure as introduced in \cite{CiFrOa}. The relationship comes from the isomorphism between the symplectic (co)homology of the unit cotangent bundle of $M$ and the singular (co)homology of the free loop space of $M$. In fact, one should compare Theorem 1.10 in \cite{CiFrOa} with Theorem 7.1 of this article. Moreover, in \cite{CiOa} an algebra structure on the Rabinowitz-Floer side is constructed and we conjecture that such structure agrees with the algebra structure of the singular Hochschild cohomology of the dga of singular cochains of $M$. Based on the constructions of \cite{CiOa}, it is also evident that the singular Hochschild cohomology is closely related to the symplectic homology of the boundary of a Liouville domain. These topics will be explored in future research.

\section*{Acknowledgment}
The first author acknowledges support by the ERC via the grant StG-259118-STEIN and the excellent working conditions at \textit{Institut de Math\'ematiques de Jussieu-Paris Rive Gauche} (IMJ-PRG) where the first half of this project was completed and the support of University of Miami and Fordecyt grant 265667 during the second half of the project. We would like to thank Hossein Abbaspour, Ragnar Buchweitz, Tobias Dyckerhoff, Yizhi Huang, Dmitry Kaledin, Liang Kong, Alexandru Oancea, Boris Tsygan, Bruno Vallette, Guodong Zhou and Alexander Zimmermann for stimulating discussions and comments. We would also like to thank Natalia Rodriguez for helping out with the diagrams and figures in this article.

\section{Preliminaries}
In this section we recall some basic notions regarding differential graded (dg) algebras. For more details we refer the reader to \cite{Abb, Kel1, Kel2}. We fix a commutative ring $\mathbb{K}$ throughout this section. 
\subsection{Differential graded algebras}

\begin{definition}
Let $\mathbb{K}$ be a commutative ring. A {\it differential graded (dg) algebra over $\mathbb{K}$} is a cochain complex $(A^{\bullet}, d^{\bullet})$
of $\mathbb{K}$-modules endowed with $\mathbb{K}$-linear maps $\mu: A^n \otimes_{\mathbb{K}} A^m \rightarrow A^{n+m}, a\otimes_{\mathbb K} b \mapsto ab:=\mu(a\otimes_{\mathbb K}b)$ such 
that 
$d^{n+m}(ab)=d^n(a)b+(-1)^n a d^m(b)$ 
and such that $\bigoplus_{n\in \Z} A^n$ becomes an associative  $\mathbb{K}$-algebra with unit. Moreover, we say that a dg algebra $A$ is {\it commutative} if $$xy=(-1)^{\deg(x)\deg(y)}yx$$ for 
any homogeneous elements $x, y \in A$.
\end{definition}

\begin{definition}
Let $(A, d)$ be a dg  $\mathbb{K}$-algebra. The {\it opposite dg algebra} is the dg  $\mathbb{K}$-algebra  $(A^{\op}, d^{\op})$   where $A^{\op}:=A$ as an $\mathbb{K}$-module, $d^{\op}=d$, and multiplication is given by $$a\cdot_{\op} b=(-1)^{\deg(a)\deg(b)} ba$$ for homogeneous elements $a, b\in A$.
\end{definition}

\begin{definition}
 Let $(A, d)$ and $(B, d)$ be two dg  $\mathbb{K}$-algebras. The {\it tensor product differential graded algebra} of $A$ and $B$ is the algebra $A\otimes B$ with multiplication defined by $$(a\otimes b)(a'\otimes b'):=(-1)^{\deg(a')\deg(b)}aa'\otimes bb'$$ endowed with differential $d_{A\otimes_{\mathbb{K}}B}$ defined by the rule $$ d_{A\otimes B}(a\otimes b)=d(a)\otimes b+(-1)^{\deg(a)} a\otimes d(b).$$
\end{definition}

\begin{definition}
 Let $(A, d)$ be a dg $\mathbb{K}$-algebra. A  {\it (left) differential graded module} $M$ over $A$ is a left $A$-module $M$ which has a grading $M=\bigoplus_{n\in \Z}M^n$ and a differential $d$ such that $A^mM^n\subset M^{m+n}$,  such that $d(M^n)\subset M^{n+1}$, and such that $$d(am)=d(a)m +(-1)^ {\deg(a)} ad(m)$$for any homogenous element $a\in A$.
\end{definition}

\begin{definition}
Let $(A, d)$ be a dg $\mathbb{K}$-algebra. Let $(M, d)$ and $(N,d)$ be two differential graded $A$-modules. We say that  $f: (M, d)\rightarrow (N, d)$ is a $(A, d)$-module morphism of degree $n\in \Z$ if the following two conditions are satisfied, 
\begin{enumerate}
\item $f$ is a $\mathbb{K}$-linear map such that  $f(M^k)\subset N^{k+n}$ for any $k\in \Z$,
\item $f\circ d=(-1)^n d\circ f$ and $f(ax)=(-1)^{\deg(a) n} a f(x)$, where $a\in A$ is any homogenous element and $x\in M$. 
\end{enumerate}
\end{definition}

\begin{remark}
We denote by $(A, d)$-$\Modu$ (or also by $A$-$\Modu$, for short) the abelian category whose objects are left differential graded $A$-modules and morphisms are the $(A, d)$-module morphisms of degree zero. In particular, when $A$ is an (ordinary) associative algebra, the category $(A, d)$-$\Modu$ is equivalent to the category of cochain complexes of $A$-modules. In the same manner, we may define the abelian category of {\it right differential graded $(A, d)$-modules}. If $(M, d)$ is a left $(A, d)$-module, then we may think of $(M, d)$ as a right $(A^{\op}, d^{\op})$-module with the action $$m\cdot_{\op} a=(-1)^{\deg(a)\deg(m)} am$$ for any homogenous elements $a\in A$ and $m\in M$.  
For simplicity, a morphism of $(A, d)$-modules of degree zero is called a morphism of $(A, d)$-modules. 
 
For any differential graded $\mathbb{K}$-module $(M, d_M)$, define a new differential graded $\mathbb{K}$-module $(sM, d_{sM})$, as a graded $\mathbb{K}$-module $(sM)^n:=M^{n+1}$ and the differential $d_{sM}:=-d_M.$ Thus we have a \textit{shift functor} $s: \mbox{$(A, d)$-$\Modu$}  \rightarrow \mbox{$(A, d)$-$\Modu$}$  sending a dg $A$-module  $M$ to   $sM$.  It is clear that the shift functor $s$ is an equivalence and, moreover, it induces the shift functor  of  the homotopy category $\KK(\mbox{$A$-$\Modu$})$ and derived category $\DD(\mbox{$A$-$\Modu$})$  as triangulated categories (cf. \cite{Kel1}).
\end{remark}

 Given two differential graded $\mathbb{K}$-modules $(M, d)$ and $(N, d)$, we denote the {\it braiding morphism} $T: M\otimes_{\mathbb{K}}N \rightarrow N\otimes_{\mathbb{K}}M$ by $$T(x\otimes_{\mathbb{K}}y):=(-1)^{\deg(x)\deg(y)} y \otimes_{\mathbb{K}} x,$$ for any homogeneous elements $x\in M$ and $y\in N$. 

\begin{definition}\label{definition-dg-symmetric}
Let $\mathbb{K}$ be a field. A {\it differential graded (dg)  Frobenius $\mathbb{K}$-algebra} of degree $k \geq 0$ is the data $(A, d, \mu, \Delta)$ where
\begin{enumerate}
\item $(A, d, \mu)$ is a dg $\mathbb{K}$-algebra such that $A$ is non-negatively graded (i.e., $A^{<0}=0$);
\item $(A, \Delta)$ is a differential graded coassociative  coalgebra of degree $k$. That is, $\Delta: A \to A \otimes_{\mathbb{K}} A$ is a linear map of degree $k$ such that
\begin{enumerate}
\item $(\Delta\otimes \id)\Delta= (\id\otimes \Delta)\Delta$, 
\item  $\Delta d=(d\otimes \id+\id\otimes d) \Delta$,
\end{enumerate}
where the identities respect the Koszul sign rule;
\item $\Delta: A\rightarrow  A\otimes_{\mathbb{K}}A$ is a left and right differential graded $A$-module map (i.e. an $A$-$A$-bimodule map).
\end{enumerate}
We say $(A, d, \mu, \Delta)$ is \textit{counital} if it is equipped with a map $\epsilon:A \to \mathbb{K}$ of degree $-k$, which is a counit for $\Delta:A \to A \otimes A$, i.e. $(\id\otimes \epsilon) \Delta=\id = (\epsilon\otimes \id) \Delta$.  We say $(A, d, \mu, \Delta)$ is \textit{commutative} if $(A,d,\mu)$ is a commutative dg algebra and {\it cocommutative} if $T \Delta=\Delta$.  Let $(A, d, \mu, \Delta)$ be a dg Frobenius $\mathbb{K}$-algebra of degree $k$. We say $(A, d, \mu, \Delta)$ is a {\it dg symmetric Frobenius $\mathbb{K}$-algebra} if it is counital and if $T\Delta(1)=\Delta(1)$. Note that a dg symmetric Frobenius $\mathbb{K}$-algebra is not necessarily commutative or cocommutative.
 \end{definition}

\begin{remark}\label{remark-dg-symmetric}
We give some facts and remarks on a dg symmetric Frobenius algebra $(A, d, \mu, \Delta, \epsilon)$ of degree $k$ (cf. \cite{Abb}). 
\begin{enumerate}
\item We will use generalized Sweedler's notation for the coproduct $$\Delta(x)=\sum_{(x)}x' \otimes x'' = \sum x'\otimes x''.$$ Using this notation, a counit may be defined as a linear map $\epsilon: A \to \mathbb{K}$ satisfying $$x=\sum (-1)^{k\deg(x')} x'\epsilon(x'')=\sum \epsilon(x')x''.$$ We will also write $\mu(xy)=xy$ for the product. 
\item If we write $\Delta(1)=\sum_{i} e_i\otimes f_i$, then $\deg(e_i)+\deg(f_i)=k$.  Since $\Delta$ is a $A$-$A$-bimodule homomorphism, then $$\Delta(x):=\sum x'\otimes x''=\sum_i e_i\otimes f_i x=\sum_i (-1)^{k\deg(x)} xe_i\otimes f_i.$$
\item A counit $\epsilon: A \to \mathbb{K}$ defines a symmetric, invariant, non-degenerate inner product $\langle\cdot, \cdot\rangle: A\times A\rightarrow k$ by the formula $\langle x, y \rangle:= \epsilon(xy).$ In particular, $A$ is finite dimensional as a $\mathbb{K}$-vector space.
%Indeed, it is clear that $\langle xy, z\rangle=\langle x, yz\rangle$. Note that 
%we have $$xy=\sum_i (-1)^{\deg(e_i)m}e_i\epsilon(f_ix) y,$$
%hence $\langle x, y\rangle=\sum_i (-1)^{\deg(e_i)m} \epsilon(e_iy)\epsilon(f_ix)$. Since
%$\Delta$ is cocomunicative, we have 
%$$\sum_i e_i\otimes f_i=\sum_i (-1)^{m+\deg(e_i)\deg(f_i)}f_i\otimes e_i,$$
%thus, 
%we have 
%$$\langle x, y \rangle=(-1)^{\deg(e_i)m+m+\deg}
\item A commutative dg symmetric Frobenius algebra is cocommutative.
\end{enumerate}
\end{remark}
\begin{lemma}
Let $A$ be a  dg symmetric Frobenius $\mathbb{K}$-algebra of degree $k\geq 0$. Then $A^{>k}=0$.
\end{lemma}
\begin{proof}
By Remark \ref{remark-dg-symmetric} (3) above, there is an $A$-$A$-bimodule isomorphism of degree $-k$ between $A$ and $\Hom_{\mathbb{K}}(A, \mathbb{K})$, where the degree of $\Hom_{\mathbb{K}}(A^i, \mathbb{K})$ is $-i$. Since $A$ is finite dimensional,  we obtain that  $A^{>k}=0$ and $A^i\cong \Hom_{\mathbb{K}}(A^{k-i}, \mathbb{K})$ for $i=0, \cdots, k$.
\end{proof}

%\begin{example}
%Let 

%\end{example}

\subsection{Hochschild homology and cohomology}
Let $\mathbb{K}$ be a commutative ring. Let $(A, d)$ be a differential graded $\mathbb{K}$-algebra. Denote $\A:=A/( \mathbb{K}\cdot 1)$. For any $m\in \Z_{\geq 0}$, we define differential graded  $A$-$A$-bimodules $$\Barr_{-m}(A):= A\otimes (\sA)^{\otimes m} \otimes A.$$  For simplicity, we will drop the $s$ indicating the shift writing an element $a_0 \otimes s\overline{a_1} \otimes \cdots  \otimes s\overline{a_m} \otimes a_{m+1} \in A\otimes (\sA)^{\otimes m} \otimes A$ as $a_0 \otimes \overline{a_1}\otimes \cdots \otimes \overline{a_m} \otimes a_{m+1}$. Thus $\text{deg}(a_0 \otimes \overline{a_1} \otimes \cdots\otimes  \overline{a_m} \otimes a_{m+1})=\sum_{i=0}^{m+1} \text{deg}(\overline{a_i}) - m$. Hence, an element $\overline{a}$ should always be interpreted as $s\overline{a}$, namely, as the class in $\overline{A}$ represented by $a \in A$ with a degree shift indicated by $s$. 
\\
\\
For each $m > 0$ we have a morphism (of degree one) of dg $A$-$A$-bimodules
\begin{equation}\label{defn-external-dif}
b_{-m}: \Barr_{-m}(A)\rightarrow \Barr_{-m+1}(A)
\end{equation}
defined by sending $a_0\otimes \overline{a_1} \otimes \cdots \otimes a_{m+1}\in \Barr_{-m}(A)$ to 
\begin{equation*}
\begin{split}
&a_0a_1\otimes \overline{a_2}\otimes \cdots \otimes \overline{a_m}\otimes a_{m+1}+\\&\sum_{i=1}^{m-1} (-1)^{\epsilon_i}a_0 \otimes \overline{a_1}\otimes \cdots \otimes \overline{a_ia_{i+1}}\otimes \overline{a_{i+2}}\otimes \cdots \otimes \overline{a_m} \otimes a_{ m+1}+\\
&(-1)^{\epsilon_m}a_0\otimes \overline{a_1}\otimes \cdots \otimes \overline{a_{m-1}}\otimes a_ma_{m+1},
\end{split}
\end{equation*} 
where $\epsilon_i:=\sum_{j=0}^{i-1} \deg(a_j)-i$; and we define $b_0:\Barr_0(A)= A \otimes A \rightarrow A$ to be the multiplication $\mu$ of $A$. As pointed above, we have dropped the shift functor $s$ and we have written $a_0 \otimes \overline{a_1}\otimes \cdots \otimes \overline{a_ia_{i+1}}\otimes \overline{a_{i+2}}\otimes \cdots \otimes \overline{a_m} \otimes a_{m+1}$ for the element $a_0 \otimes s\overline{a_1} \otimes \cdots  \otimes s\overline{a_{i-1}} \otimes s(\overline{a_ia_{i+1}}) \otimes s\overline{a_{i+2} }\otimes \cdots  \otimes s\overline{a_m} \otimes a_{m+1}$. We will use this convention throughout the rest of the paper. 
\\
\\
It is straightforward to verify that we obtain a well defined map $b=b_*: \Barr_*(A) \to \Barr_{*+1}(A)$  satisfying $b\circ b=0$.  Note that $s^{-m}b_{-m}: s^{-m} \Barr_{-m}(A)\rightarrow s^{-m+1}\Barr_{-m+1}(A)$ is a degree zero morphisms of dg $A$-$A$-bimodules, hence we obtain a complex of  dg $A$-$A$-bimodules:
\begin{equation}\label{equa-bar-complex}
\xymatrix{
\Barr_*(A): \cdots\ar[r] & s^{-m}\Barr_{-m}(A) \ar[r]^-{s^{-m}b_{-m}} & s^{-m+1}\Barr_{-m+1}(A) \ar[r]& \cdots \ar[r]^-{b_0}& A\ar[r] & 0
}
\end{equation}

\begin{lemma}\label{lemma-bar-exact}
The complex $\Barr_*(A)$ is exact in the category of dg $A$-$A$-bimodules.
\end{lemma}
\begin{proof}
 Define $s_m: s^{-m}\Barr _{-m}(A)\rightarrow s^{-m-1}\Barr _{-m-1}(A)$ to be the map which sends $$a_0\otimes \overline{a_1}\otimes \cdots \otimes \overline{a_m}  \otimes a_{m+1}\in \Barr_{-m}(A)$$ to $$a_0\otimes \overline{a_1}\otimes \cdots \otimes \overline{a_{m+1}}\otimes 1\in \Barr_{-m-1}(A).$$ We then have $s\circ b+b\circ s=id_{\Barr_*(A)}$, as desired. 
\end{proof}
%\begin{remark}
%We denote by $\Omega^p(A)$ the kernel 
%$\Ker (b_{p-1})$ for any $p\in \Z_{\geq 0}$. In particular, 
%$\Omega^0(A):=A$.
% Then from Lemma \ref{lemma-bar-exact}, we have a short exact sequence
%of dg $A$-$A$-bimodules for any $p\in \Z_{\geq 0}$,
%$$
%\xymatrix{
%0\ar[r]  & \Omega^{p+1}(A) \ar[r]  & \Barr_p(A) \ar[r] & \Omega^p(A)\ar[r]  & 0.
%}
%$$
%\end{remark}

%Let $V$ be a differential graded $\mathbb{K}$-module. Then
%the tensor $\mathbb{K}$-module $T(V)$ of $V$ is the differential graded $\mathbb{K}$-module whose 
%term of degree $m$ is $$T(V)^m:=\{ v_1\otimes \cdots \otimes v_p\in V^{\otimes p} \ | \  \sum_{i=1}^p \deg(v_i)=m \}$$
%and the differential is given by the following rule:
%$$d(v_1\otimes \cdots \otimes v_p):=\sum_{\epsilon_i}^p (-1)^{i-1}v_1\otimes\cdots \otimes v_{i-1}\otimes dv_i\otimes v_{i+1}\otimes \cdots v_p,$$
%where $\epsilon_i:=\sum_{j=1}^{i-1} \deg(v_j).$
\begin{remark}\label{rem-bar-construction}
Recall that the {\it two-sided bar construction} of a differential graded $\mathbb{K}$-algebra $A$ is given by the differential graded $A$-$A$-bimodule $B(A, A, A):=A\otimes T(\sA)\otimes A$ with the differential $d=d_0+d_1$, where $d_0$ is the internal differential defined by the following rule:
\begin{equation*}
\begin{split}
d_0(a_0\otimes \overline{a_1}\otimes \cdots \otimes \overline{a_m} \otimes a_{m+1})=&d(a_0)\otimes \overline{a_1}\otimes\cdots\otimes \overline{a_m}\otimes  a_{m+1}-\\
&\sum_{i=1}^m(-1)^{\epsilon_i} a_0\otimes \overline{a_1}\otimes \overline{d(a_i)} \otimes \overline{a_{i+1}}\otimes \cdots \otimes a_{m+1}+\\
& (-1)^{\epsilon_{m+1}} a_0\otimes \overline{a_{1}}\otimes\cdots \otimes d(a_{m+1})
\end{split}
\end{equation*}
 and the differential $d_1$ is the external differential given by $b$ defined in (\ref{defn-external-dif}).
%\begin{equation*}
%\begin{split}
%d_1(a_0\otimes a_{1, m}\otimes a_{m+1}):=&(-1)^{\deg(a_0)} a_0a_1\otimes a_{2, m}\otimes a_{m+1}+\\
%&\sum_{i=1}^{m-1}(-1)^{\epsilon_{i}} a_0 \otimes a_{1, i-1}\otimes a_ia_{i+1}\otimes a_{i+2, m}\otimes a_{m+1}-\\
%&(-1)^{\epsilon_{m}} a_0\otimes a_{1, m-1} \otimes a_m a_{m+1}
%\end{split}
%\end{equation*}
%where $\epsilon_i:=\sum_{j=0}^{i-1} \deg(a_j)+i-1.$ 
In fact, the two-sided bar construction $B(A, A, A)$ above is initially constructed from the following double complex $\Barr_{*, *}(A)$ whose $(m, p)$-term is $$\Barr_{-m, p}(A): =(A\otimes (\sA)^{\otimes m} \otimes A)^p.$$ That is, $a_0\otimes \overline{a_{1}}\otimes \cdots\otimes a_{m+1}\in \Barr_{-m, p}(A)$ if  $$\sum_{i=0}^{m+1} \deg(a_i)-m=p.$$ The vertical differential $d_v: \Barr_{*, p}(A)\rightarrow \Barr_{*, p+1}(A)$ is exactly $d_0$ defined above and the horizontal differential $d_h: \Barr_{-m, *}(\sA)\rightarrow \Barr_{-m+1, *}(\sA)$ is defined to be $d_1$.  From Lemma \ref{lemma-bar-exact}, it follows that the canonical morphism $\pi: B(A, A, A)\rightarrow A$ is a quasi-isomorphism of  dg $A$-$A$-bimodules, hence $B(A, A, A)$ is a cofibrant resolution of $A$ (cf. \cite[Section 3.2]{Kel2}). In particular, if $A$ is an (ordinary) associative $\mathbb{K}$-algebra such that $A$ is projective as a $\mathbb{K}$-module, then $B(A, A, A)$ is a projective resolution of  $A$ as an $A$-$A$-bimodule.
\end{remark}

For any $n\in \Z$, the {\it Hochschild cohomology} $\HH^n(A, A)$ of degree $n$ of a differential graded algebra $A$ is defined as the Hom-space $\Hom_{\DD(A\otimes A^{\op})}(A, s^nA)$ in the derived category $\DD(A\otimes A^{\op})$. Since $B(A, A, A)$ is a cofibrant resolution of $A$ as an $A$-$A$-bimodule, we have  that 
\begin{equation*}
\begin{split}
\Hom_{\DD(A\otimes A^{\op})}(A, s^nA)\cong & \Hom_{\KK(A\otimes A^{\op})}(B(A, A, A), s^nA)\\
\cong&H_n( R\Hom_{\mathbb{K}}(T(\sA), A)).
\end{split}
\end{equation*}
Note that $R\Hom_{\mathbb{K}}(T(\sA), A)$ is the (product) total complex of the following double complex $C^{*, *}(A, A)$  located in the right half plane,
\begin{equation}
\xymatrix{
& \vdots  & \vdots & \vdots &\\
0\ar[r] & A^1 \ar[r] \ar[u] & \Hom_{\mathbb{K}}(\sA, A)^2 \ar[r] \ar[u]   & \Hom_{\mathbb{K}}((\sA)^{\otimes 2}, A)^3\ar[r]\ar[u] &\cdots\\
0\ar[r]  & A^0\ar[u] \ar[r]  & \Hom_{\mathbb{K}}(\sA, A)^1\ar[u] \ar[r] & \Hom_{\mathbb{K}}((\sA)^{\otimes 2}, A)^2\ar[u]\ar[r] &\cdots\\
& 0\ar[r]\ar[u] & \Hom_{\mathbb{K}}(\sA, A)^0\ar[u]\ar[r]  & \Hom_{\mathbb{K}}((\sA)^{\otimes 2}, A)^1\ar[u]\ar[r]& \cdots\\
&&  \vdots \ar[u] & \vdots\ar[u]
}
\end{equation}
where $C^{m, p}(A, A):=\Hom_{\mathbb{K}}((\sA)^{\otimes m}, A)^p$ is the set of morphisms $f: (\sA)^{\otimes m} \rightarrow A$ of degree $p$ and the differentials are induced from the ones of $\Barr_{*, *}(\sA)$ (cf. Remark \ref{rem-bar-construction}). More precisely, the vertical differential $\delta^v: \Hom_{\mathbb{K}}((\sA)^{\otimes m}, A)^p \rightarrow \Hom_{\mathbb{K}}((\sA)^{\otimes m}, A)^{p+1}$ is given by 
\begin{equation*}
\begin{split}
\delta^v(f)(\overline{a_{1}}\otimes \cdots\otimes \overline{a_m}):=&df(\overline{a_1}\otimes \cdots \otimes \overline{a_ m}) +\\
&\sum_{i=1}^m (-1)^{\epsilon_i} f(\overline{a_{1}}\otimes \cdots \otimes \overline{da_i} \otimes \overline{a_{i+1}}\otimes \cdots \otimes \overline{a_m}), 
\end{split}
\end{equation*}
where $\epsilon_i:= p+i-1+\sum_{j=1}^{i-1}\deg(a_j)$ and the horizontal differential 
$\delta^h: \Hom_{\mathbb{K}}((\sA)^{\otimes m}, A)^p \rightarrow \Hom_{\mathbb{K}}((\sA)^{\otimes m+1}, A)^{p+1}$ is given by 
\begin{equation*}
\begin{split}
\delta^h(f)(\overline{a_{1}} \otimes \cdots \otimes \overline{a_{m+1}}):=& (-1)^{\deg(a_1)p} a_1 f(\overline{a_{2}}\otimes \cdots \otimes\overline{a_{ m+1}})+ \\
& \sum_{i=1}^m (-1)^{\epsilon_i}f(\overline{a_{1}}\otimes \cdots \otimes \overline{a_{ i-1}}\otimes\overline{ a_ia_{i+1} }\otimes \overline{a_{i}}\otimes \cdots \otimes \overline{a_{ m+1}})-\\
&(-1)^{\epsilon_{m+1}} f(\overline{a_{1}}\otimes \cdots \otimes \overline{a_ m})a_{m+1}.
\end{split}
\end{equation*}
%We have a double complex $C^{*,*}(A,A)$ in the right half plane given by $C^{p,q}(A,A):= \Hom( \sA^{\otimes p}, A)^{p+q}$ together with horizontal differential $\delta^h: C^{p,q}(A,A) \to C^{p+1,q}(A,A)$ and vertical differential $\delta^v: C^{p,q}(A,A) \to C^{p,q+1}(A,A)$ as defined above. 
Define $(C^*(A,A), \delta)= \Tot^{\prod}(C^{*,*}(A,A))$, so $C^n(A,A)= \prod_{p\in\Z_{\geq 0}} C^{p,n}(A,A)$. 
We call $(C^*(A, A), \delta)$ the Hochschild cochain complex of $A$. Clearly, $H^n(C^*(A, A)) \cong \HH^n(A, A)$ for any $n\in \Z$.
\\
\\
For any $n\in \Z$, the {\it Hochschild homology} $\HH_n(A, A)$ of degree $n$ of a differential graded algebra $A$ is defined as the $n$-th homology group of the derived tensor product $A\otimes^{\LL}_{A\otimes A^{\op}} A$. Using the fact that $B(A, A, A)$ is a cofibrant resolution, we obtain that 
\begin{equation*}
\begin{split}
\HH_n(A, A)\cong &H_n(B(A, A, A)\otimes_{A\otimes A^{\op}} A )\\
\cong & H_n(T(\sA))\otimes_{\mathbb{K} }A)
\end{split}
\end{equation*}
Similarly, we note that $T(\sA))\otimes_{\mathbb{K} }A$ is the (direct sum) total complex of the following double complex $ C_{*, *}(A, A)$ located in the left half plane,
\begin{equation}
\xymatrix{
& \vdots & \vdots & \vdots& \vdots\\
\cdots \ar[r] & ( (\sA)^{\otimes 3}\otimes A)^2 \ar[r] \ar[u]& ( (\sA)^{\otimes 2}\otimes A)^3 \ar[r]\ar[u] & ( \sA\otimes A)^4 \ar[r] \ar[u]& A^5 \ar[r]\ar[u]& 0 \\
\cdots \ar[r] & ( (\sA)^{\otimes 3}\otimes A)^1\ar[r]  \ar[u]& ( (\sA)^{\otimes 2}\otimes A)^2\ar[r]\ar[u]  & ( \sA\otimes A)^3\ar[r]  \ar[u]&A^4\ar[u]\ar[r] & 0\\
\cdots \ar[r] & ( (\sA)^{\otimes 3}\otimes A)^0 \ar[r] \ar[u] & ( (\sA)^{\otimes 2}\otimes A)^1\ar[r]\ar[u]  & (\sA\otimes A)^2 \ar[r] \ar[u] & A^3\ar[u]\ar[r] &  0\\
& \vdots \ar[u]  & \vdots \ar[u] & \vdots \ar[u] & \vdots \ar[u]\\
}
\end{equation}
where $C_{-m, p}(A, A):=( (\sA)^{\otimes m}\otimes A)^p$ is the set of the elements of degree $p$ in $ (\sA)^{\otimes m}\otimes A.$ The vertical differential $\delta^v: ( (\sA)^{\otimes m}\otimes A)^p\rightarrow ( (\sA)^{\otimes m}\otimes A)^{p+1}$ is given by 
\begin{equation*}
\begin{split}
\delta^v(\overline{ a_1}\otimes \cdots \otimes \overline{a_m}\otimes a_{m+1}):=&\sum_{i=1}^m(-1)^{\epsilon_i} \overline{a_1}\otimes \cdots \otimes \overline{a_{ i-1}} \otimes
d(\overline{a_i})\otimes \overline{a_{i+1}}\otimes \cdots \otimes a_{m+1}\\
&+(-1)^{\epsilon_m} \overline{a_1}\otimes\cdots \otimes \overline{a_m} \otimes d(a_{m+1})
\end{split}
\end{equation*}
where $\epsilon_i:= |a_1| + ... + |a_i| - i $ and the horizontal differential $\delta^h:(A\otimes (\sA)^{\otimes m})^p \rightarrow (A\otimes (\sA)^{\otimes m-1})^{p-1}$ is given by 
\begin{equation*}
\begin{split}
\delta^h(\overline{a_{1}}\otimes \cdots \otimes \overline{a_m}\otimes a_{m+1}):=&\sum_{i=1}^{m-1} (-1)^{\epsilon_i} \overline{ a_1} \otimes \cdots \otimes \overline{a_{i}a_{i+1}}\otimes \overline{a_{i+2}} \otimes\cdots \otimes 
a_{m+1}-\\
&(-1)^{\epsilon_{m}} \overline{a_1}\otimes \cdots \otimes \overline{a_{m-1}}\otimes a_ma_{m+1}+\\
& (-1)^{(|a_2|+...+|a_{m+1}| - m +1)|a_1| } \overline{a_2}\otimes \cdots \otimes \overline{a_m}\otimes a_{m+1}a_1.
\end{split}
\end{equation*}
%We have a double complex $ C_{*,*}(A,A)$ in the left half plane given by $C_{-p,q}(A,A):= ( \sA^{\otimes p} \otimes A)^{-p+q}$ together with horizontal differential $\delta^h: C_{-p,q}(A,A) \to C_{-p+1,q}(A,A)$ and vertical differential $\delta^v: C_{-p,q}(A,A) \to C_{-p,q+1}(A,A)$ as defined above.  
Define $(C_*(A,A),\delta) := \Tot^{\bigoplus}(C_{*,*}(A,A))$, so $C_n(A,A)= \bigoplus_{p\in\Z_{\geq 0}} C_{-p, n}(A,A)$.  We call $(C_*(A, A), \delta)$ the Hochschild chain complex of $A$. Note that $C_n(A,A)$ might be non zero even for $n$ a negative integer. Also note that in our convention $\deg{\delta}=+1$ even if we call $C_n(A,A)$ a \textit{chain} complex. Clearly, $H_n(C_*(A, A)) \cong \HH_n(A, A)$ for any $n\in \Z$.

\begin{remark}
For any dg $A$-$A$-bimodule $M$, the Hochschild cohomology $\HH^*(A, M)$ and homology $\HH_*(A, M)$ can be defined in a similar way.  They are computed by the  Hochschild cochain complex $(C^*(A, M),\delta)$  and the Hochschild chain complex $(C_*(A, M), \delta)$ with coefficients in $M$ constructed similarily as above. 
\end{remark}

\section{Singular Hochschild cohomology}

In this section we define the singular Hochschild cohomology of a dg associative algebra $A$. We continue by defining the Tate-Hochschild complex $\calD^*(A,A)$ for a dg symmetric Frobenius algebra $A$ and show that $\calD^*(A,A)$ computes the singular Hochschild cohomology of $A$. Finally, for any dg associative algebra $A$ we describe its singular Hochschild complex $\calC_{\sg}^{*}(A, A)$ and show that, if $A$ is a dg symmetric Frobenius algebra, then  $\calC_{\sg}^{*}(A, A)$ and $\calD^*(A,A)$ are chain homotopy equivalent through a homotopy retraction. 

\subsection{The definition of singular Hochschild cohomology}
Throughout this subsection we fix a field $\mathbb{K}$ and a dg $\mathbb{K}$-algebra $A$ concentrated on non-negative degrees.

\begin{definition}
 The \textit{singularity category} $\DD_{\sg}(A)$ of $A$ is defined as the Verdier quotient of triangulated categories $\DD^b(\mbox{$A$-$\modu$})/\Perf(A)$, where $\DD^b(\mbox{$A$-$\modu$})$ is the bounded derived category of finitely presented left dg $A$-modules and $\Perf(A)$ is the full sub-category of $\DD^b(\mbox{$A$-$\modu$})$ whose objects are perfect dg $A$-modules (\cite{KoSo}). 
\end{definition}

\begin{definition}
Let us  denote $A$-$\Modu_{\inj}$ the full sub-category of $A$-$\Modu$ consisting of the objects whose underlying graded $A^{\dag}$-modules are injective, where $A^{\dag}$ denotes the underlying graded algebra of $(A,d)$ obtained by forgetting the differential $d$. Let $\KK(\mbox{$A$-$\Modu_{\inj}$})$ be the homotopy category of $A$-$\Modu_{\inj}$, $\KK_{\ac}(\mbox{$A$-$\Modu_{\inj}$})$ the full sub-category of  $\KK(\mbox{$A$-$\Modu_{\inj}$})$ consisting of objects $C$ which are acyclic (i.e., $H^*(C )=0$), and $\KK_{\ac}(\mbox{$A$-$\Modu_{\inj}$})^c$ the full sub-category of compact objects. 
\end{definition}

\begin{theorem}[Corollary 5.4.\cite{Kra}, Corollary 2.2.2.\cite{Bec}]\label{theorem3.3}
 $\KK_{\ac}(\mbox{$A$-$\Modu_{\inj}$})^c$  is equivalent (up to direct summands) to the singularity category  $\DD_{\sg}(A)$.  That is, we have an equivalence of triangulated categories
$$
\xymatrix{
S: \DD_{\sg}(A)\ar[r]^-{\cong} & (\KK_{\ac}(\mbox{$A$-$\Modu_{\inj}$}))^c.
}$$
where $S$ is the stabilization functor (cf. \cite[Corollary 5.4.]{Kra}).
\end{theorem}
\begin{proof}
Recall from \cite[Corollary 2.2.2.]{Bec} that we have the following recollement of triangulated categories.
\begin{equation*}
\xymatrix{
\\
\KK_{\ac}(\mbox{$A$-$\Modu_{\inj}$})  \ar[r]^-{I}&  \ar@/_2pc/[l]_-{I_{\rho}}   \ar@/^2pc/[l]^-{I_{\lambda}} \KK(\mbox{$A$-$\Modu_{\inj}$})  \ar[r]^-{Q} & \DD(A)\ar@/_2pc/[l]_-{Q_{\rho}}  \ar@/^2pc/[l]^-{Q_{\lambda}}
\\
}
\end{equation*}
It follows from \cite[Theorem 2.1.]{Nee} that $$I_{\rho}: (\KK(\mbox{$A$-$\Modu_{\inj}$}))^c/ (\DD(A))^c\rightarrow (\KK_{\ac}(\mbox{$A$-$\Modu_{\inj}$}))^c$$ is fully-faithful and moreover the idempotent completion $ ((\KK(\mbox{$A$-$\Modu_{\inj}$}))^c/ (\DD(A))^c)^{\omega} $ is equivalent to $(\KK_{\ac}(\mbox{$A$-$\Modu_{\inj}$}))^c.$ On the other hand, we note that $$(\KK(\mbox{$A$-$\Modu_{\inj}$}))^c \cong \DD^b(\mbox{$A$-$\modu$})$$ and $$ D(A)^c\cong \Perf(A).$$ 
So $I_{\rho}: \DD_{\sg}(A) \rightarrow \KK_{\ac}(\mbox{$A$-$\Modu_{\inj}$}))^c$ is an equivalence (up to direct summands) of triangulated categories.
\end{proof}

\begin{definition}
Let $A$ be a dg $\mathbb{K}$-algebra concentrated on non-negative degrees such that $H^i(A)=0$ for 
$i\gg0$. Suppose that $A$ is projective as a $\mathbb{K}$-module.
The \textit{singular Hochschild cohomology group} of degree $i$ is defined as $$\HH_{\sg}^i(A, A):=\Hom_{\DD_{\sg}(A\otimes A^{\op})}(A, s^iA)$$ for any $i\in \Z$.
\end{definition}

\begin{remark}\label{remark3.5}
From Theorem \ref{theorem3.3}, it follows that for any $i\in \Z$, $$\HH_{\sg}^i(A, A)\cong \Hom_{\KK_{\ac}\left(\mbox{$A\otimes A^{\op}$-$\Modu_{\inj}$}\right)^c}(S(A), s^iS(A)).$$ 
%We can also use the projective model to define the singular Hochschild cohomology. More precisely, the full subcategory of $\KK_{\ac}(\mbox{$A$-$\Modu_{\proj}$})$ consisting of compact objects is equivalent  (up to direct summands) to  the singularity category $\DD_{\sg}(A)$. Thus we have $$\HH_{\sg}^i(A, A)\cong \Hom_{\KK_{\ac}\left(\mbox{$A\otimes A^{\op}$-$\Modu_{\proj}$}\right)}(A, s^iA)$$ where we consider $A$ as an object of $\KK_{\ac}\left(\mbox{$A\otimes A^{\op}$-$\Modu_{\proj}$}\right)$ via the equivalence between $ \DD_{\sg}(A\otimes A^{\op})$ and $ (\KK_{\ac}(\mbox{$A\otimes A^{\op}$-$\Modu_{\inj}$}))^c.$

\end{remark}

The rest of this section is devoted to constructing two different complexes the homologies of which  are both isomorphic to the singular Hochschild cohomology $\HH_{\sg}^*(A, A)$ in the case of a dg symmetric Frobenius $\mathbb{K}$-algebra $A$ of degree $k\geq 0.$

\subsection{The Tate-Hochschild complex of a dg symmetric Frobenius algebra}
In this subsection we fix a field $\mathbb{K}$ and a dg symmetric Frobenius $\mathbb{K}$-algebra $A$ of degree $k$.
\\

Let $(C, d, \Delta, \epsilon)$ be a counital dg coalgebra of degree zero. Denote $\overline{C}: =\Ker(\epsilon)$. Define the {\it two-sided cobar construction} $\Omega(C, C, C)$ of $C$ to be $C\otimes T(\sC)\otimes C$ with the differential $d=d_0+d_1,$ where $d_0$ is the internal differential given by 
\begin{equation*}
\begin{split}
d_0(a_0\otimes \overline{a_{1}}\otimes \cdots \otimes \overline{a_m} \otimes a_{m+1})=&d(a_0)\otimes \overline{a_{1}}\otimes \cdots\otimes \overline{a_m} \otimes  a_{m+1}-\\
&\sum_{i=1}^m(-1)^{\epsilon_i} a_0\otimes \overline{a_{1}}\otimes \cdots \otimes d(\overline{a_i}) \otimes \cdots \otimes \overline{a_m} \otimes  a_{m+1}+\\
& (-1)^{\epsilon_{m+1}} a_0\otimes \overline{a_{1}} \otimes \cdots\otimes \overline{a_m} \otimes  d(a_{m+1})
\end{split}
\end{equation*}
and the differential $d_1$ is the external differential given by 
\begin{equation*}
\begin{split}
d_1(a_0\otimes \overline{a_{1}}\otimes \cdots \otimes a_{m+1}):=&\Delta(a_0)\otimes \overline{a_{1}}\otimes \cdots \otimes \overline{a_m} \otimes a_{m+1}+\\
&\sum_{i=1}^m (-1)^{\epsilon_{i-1}}a_0\otimes \overline{a_{1}}\otimes \cdots \otimes \Delta( \overline{a_i} )\otimes \cdots \otimes \overline{a_m} \otimes a_{m+1}-\\
&(-1)^{\epsilon_m}a_0\otimes \overline{a_{1}}\otimes \cdots \otimes \overline{a_m} \otimes \Delta(a_{m+1})
\end{split}
\end{equation*}a
where $\epsilon_i:= \sum_{j=0}^i \deg(a_i)+i$. Note that the range of $\Delta$ is $C \otimes C$, however, from the following lemma it follows that $d_1$ is well-defined as a map $C \otimes T(s^{-1}\overline{C}) \otimes C \to C \otimes T(s^{-1}\overline{C}) \otimes C$. 
\begin{lemma}
Let $x:=a_0\otimes \overline{a_{1}}\otimes \cdots\otimes a_{m+1}$ be an element in  $C\otimes (\sC)^{\otimes m}\otimes C$. Then we have $d_1(x)\in C\otimes (\sC)^{\otimes m+1}\otimes C$.
\end{lemma}
\begin{proof} 
It is sufficient to show that $(\id^{\otimes i} \otimes \epsilon \otimes \id^{\otimes m+1-i})(d_1(x))=0$  for any $0<i <m$. Indeed, we have
\begin{equation*}
\begin{split}
(\id^{\otimes i}\otimes \epsilon\otimes \id^{\otimes m+1-i})(d_1(x))=&
(-1)^{\epsilon_{i-1}} a_0\otimes \overline{a_{1}}\otimes \cdots \otimes  (\epsilon \otimes \id)\Delta(a_i)\otimes
\cdots \otimes a_{m+1}+\\
&(-1)^{\epsilon_{i-2}} a_0\otimes \overline{a_1}\otimes \cdots \otimes (\id\otimes \epsilon)\Delta(a_{i-1})\otimes
\cdots \otimes a_{m+1}\\
=& 0,
\end{split}
\end{equation*}
where the first identity follows from the fact that $\epsilon(\overline{a_i})=0$ for any $i=1, \cdots, m$ and the second identity holds since 
$(\epsilon\otimes \id)\Delta=\id=(\id\otimes \epsilon)\Delta$.
\end{proof}

\begin{remark}
Analogous to the two-sided bar construction, the two-sided cobar construction $\Omega(C, C, C)$ is the total complex of the  double complex whose $(m, p)$-term is $(C\otimes (\sC)^{\otimes m} \otimes C)^p$ with the horizontal differential $d_1$ and the vertical differential $d_0$. 

We have a morphism of dg $\mathbb{K}$-modules $\iota: C\rightarrow \Omega(C, C, C)$ which is induced by the coproduct $\Delta: C \rightarrow C\otimes C$.
\end{remark}

\begin{lemma}\label{lemma-cobar-exact}
$\iota: C\rightarrow \Omega(C, C, C)$ is a quasi-isomorphism of dg $\mathbb{K}$-modules.
\end{lemma}

\begin{proof} 
It is sufficient to show that each row of the extended double complex is exact.  Namely, for each $p\in \Z$, the following complex is exact.
$$
\xymatrix{
0\ar[r] & C^p \ar[r]^-{\Delta} & (C\otimes C)^p \ar[r]^-{d_1} &\cdots \ar[r] &  (C\otimes (\sC)^{\otimes m} \otimes C)^{p+m} \ar[r] & \cdots
}
$$
Let us construct a homotopy $s_m: (C\otimes (\sC)^{\otimes m} \otimes C)^{p+m} \rightarrow (C\otimes (\sC)^{\otimes m-1}\otimes C)^{p+m-1}$ which sends $a_0\otimes \overline{a_{1}}\otimes\cdots\otimes a_{m+1}$ to  $$ (-1)^{\epsilon_m}a_0\otimes \overline{a_1}\otimes \cdots \otimes a_m \epsilon (a_{m+1})\in (C\otimes (\sC)^{\otimes m-1}\otimes C)^{p+m-1} $$
where $\epsilon_m:=\sum_{i=1}^m \deg(a_i)+m$. It is straightforward to check that  $s\circ d_1+d_1\circ s=\id$. Therefore the mapping cone $cone(\iota)$ is acyclic and thus $\iota$ is a quasi-isomorphism.
\end{proof}

Let us go back to  the case where $A$ is a dg symmetric Frobenius $\mathbb{K}$-algebra of degree $k$. Then $(s^{k}A, s^{k}\Delta, s^{k}\epsilon)$ is a counital dg coalgebra of degree zero, where
$$
\xymatrix{
s^{k}\Delta: s^{k}A\ar[r]^-{s^{k}\Delta} & s^{2k}(A\otimes A)\ar[r]^{\cong} & s^{k}A\otimes s^{k}A
}
$$
and
$$
\xymatrix{
s^{k}\epsilon: s^{k}A\ar[r]^-{s^{k} \epsilon} & s^{k} s^{-k}\mathbb{K} \ar[r]^-{\cong} & \mathbb{K}.
}
$$
To simplify the notation, we denote $(s^{k}A, s^{k}\Delta, s^{k}\epsilon)$ by $(C, \Delta, \epsilon)$. It follows from Lemma \ref{lemma-cobar-exact} and the fact that $\Delta: C \to C \otimes C$ is a morphism of dg $A$-$A$-bimodules, that  the two-sided cobar construction $\Omega(C, C, C)$ is quasi-isomorphic to $C$ as a dg $A$-$A$-bimodule.

We have a morphism of dg $A$-$A$-bimodules $\tau: B(A, A, A) \rightarrow s^{-k}\Omega(C, C, C)$ which is, by definition,  the composition of  the following morphisms
\begin{equation}\label{equation-tau}
\xymatrix{
\tau: B(A, A, A) \ar[r]^-{\pi} & A \ar[r]^-{\cong} & s^{-k} C\ar[r] ^-{s^{-k}\iota}& s^{-k} \Omega(C, C, C)
}
\end{equation}
where $\pi$ is the composition $B(A,A,A) \twoheadrightarrow A \otimes A \xrightarrow{\mu} A$ and $\iota$ is the composition $C \xrightarrow{\Delta} C \otimes C \hookrightarrow \Omega(C,C,C)$. Since $\pi$ and $\iota$ are both quasi-isomorphisms, so is the morphism $\tau$. Thus the mapping cone $cone(\tau)$ is acyclic, so $cone(\tau)\in \KK_{\ac}(\mbox{$A\otimes A^{\op}$-$\Modu_{\inj}$})$.

\begin{lemma}\label{lemma3.9}
We have an isomorphism 
$$S(A) \cong cone(\tau)$$ in $\KK_{\ac}(\mbox{$A\otimes A^{\op}$-$\Modu_{\inj}$})^c,$
where $S$ is the stabilization functor from 
$\DD_{\sg}(A\otimes A^{\op})$ to $\KK_{\ac}(\mbox{$A\otimes A^{\op}$-$\Modu_{\inj}$})^c$ (cf. Theorem \ref{theorem3.3}).
\end{lemma}

\begin{remark}\label{rem-cone}
The mapping cone $cone(\tau)$ can be viewed as the total complex of the double complex $\mathcal{E}(A, A, A)$ whose $(p, q)$-term is defined to be $(s^{-k}(C\otimes (\sC)^{\otimes p-1}\otimes C))^{p+q-1}$ when $p > 0$ and $(A\otimes (\sA)^{\otimes -p}\otimes A)^{p+q} $ when $p \leq 0$. Roughly speaking, the double complex  is obtained by connecting $B(A,A, A)$ and $s^{-k}\Omega(C, C, C)$ via the morphism $\Delta \circ \mu$, namely
$$
\cdots \rightarrow A \otimes (\sA)^{\otimes 2} \otimes A \rightarrow A \otimes \sA \otimes A \rightarrow A \otimes A \xrightarrow{\Delta \circ \mu} s^{-1-k}(C \otimes C) \rightarrow  s^{-1-k}(C  \otimes \sC \otimes C) \rightarrow \cdots
$$
\end{remark}

Let us consider the differential graded Hom-space $$\mathcal Hom_{A\otimes A^{\op}}(cone(\tau), A)= \Hom_{A\otimes A^{\op}}(\mathcal{E}(A, A, A), A).$$ By base change, we obtain a differential graded $\mathbb{K}$-module $\calD^*(A, A)$, where $$\calD^n(A, A):=\prod_{p\in \Z_{\geq 0}} \Hom_{\mathbb{K}}((\sA)^{\otimes p}, A)^n \oplus \bigoplus_{p\in \Z_{\geq 0}}\Hom_{\mathbb{K}}((\sC)^{\otimes p}, A)^{n-k+1}$$ with the differential $\delta:=\delta_0+\delta_1$, 
where $\delta_0$ is the internal differential given by, for $f\in \Hom_{\mathbb{K}}((\sA)^{\otimes p}, A)^n$ or $f \in \Hom_{\mathbb{K}}((\sC)^{\otimes p}, A)^{n-k+1},$
$$
\delta_0(f)(\overline{a_{1}}\otimes \cdots \overline{a_p})=d(f(\overline{a_1}\otimes \cdots \otimes \overline{a_p}))+\sum_{i=1}^p (-1)^{\epsilon_i}f(\overline{a_1}\otimes \cdots  \otimes \overline{da_i} \otimes \cdots \otimes \overline{a_p}).
$$
$\delta_1$ is the external differential given by
\begin{enumerate}
\item for $f\in \Hom_{\mathbb{K}}((\sA)^{\otimes p}, A)^n$
\begin{equation*}
\begin{split}
\delta_1(f)(\overline{a_{1}}\otimes \cdots\otimes  \overline{a_{p+1}}):=&a_1 f(\overline{a_{2}}\otimes \cdots \otimes \overline{a_{ p+1}})+\\&\sum_{i=1}^p (-1)^{\epsilon_i}f(\overline{a_1}\otimes \cdots \otimes \overline{a_ia_{i+1}}\otimes \cdots \overline{a_{ p+1}})+\\
&(-1)^{\epsilon_{p+1}}f(\overline{a_{1}}\otimes \cdots \otimes\overline{a_p})a_{p+1},
\end{split}
\end{equation*}
\item for $f \in \Hom_{\mathbb{K}}((\sC)^{\otimes p}, A)^{n-k+1}$ and $p>0$,
\begin{equation*}
\begin{split}
\delta_1(f)(\overline{a_{1}}\otimes \cdots \otimes \overline{a_ {p-1}}) :=&\mu (\id\otimes f)(\Delta(1)\otimes \overline{a_{1}}\otimes \cdots \otimes \overline{a_{p-1}})+\\
& \sum_{i=1}^{p-2}(-1)^{\epsilon_i} f(\overline{a_{1}}\otimes \cdots \otimes \overline{\Delta(a_i)}\otimes \cdots \otimes \overline{a_{p-1}})+\\
& (-1)^{\epsilon_n}\mu (f\otimes \id) (\overline{a_{1}}\otimes \cdots \otimes \overline{a_{p-1}} \otimes  \Delta(1)),
\end{split}
\end{equation*}
\item for $f \in \Hom_{\mathbb{K}}((\sC)^{\otimes p}, A)^{n-k+1}$ and $p=0$,
$$
\delta_1(f)=\mu (\Delta(1) \otimes (f\otimes 1)).
$$
\end{enumerate}
We call $(\calD^*(A, A), \delta)$ the  {\it Tate-Hochschild complex} of $A$ and define   the {\it Tate-Hochschild cohomology group} of degree $n$, denoted by $\THH^n(A, A)$, for any $n\in \mathbb{Z}$, to be the cohomology group $H^n(\calD^*(A, A), \delta)).$  More explicitly, $(\calD^*(A, A), \delta)$ is the totalization complex of the double complex $\calD^{*,*}(A,A)$ given by
$$
\cdots \to s^{1-k}C_{-1,*}(A,A) \to s^{1-k}C_{0,*}(A,A) \xrightarrow{\gamma} C^{0,*}(A,A) \to C^{1,*}(A,A) \to \cdots
$$
where $\gamma$ is the composition 
$$
s^{1-k}C_{0,*}(A,A) \cong s^{1-k}A \xrightarrow{\Delta} s(A \otimes A) \xrightarrow{T} s(A \otimes A) \xrightarrow{\mu} sA \xrightarrow{s^{-1}} A \cong C^{0,*}(A,A),
$$ 
and by totalization we mean the direct sum totalization in the Hochschild chains direction and the direct product totalization in the Hochschild cochains direction.

\begin{proposition}\label{prop-tate-isomorphism} 
Let $A$ be a dg symmetric Frobenius algebra over a field $\mathbb{K}$. Then for any $i \in \Z$, $$\THH^i(A, A) \cong \HH_{\sg}^i(A, A).$$
\end{proposition}

\begin{proof}
From Remark \ref{remark3.5} and Lemma \ref{lemma3.9}, it follows that 
\begin{equation*}
\begin{split}
\HH_{\sg}^i(A, A) &\cong \Hom_{\KK_{\ac}\left(\mbox{$A\otimes A^{\op}$-$\Modu_{\inj}$}\right)}(S(A), s^iS(A))\\
&\cong \Hom_{\KK_{\ac}\left(\mbox{$A\otimes A^{\op}$-$\Modu_{\inj}$}\right)}(cone(\tau), s^i cone(\tau))\\
&\cong \Hom_{A\otimes A^{\op}}(cone(\tau), s^iA)\\
&\cong H^i(\calD^*(A, A))\\
&= \THH^i(A, A).
\end{split}
\end{equation*}
\end{proof}

%Recall that we denote $\Delta(1)=\sum_i e_i\otimes f_i$. Then we have a dg $A$-$A$-bimodule homomorphism 
%$$
%\tau: A \rightarrow C
%$$
%which sends $a\in A$ to 
%$\sum_i e_i a f_i$.
\subsection{The singular Hochschild complex of a differential graded algebra}
In this subsection we fix a commutative ring $\mathbb{K}$ and a differential graded $\mathbb{K}$-algebra $A$ (not necessarily a dg symmetric Frobenius algebra). 
\\

Recall that we have defined a morphism of degree one of dg $A$-$A$-bimodules (cf. (\ref{defn-external-dif})), $$b_{-m}: \Barr_{-m}(A)\rightarrow \Barr_{-m+1}(A),$$ which induces a morphism of degree zero $$sb_{-m}: s\Barr_{-m}(A) \rightarrow \Barr_{-m+1}(A).$$ Let us denote by $\Omega^{m+1}(A)$ the kernel of $sb_{-m}$. In particular, we denote by  $\Omega^1(A)$ the kernel of  $s\mu: s(A\otimes A)\rightarrow sA$ where $\mu$ is the multiplication of $A$,  and we write $\Omega^0(A):=A$. Obviously, $\Omega^{m}(A)$ is a dg $A$-$A$-bimodule for any $m\in \Z$. Denote by $b$ the differential of $\Barr_*(A)$ and by $\pi: A\twoheadrightarrow \sA$ the natural projection map of degree one.
 
\begin{lemma}\label{lemma-bimodule}
For each $p \in \mathbb{Z}_{\geq 0}$, there is an isomorphism of dg $A$-$A$-bimodules $$\alpha: \Omega^p(A)\xrightarrow{\cong} (\sA)^{\otimes p}\otimes A,$$ where the left $A$-module structure in $(\sA)^{\otimes p}\otimes A$ is given by $$a\blacktriangleright( \overline{x_1}\otimes\cdots \otimes \overline{x_p}\otimes x_{p+1}):=(\pi\otimes \id^{\otimes p})(b(a\otimes \overline{x_1}\otimes\cdots \otimes \overline{x_p}\otimes x_{p+1})),$$ the right $A$-module structure is given by multiplication on the right $A$ factor of $(\sA)^{\otimes p}\otimes A$, and the differential on $(\sA)^{\otimes p}\otimes A$ is given by $$d(\overline{x_1}\otimes\cdots \otimes \overline{x_p}\otimes x_{p+1})=\sum_{i=1}^{p+1} (-1)^{\epsilon_{i-1}}\overline{x_1}\otimes \cdots \otimes d(\overline{x_i})\otimes \cdots \otimes x_{p+1},$$ where $\epsilon_{i-1}=\sum_{j=1}^{i-1} \deg(x_j) - i +1$. 
\end{lemma}

\begin{proof}
It is easy to check that $\blacktriangleright$ defines a dg left $A$-module structure on $(\sA)^{\otimes p} \otimes A$, namely, for any $a_1, a_2\in A$, $$a_1\blacktriangleright (a_2\blacktriangleright (\overline{x_1}\otimes \cdots \otimes \overline{x_p}\otimes x_{p+1}))=(a_1a_2)\blacktriangleright (\overline{x_1}\otimes \cdots \otimes \overline{x_p}\otimes x_{p+1}).$$ The morphism $\alpha$ is defined as the composition $$\Omega^p(A)\hookrightarrow s(A\otimes (\sA)^{\otimes p-1}\otimes A)  \xrightarrow{\pi\otimes \id^{\otimes p}} \sA\otimes (\sA)^{\otimes p-1}\otimes A. $$ The inverse of $\alpha$ is given by the morphism $\beta$ defined by the composition  $$\beta: (\sA)^{\otimes p}\otimes A \rightarrow A\otimes (\sA)^{\otimes p}\otimes A\xrightarrow{b} \Omega^p(A),$$ where the first morphism is given by $$\overline{x_1}\otimes \cdots \otimes \overline{x_p}\otimes x_{p+1} \mapsto 1\otimes \overline{x_1}\otimes \cdots \otimes \overline{x_p}\otimes x_{p+1}.$$
\end{proof}

\begin{remark}
From now on, we identify $\Omega^p(A)$ with $(\sA)^{\otimes p}\otimes A$ via the isomorphism $\alpha$. 
\end{remark}

\begin{figure} [H]
\centering
  \includegraphics[width=125mm]{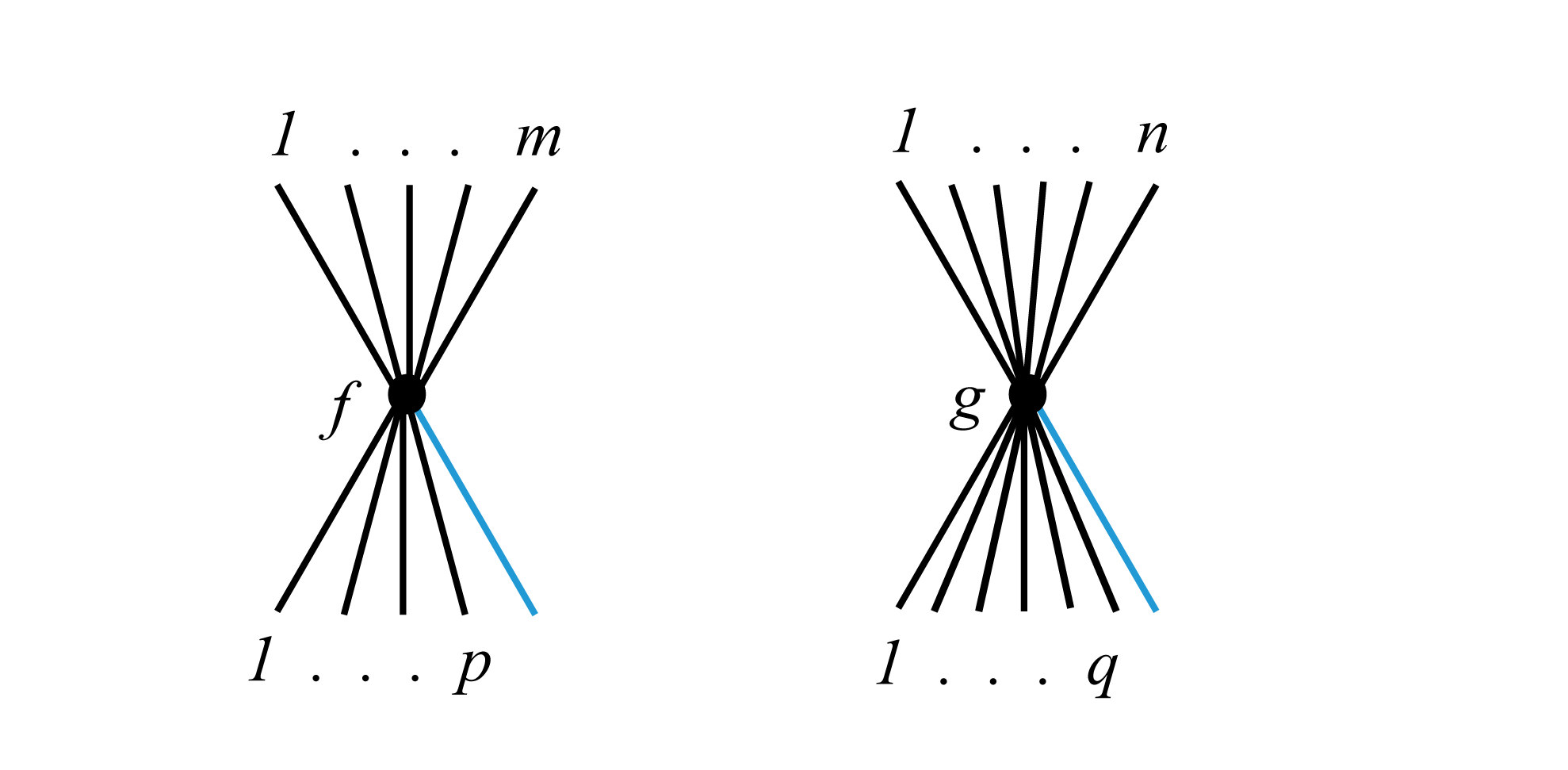}
  \caption{Using the identification of Lemma 3.12, we may represent diagrammatically elements $f \in C^{m,*}(A,\Omega^p(A))$ and $g \in C^{n,*}(A, \Omega^q(A))$ as corollas. For example $f$  has $m$ input legs at the top and $p+1$ output legs at the bottom. The rightmost output is colored blue, indicating that such a leg represents an element of $A$, while black output legs should be interpreted as elements in $s\overline{A}$. Compositions of maps will be represented by stacking trees, as usual.}
  \label{fig:corollas}
\end{figure}

Consider the Hochschild cochain complex $C^*(A, \Omega^p(A))$ with coefficients in the dg $A$-$A$-bimodule $\Omega^p(A)$. Let us define a morphism  $$\widetilde{\theta}_{ p}: C^*(A, \Omega^p(A))\rightarrow C^{*}(A, \Omega^{p+1}(A))$$ which sends an element $f\in C^*(A, (\sA)^{\otimes p} \otimes A)$ to  $\widetilde{\theta}_{m, p}(f)$ given by the following formula, 
\begin{equation}\label{equation-definition-theta}
\widetilde{\theta}_{ p}(f)(\overline{x_1}\otimes \cdots\otimes \overline{x_{k+1}})=(-1)^{\text{deg}(x_1) \text{deg}( f) } \overline{x_1} \otimes f(\overline{x_2}\otimes \cdots \otimes \overline{x_{k+1}}).
\end{equation}

\begin{lemma}\label{lemma-singular-differential}
$\widetilde{\theta}_p$ is a morphism of complexes (of degree zero) for each $p\in\Z_{\geq 0}$. Namely,  the following diagram commutes
$$
\xymatrix{
C^*(A, \Omega^p(A)) \ar[r]^-{\widetilde{\theta}_p} \ar[d]^{\delta}& C^{*}(A, \Omega^{p+1}(A))\ar[d]^-{\delta}\\
C^{*+1}(A, \Omega^p(A))\ar[r]^-{\widetilde{\theta}_p} & C^{*+1}(A, \Omega^{p+1}(A)).
}
$$
\end{lemma}

\begin{proof}
This result can be proved by a straightforward computation, but we prove it using the observation that the map $\widetilde{\theta}_p$ is given by cup product with the cocycle $d: \overline{a}\mapsto \overline{a}\otimes 1$ in $C^{1, *}(A, \Omega^1(A))$ (cf. Remark \ref{remark4.1}). 
Thus to verify the commutative diagram above is equivalent to verify the identity $d\cup \delta(f)=\delta(d\cup f)$ for any $f\in C^*(A, \Omega^p(A))$. The latter follows from the fact that $d$ is a cocycle and the cup product is compatible with the differential $\delta$. So we prove this lemma. 
\end{proof}

It follows that we have an inductive system of complexes
$$
\xymatrix{
\cdots\ar[r] & C^*(A, \Omega^p(A)) \ar[r]^-{\widetilde{\theta}_{ p}} &C^{*}(A, \Omega^{p+1}(A)) \ar[r]^{\widetilde{\theta}_{p+1}} & C^{*}(A, \Omega^{p+2}(A))\ar[r] & \cdots
}
$$
and we denote the colimit by 
$$
\calC_{\sg}^{*}(A, A):= \lim_{\substack{\longrightarrow\\p \in \Z_{\geq 0}}} C^{*}(A, \Omega^{p}(A)).
$$
Since the $\widetilde{\theta_p}$ are compatible with the differentials we obtain a differential  $$\delta_{\sg}^{m}: \calC^{m}_{\sg}(A, A)\rightarrow \calC_{\sg}^{m+1}(A, A).$$ We call the complex $(\calC_{\sg}^*(A, A), \delta_{\sg})$ the {\it singular Hochschild cochain complex} of $A$.
\\

By Lemma \ref{lemma-bar-exact},  for any $p\in \Z_{\geq 0}$, we have the following exact sequence of dg $A$-$A$-bimodules
$$
0\rightarrow \Omega^{p+1}(A)\hookrightarrow s\Barr_{-p}(A) \rightarrow s\Omega^{p}(A)\rightarrow 0.
$$
Therefore, we may take the derived functor $\HH^*(A, - )$ in the derived category $\DD(\mbox{$A\otimes A^{\op}$-$\Modu$})$  to obtain a long exact sequence
$$
\xymatrix{
\cdots\ar[r]  &  \HH^m(A, s\Barr_p(A)) \ar[r] &  \HH^m(A, s\Omega^p(A)) \ar[r]^-{\theta_{m, p}} & \HH^{m+1}(A, \Omega^{p+1}(A)) \ar[r] & \cdots 
}
$$
where $\theta_{m, p}$ denotes the connecting homomorphism.  Since $\HH^m(A, s\Omega^p(A))\cong \HH^{m+1}(A, \Omega^p(A))$, we get an inductive system for any $m\in \Z$, 
$$
\xymatrix{
\cdots\ar[r] & \HH^{m+1}(A, \Omega^p(A)) \ar[r]^-{\theta_{p}} & \HH^{m+1}(A, \Omega^{p+1}(A)) \ar[r]^{\theta_{p+1}} & 
\HH^{m+1}(A, \Omega^{p+2}(A))\ar[r] & \cdots
}
$$
and denote its colimit  by $ \lim\limits_{\substack{\longrightarrow\\p \in \Z_{\geq 0}}} \HH^{m+1}(A, \Omega^{p}(A)).$
%where the colimit  is taken along $r\in \Z$ such that $p+r\geq 0$.
%We call it the {\it singular Hochschild cohomology group} of degree $m-p$ of $A$.

\begin{lemma}\label{lemma-commute-theta}
For any $p\in \Z_{\geq 0}$ and $m\in \Z$, we have that  $H^m(\widetilde{\theta}_p)=\theta_{ p}$. Namely, the following diagram commutes
$$
\xymatrix@C=3pc{
H^m(C^{*}(A, \Omega^p(A))) 
\ar[d]^-{\cong}\ar[r]^-{H^m(\widetilde{\theta}_p)} & H^{m}(C^*(A, \Omega^{p+1}(A))\ar[d]^-{\cong}\\
\HH^{m}(A, \Omega^p(A))\ar[r]^-{\theta_{p}}&  \HH^{m}(A, \Omega^{p+1}(A)).
}
$$
\end{lemma}

\begin{proof} 
First let us recall the construction of the connecting homomorphism $$\theta_{p}:\HH^{m}(A, \Omega^p(A))\rightarrow \HH^{m}(A, \Omega^{p+1}(A)).$$ Since the bar resolution $B(A, A, A)$ is a projective resolution of $A$-$A$-bimodule $A$,  we have the following isomorphisms 
\begin{equation*}
\begin{split}
\HH^{m}(A, \Omega^p(A))&\cong \Hom_{\DD(A\otimes A^{\op})}(A, s^{m}\Omega^p(A) )\\
&\cong \Hom_{\KK(A\otimes A^{\op})}(B(A, A, A), s^{m} \Omega^p(A) ).
\end{split}
\end{equation*}
Any $f \in  \Hom_{\KK(A\otimes A^{\op})}(B(A, A, A), s^{m} \Omega^p(A))$ may be lifted uniquely to $\theta(f)$ so that the following diagram commutes
$$
\xymatrix@R=4pc@C=4pc{
B(A, A, A) \ar[rd]^-{1\otimes \overline{f}} \ar[r]^-{f} & s^m((\sA)^{\otimes p}\otimes A)\\
B(A, A, A)\ar[u]^-{d} \ar@{.>}[rd]_-{\theta(f)}  & s^{m}\Barr_{-p}(A) \ar@{->>}[u]\\
&s^{m-1}((\sA)^{\otimes p+1}\otimes A)\ar[u]
}
$$
where $d: B(A,A,A) \to B(A,A,A)$ is the differential of the two sided bar construction of $A$, the vertical maps in the right column are the two middle maps in the short exact sequence $0\rightarrow s^{m-1}\Omega^{p+1}(A)\hookrightarrow s^m\Barr_{-p}(A) \rightarrow s^m\Omega^{p}(A)\rightarrow 0$, and $\overline{f}: T(\sA) \otimes A \rightarrow s^{m} \Omega^p(A)$ is defined by $$\overline{f}(\overline{a_1} \otimes .. \otimes \overline{a_{m+p+1}} \otimes a_{m+p+1})=f(1\otimes \overline{a_1} \otimes \cdots  \otimes \overline{a_{m+1}} \otimes a_{m+p+1}).$$ The map ${\theta}$ is precisely the connection homomorphism in the long exact sequence. But note if we place the map $H(\widetilde{\theta})$ in the dotted morphism it also makes the diagram commute, so by uniqueness it follows that $H(\widetilde{\theta})=\theta$.
\end{proof}

\begin{proposition}\label{prop-quasi-isomorphism}
For any $m\in \Z$,  we have a natural isomorphism 
$$H^m(\calC_{\sg}^*(A, A)) \cong  \lim\limits_{\substack{\longrightarrow\\p \in \Z_{\geq 0}}} \HH^{m}(A, \Omega^{p}(A)).$$
In particular, via such an isomorphism, the quotient functor  $\DD^b(A\otimes A^{\op}) \rightarrow \DD_{\sg}(A\otimes A^{\op})$ induces a natural morphism $$\chi:H^*(\calC_{\sg}^*(A, A)) \rightarrow \HH_{\sg}^*(A, A).$$
\end{proposition}
\begin{proof} 
The first isomorphism  follows from Lemma \ref{lemma-commute-theta} and the fact that the homology functor commutes with colimit.  For all $p \in \mathbb{Z}_{\geq 0}$ we have a natural morphism  $$\HH^*(A, \Omega^p(A))\rightarrow \HH_{\sg}^*(A, A)$$ induced from the quotient functor $\DD^b(A\otimes A^{\op}) \rightarrow \DD_{\sg}(A\otimes A^{\op})$ and  the isomorphism $\Omega^p(A)\cong A$ in $\DD_{\sg}(A\otimes A^{\op})$ (since $\Omega^1(A)\cong A$ in $\DD_{\sg}(A\otimes A^{\op})$). These morphisms are compatible with the structure maps $\theta_{m,p}$ and thus induce a natural morphism $\chi: H^*(\calC_{\sg}^*(A, A))\rightarrow \HH_{\sg}^*(A, A)$. 
\end{proof}

\begin{remark}
It follows from \cite{Wan} that the morphism $\chi$ is an isomorphism in the case when $A$ is an ordinary associative algebra. To the best of our knowledge, it is unknown whether it is an isomorphism for any dg associative algebra.  Nevertheless, in the following we will show that the $\chi$ is an isomorphism for any dg symmetric Frobenius algebra, which is enough for the purposes of this paper. 
\end{remark}

\subsection{A homotopy retract between $\calD^*(A, A)$ and $\calC_{\sg}^*(A, A)$} \label{section-3.4}

In this subsection we assume that $(A, d, \mu, \Delta)$ is a dg symmetric Frobenius algebra of degree $k$ ($k>0$). We will construct a (strong) homotopy retract (cf. \cite[Section 10]{LoVa}) of complexes, 
\begin{equation}\label{equation-homotopy-transfer}
\xymatrix@C=4pc{
\calD^*(A, A) \ar@{^{(}->}@<2pt>[r]^-{\iota} &  \calC_{\sg}^*(A, A)\ar@{->>}@<2pt>[l]^-{\Pi}\ar@(dl, dr)_-{h}
}
\end{equation}
namely, $$ \Pi\circ \iota=\id$$ and $$ \id-\iota\circ \Pi =\delta \circ h +h\circ \delta. $$

First of all, let us construct an  injection  of dg $\mathbb{K}$-modules  $$\iota: \calD^*(A, A)\hookrightarrow \calC_{\sg}^*(A, A).$$ Recall that for any $n\in \Z$, we have $$ \calD^n(A, A):=\prod_{p\in \Z_{\geq 0}} \Hom_{\mathbb{K}}((\sA)^{\otimes p}, A)^n \oplus \bigoplus_{p\in \Z_{\geq 0}}\Hom_{\mathbb{K}}((\sC)^{\otimes p}, A)^{n-k+1}, $$ where $C:=s^{k} A$ and $\overline{C}$ is the kernel of the counit $s^k\epsilon: C\rightarrow k$.  Note that  $$ \Hom_{\mathbb K}((\sC)^{\otimes p}, A)\cong  ((\sC)^{\otimes p})^{\vee}\otimes A \cong  ((\sC)^{\vee})^{\otimes p}\otimes A\cong (\sA)^{\otimes p} \otimes A$$ where the last isomorphism follows from the isomorphism  $\sA\cong (\sC)^{\vee}$. Thus we have 
\begin{equation*}
\begin{split}
\calD^n(A, A)&\cong \prod_{p\in \Z_{\geq 0}} \Hom_{\mathbb{K}}((\sA)^{\otimes p}, A)^n \oplus \bigoplus_{p\in \Z_{\geq 0}}( (\sA)^{\otimes p}\otimes A)^{n-k+1}\\
&\cong C^n(A, A)\oplus C_{n-k+1}(A, A).
\end{split}
\end{equation*}
For any $n\in \Z$, define  $\iota_n: \calD^n(A, A) \rightarrow \calC_{\sg}^n(A, A)$ as follows:
\begin{enumerate}
\item If $f\in C^n(A, A)$, then $\iota_n(f):=f \in C^n(A, A)\subset \calC_{\sg}^n(A, A)$
\item If $f\in C_{n-k+1}(A, A)$, write $f:=\overline{a_1}\otimes \cdots \otimes \overline{a_p}\otimes a_{p+1}\in ( (\sA)^{\otimes p}\otimes A)^{n-k+1}$. Then define $\iota_n(f)\in \Hom_{\mathbb{K}}(\mathbb{K}, (\sA)^{\otimes p+1}\otimes A)^n$ by
$$\iota_n(f)(1):=\sum_{i} (-1)^{\text{deg}(f_i)\text{deg}(f)} \overline{e_i}\otimes \overline{a_1}\otimes \cdots \otimes \overline{a_p}\otimes a_{p+1}f_i. $$
\end{enumerate}
where $\sum_i e_i \otimes f_i = \Delta(1)$, as defined in Remark 2.8. It is clear that $\iota$ is indeed an injection of dg $\mathbb{K}$-modules.
\\

We now construct the morphism $\Pi: \calC_{\sg}^*(A, A)\rightarrow \calD^*(A, A)$.  If  $m,p \in \Z_{>0}$ define   $$\pi_{m, p}: C^{m, *}(A, \Omega^p(A))\rightarrow C^{m-1, *}(A, \Omega^{p-1}(A))$$ as follows: for $f\in C^{m, *}(A, (\sA)^{\otimes p}\otimes A),$ that is, $f: (\sA)^{\otimes m} \rightarrow (\sA)^{\otimes p} \otimes A,$
\begin{eqnarray*}
\lefteqn{\pi_{m,p}(f)(\overline{x_1}\otimes \cdots\otimes \overline{x_{m-1}})}\\
&:=&\sum_i (-1)^{\text{deg}(e_i)(\text{deg}(f)+1-k)} e_i \blacktriangleright(\epsilon \otimes \id^{\otimes p}) (f(\overline{f_i}\otimes \overline{x_1}\otimes \cdots \otimes \overline{x_{m-1}})))\\
&=&\sum_i (-1)^{\text{deg}(e_i)(\text{deg}(f)+1-k)} (\pi\otimes \id^{\otimes p-1 })b(e_i\otimes  (\epsilon\otimes \id^{\otimes p}) (f(\overline{f_i}\otimes \overline{x_1}\otimes \cdots \otimes\overline{x_{m-1}}))),
\end{eqnarray*}
where we recall that $\epsilon: A\rightarrow \mathbb{K}$ is the counit and the  $\blacktriangleright$-action is defined in Lemma 3.12. 
\\

Define $\pi_{m,0}=\id: C^{m,*}(A,A) \to C^{m,*}(A,A)$ for $m>0$ and for $m=0, p\in \Z_{>0}$ define  $$\pi_{0, p}: C^{0, *}(A, \Omega^p(A)) \rightarrow C_{-(p-1), *}(A, A)$$ as follows:  for any $x:=\overline{x_1}\otimes \cdots \otimes \overline{x_p}\otimes x_{p+1} \in (\sA)^{\otimes p}\otimes A$ let $$\pi_{0, p}(x):=(\epsilon(\overline{x_1})\overline{x_2}\otimes \cdots \otimes \overline{x_p}\otimes x_{p+1}).$$

\begin{remark}
Since $\epsilon(1)=0$, it follows that the counit induces a well-defined map $\epsilon: \overline{A}\rightarrow k$. In Figures \ref{fig:iota} and \ref{fig:pi} below we give a diagrammatic representation of the maps $$\iota: C_{-m,*}(A,A) \to C^0(A, (s\overline{A})^{\otimes m+1} \otimes A)$$ and $$\pi_{m,p}: C^{m,*}(A,(s\overline{A})^{\otimes p} \otimes A) \to C^{m-1,*}(A,(s\overline{A})^{\otimes p-1} \otimes A),$$ the latter in the case $m,p \in \Z_{>0}$. We represent an element of $C_{-m,*}(A,A)$ as a corolla with no inputs (this is what the solid black circle in the top leg means) and $m+1$ output legs $m$ which are colored black indicating these legs represent an element in $(s\overline{A})^{\otimes m}$ and one blue leg which represents an element of $A$. A blue circle with white interior represents the unit $1$ of the algebra $A$. We denote by $\Delta: A \to A\otimes A$ the coproduct of the dg symmetric Frobenius algebra $A$, by $\pi: A \to s\overline{A}$ the natural projection map, by $\epsilon: s\overline{A} \to\mathbb{K}$ the map induced by the counit, and by $\mu$ the product of $A$. 

\end{remark}

\begin{figure}[H]
\centering
  \includegraphics[width=125mm ]{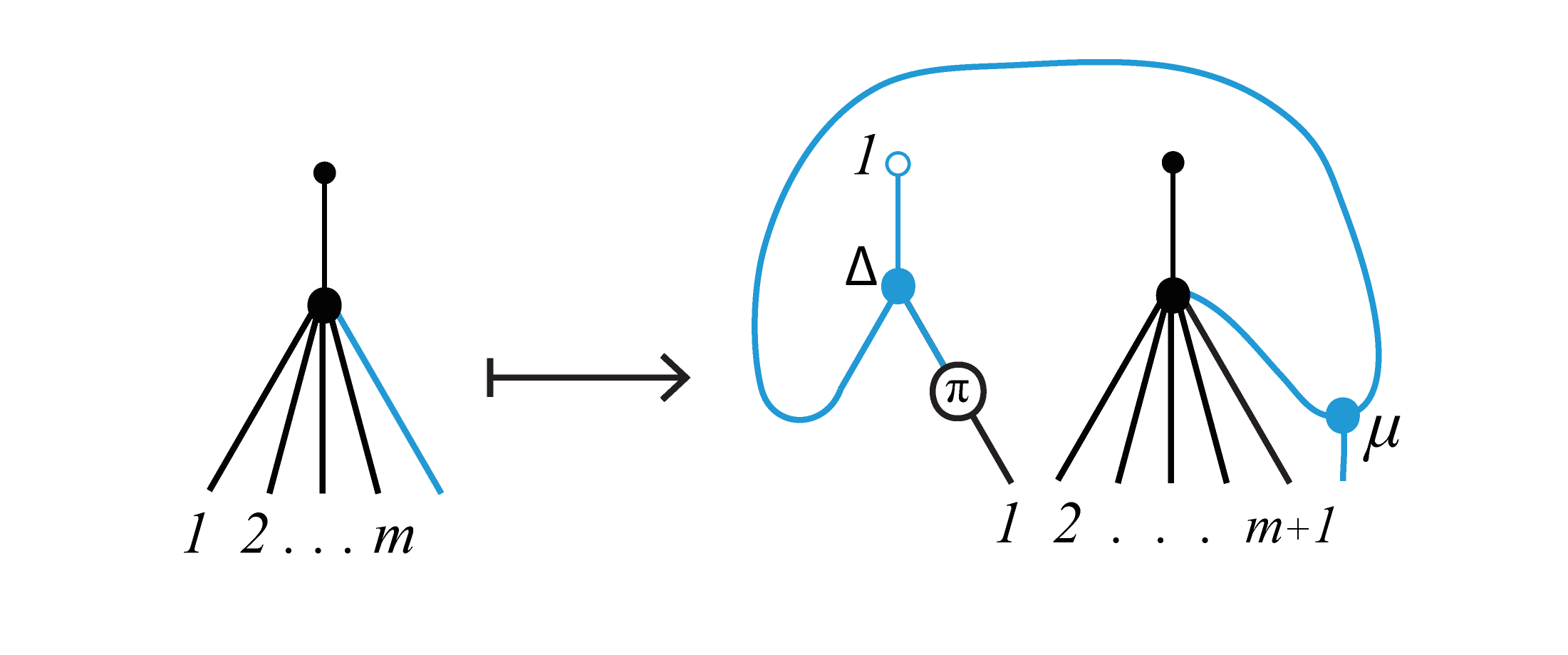}
  \caption{$\iota: C_{-m,*}(A,A) \to C^{0, *}(A, (s\overline{A})^{\otimes m+1} \otimes A) \cong \text{Hom}_{\mathbb{K}}(\mathbb{K}, (s\overline{A})^{\otimes m+1} \otimes A)$}
  \label{fig:iota} 
\end{figure}

\begin{figure}[H]
\centering
  \includegraphics[width=165mm ]{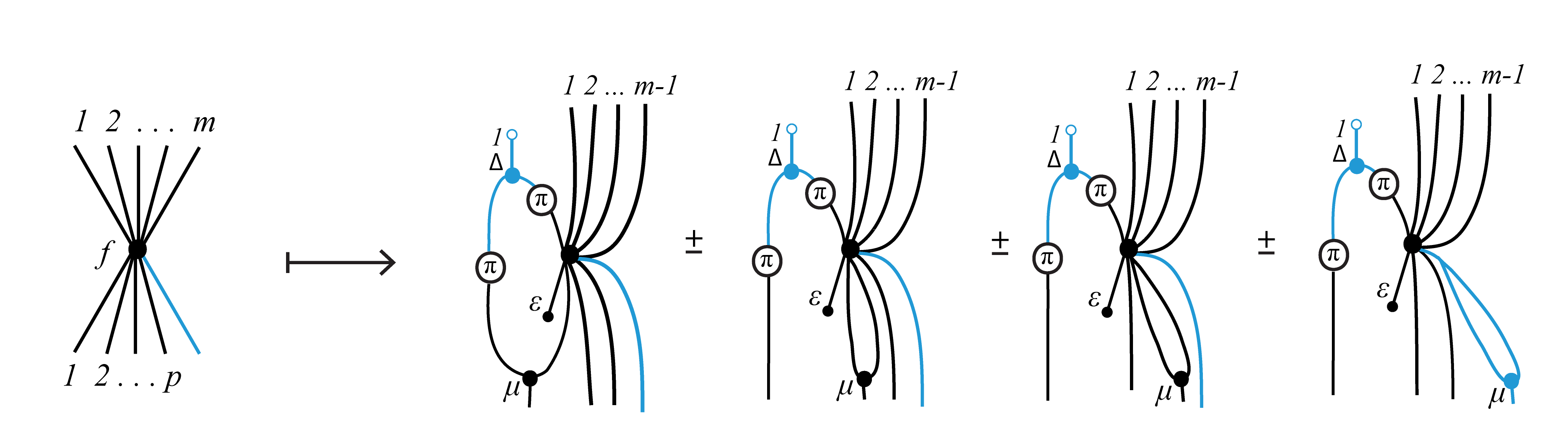}
  \caption{$\pi_{m,p}: C^{m,*}(A,(s\overline{A})^{\otimes p} \otimes A) \to C^{m-1,*}(A,(s\overline{A})^{\otimes p-1} \otimes A)$}
  \label{fig:pi} 
\end{figure}

\begin{lemma}\label{lemma5.2}
$\pi_{>0, *}$ is compatible with the differentials.
\end{lemma}

\begin{proof}
First of all, let us check that $\pi_{>0, *}$ is compatible with the external differentials. Let $f\in C^{m+1, *}(A, \Omega^p(A)), (m>-1)$. Then we have for any $\overline{x_1}\otimes \cdots \otimes \overline{x_{m+1}}\in 
(\sA)^{\otimes m+1}$
\begin{eqnarray*}
\lefteqn{\pi_{*, *}\circ \delta_1(f) (\overline{x_1}\otimes \cdots \otimes \overline{x_{m+1}})}\\
&=&\sum_i  \pm e_i \blacktriangleright(\epsilon \otimes \id^{\otimes p}) (\delta_1(f)(\overline{f_i}\otimes \overline{x_1}\otimes \cdots \otimes
\overline{x_{m+1}}))\\
&=&\sum_i \pm e_i \blacktriangleright(\epsilon \otimes \id^{\otimes p}) (f(\overline{f_i} \overline{x_1}\otimes  \cdots \otimes \overline{x_{m+1}}))\\
&&+\sum_{j=1}^m\sum_i \pm e_i \blacktriangleright(\epsilon \otimes \id^{\otimes p}) (f(\overline{f_i}\otimes \overline{x_1}\otimes \cdots \otimes\overline{x_jx_{j+1}}\otimes \cdots \otimes 
\overline{x_{m+1}}))\\
&&+\sum_i \pm e_i \blacktriangleright(\epsilon \otimes \id^{\otimes p}) (f_i\blacktriangleright f ( \overline{x_1}\otimes \cdots \otimes
\overline{x_{m+1}}))\\
&&+\sum_i \pm e_i \blacktriangleright(\epsilon \otimes \id^{\otimes p}) (f(\overline{f_i}\otimes \overline{x_1}\otimes \cdots \otimes \overline{x_m})
x_{m+1})
\end{eqnarray*}
On the other hand, we have 
\begin{eqnarray*}
\lefteqn{\delta_1\circ \pi_{*, *}(f) (\overline{x_1}\otimes \cdots \otimes \overline{x_{m+1}})}\\
&=&
%d(\pi_{*, *}(f) (\overline{x_1}\otimes \cdots \otimes \overline{x_{m+1}}))\\
%&+ \sum_{i=1}^{m+1} \pm \pi_{*, *}(f) (\overline{x_1}\otimes \cdots \otimes d(\overline{x_i})\otimes \cdots \otimes \overline{x_{m+1}})\\
\sum_{j=1}^m \pm \pi_{*, *}(f)(\overline{x_1}\otimes \cdots \otimes \overline{x_jx_{1+j}}\otimes \cdots \otimes \overline{x_{m+1}})\\
&&\pm x_1 \blacktriangleright(\pi_{*, *}(f) (\overline{x_2}\otimes \cdots \otimes \overline{x_{m+1}})) \pm (\pi_{*, *}(f) (\overline{x_1}\otimes \cdots \otimes \overline{x_m})) x_{m+1}\\
&=&
%d(\sum_i e_i \blacktriangleright(\epsilon \otimes \id^{\otimes p}) (f(\overline{f_i}\otimes \overline{x_1}\otimes \cdots \otimes
%\overline{x_{m+1}})))\\
%&+ \sum_{j=1}^{m+1} \sum_i \pm e_i \blacktriangleright(\epsilon \otimes \id^{\otimes p}) (f(\overline{f_i}\otimes \overline{x_1}\otimes \cdots \otimes d(\overline{x_j})\otimes \cdots\otimes
%\overline{x_{m+1}}))\\
\sum_{j=1}^{m}\sum_i  \pm e_i \blacktriangleright(\epsilon \otimes \id^{\otimes p}) (f(\overline{f_i}\otimes \overline{x_1}\otimes \cdots \otimes \overline{x_jx_{j+1}}\otimes \cdots \otimes
\overline{x_{m+1}}))\\
&&+x_1\blacktriangleright(\sum_i \pm e_i \blacktriangleright(\epsilon \otimes \id^{\otimes p}) (f(\overline{f_i}\otimes \overline{x_2}\otimes \cdots \otimes
\overline{x_{m+1}})))\\
&&+(\sum_i \pm e_i \blacktriangleright(\epsilon \otimes \id^{\otimes p}) (f(\overline{f_i}\otimes \overline{x_1}\otimes \cdots \otimes
\overline{x_{m}})))x_{m+1}.\\
\end{eqnarray*}

Thus, we may cancel terms to obtain:
\begin{eqnarray*}
\lefteqn{(\pi_{*, *}\circ \delta_1-\delta_1\circ \pi_{*, *})(f)(\overline{x_1}\otimes \cdots \otimes \overline{x_{m+1}})}\\
&= &\sum_i \pm e_i \blacktriangleright(\epsilon \otimes \id^{\otimes p}) (f(\overline{f_i} \overline{x_1}\otimes  \cdots \otimes 
\overline{x_{m+1}}))\\
%\sum_i \pm e_i \blacktriangleright(\epsilon \otimes \id^{\otimes p}) (d(f(\overline{f_i}\otimes \overline{x_1}\otimes \cdots \otimes
%\overline{x_{m+1}})))\\
%&+ \sum_i \pm e_i \blacktriangleright(\epsilon \otimes \id^{\otimes p}) (f(d(\overline{f_i})\otimes \overline{x_1}\otimes \cdots \otimes \cdots \otimes
%\overline{x_{m+1}})))\\
&&+\sum_i \pm e_i \blacktriangleright(\epsilon \otimes \id^{\otimes p}) (f_i\blacktriangleright f ( \overline{x_1}\otimes \cdots \otimes
\overline{x_{m+1}}))\\
%&- d(\sum_i \pm e_i \blacktriangleright(\epsilon \otimes \id^{\otimes p}) (f(\overline{f_i}\otimes \overline{x_1}\otimes \cdots \otimes
%\overline{x_{m+1}}))).\\
&&+x_1\blacktriangleright(\sum_i \pm e_i \blacktriangleright(\epsilon \otimes \id^{\otimes p}) (f(\overline{f_i}\otimes \overline{x_2}\otimes \cdots \otimes\overline{x_{m+1}}))).
\end{eqnarray*}
From the fact that $\blacktriangleright$ defines a left action of $A$ on $\Omega^p(A)$, it follows that 
\begin{eqnarray*}
\lefteqn{x_1\blacktriangleright(\sum_i \pm e_i \blacktriangleright(\epsilon \otimes \id^{\otimes p}) (f(\overline{f_i}\otimes \overline{x_2}\otimes \cdots \otimes\overline{x_{m+1}})))}\\
&=&\sum_i \pm x_1 e_i \blacktriangleright(\epsilon \otimes \id^{\otimes p}) (f(\overline{f_i}\otimes \overline{x_2}\otimes \cdots \otimes\overline{x_{m+1}}))),
\end{eqnarray*}
thus the first sum cancels with the last sum since $\sum x_1e_i\otimes f_i=\sum e_i \otimes f_i x_1$. Also, it follows from $\sum_i (-1)^ { \text{deg}(e_i) k} \overline{e_i} \otimes \epsilon(\overline{f_i})= \sum_i (-1)^ { \text{deg}(e_i) k}  \overline{e_i} \epsilon(\overline{f_i})=0$ that the second sum vanishes, namely, we have 
\begin{equation*}
\begin{split}
\sum_i \pm e_i \blacktriangleright(\epsilon \otimes \id^{\otimes p}) (f_i\blacktriangleright f ( \overline{x_1}\otimes \cdots \otimes
\overline{x_{m+1}}))=0.
\end{split}
\end{equation*}
Therefore, $(\pi_{*, *}\circ \delta_1-\delta_1\circ \pi_{*, *})(f)(\overline{x_1}\otimes \cdots \otimes \overline{x_{m+1}})=0$. Similarly, we may check that $\pi_{>0, *}$ is compatible with the internal differentials. Let $f\in C^{m+1,*}(A, \Omega^p(A))$ for $m>-1$, then 
\begin{eqnarray*}
\lefteqn{\pi_{*, *}\circ \delta_0(f)(\overline{x_1}\otimes \cdots \otimes \overline{x_m})}\\
&=&\sum_i \pm e_i\blacktriangleright (\epsilon\otimes \id^{\otimes p}) (df(\overline{f_i}\otimes \overline{x_1}\otimes \cdots \otimes \overline{x_m}))\\
&&+\sum_i \pm e_i\blacktriangleright (\epsilon\otimes \id^{\otimes p})(f(d(\overline{f_i})\otimes \overline{x_1}\otimes \cdots \otimes \overline{x_m}))\\
&&+\sum_{j=1}^m\sum_i \pm e_i\blacktriangleright (\epsilon\otimes \id^{\otimes p})(f(f_i\otimes \overline{x_1}\otimes \cdots \otimes d(x_i)\otimes \cdots \otimes \overline{x_{m}}))
\end{eqnarray*}
and 
\begin{eqnarray*}
\lefteqn{\delta_0\circ \pi_{*, *}(f)(\overline{x_1}\otimes \cdots \otimes \overline{x_m})}\\
&=& d(\sum_i \pm
e_i\blacktriangleright (\epsilon\otimes \id^{\otimes p})(f(\overline{f_i}\otimes \overline{x_1}\otimes \cdots \otimes \overline{x_m})))\\
&&\sum_{j=1}^m \sum_i \pm e_i\blacktriangleright (\epsilon\otimes \id^{\otimes p})(f(\overline{f_i}\otimes 
\overline{x_1}\otimes \cdots \otimes d(\overline{x_i})\otimes \cdots \otimes \overline{x_m})),
\end{eqnarray*}
thus we may cancel terms to obtain 
\begin{eqnarray*}
\lefteqn{(\pi_{*, *}\circ \delta_0(f)-\delta_0\circ \pi_{*, *})(\overline{x_1}\otimes \cdots \otimes \overline{x_m})}\\
&=&\sum_i \pm e_i\blacktriangleright (\epsilon\otimes \id^{\otimes p}) (df(\overline{f_i}\otimes \overline{x_1}\otimes \cdots \otimes \overline{x_m}))\\
&&+\sum_i \pm e_i\blacktriangleright (\epsilon\otimes \id^{\otimes p})(f(d(\overline{f_i})\otimes \overline{x_1}\otimes \cdots \otimes \overline{x_m}))\\
&&+d(\sum_i \pm e_i\blacktriangleright (\epsilon\otimes \id^{\otimes p})(f(\overline{f_i}\otimes \overline{x_1}\otimes \cdots\otimes \overline{x_m}))).\\
\end{eqnarray*}
Using $d(\sum_i e_i \otimes f_i)=0$ we may conclude that the three sums in the above expression vanish.  Hence the maps $\pi_{>0, *}$ are compatible with the internal differentials and thus compatible with the differentials. \end{proof}

\begin{remark}
By a similar computation, we have that $\pi_{0, *}$ are compatible with the internal differentials.  Take an element $x:=\overline{x_1}\otimes \cdots \otimes \overline{x_p}\otimes x_{p+1} \in (\sA)^{\otimes p}\otimes A,$ then 
\begin{equation*}
\begin{split}
\pi_{0, p}\circ d(x)=&\pi_{0, p}( d(\overline{x_1})\otimes \overline{x_2}\otimes \cdots \otimes \overline{x_p}\otimes x_{p+1})\\
&+ \pi_{0, p}(\sum_{i=2}^{p+1} \pm \overline{x_1}\otimes \cdots \otimes d(\overline{x_i})\otimes \cdots \otimes x_{p+1})\\
=& \epsilon(d(\overline{x_1}) ) \overline{x_2} \otimes \cdots \otimes \overline{x_p}\otimes x_{p+1}\\
&+ \sum_{i=2}^{p+1} \pm \epsilon(\overline{x_1})\overline{x_2} \otimes \cdots \otimes  d(\overline{x_i})\otimes \cdots \otimes x_{p+1}\\
=&\epsilon(\overline{x_1})d(\overline{x_2})\otimes \overline{x_3} \otimes \cdots \otimes \overline{x_p}\otimes x_{p+1}\\
&+ \sum_{i=3}^{p+1} \pm \epsilon(\overline{x_1})\overline{x_2}\otimes \overline{x_3} \otimes \cdots \otimes  d(\overline{x_i})\otimes \cdots \otimes x_{p+1}\\
=&d\circ \pi_{0, p}(x).
\end{split}
\end{equation*}
Let us check that $\pi_{0, *}$ are also compatible with the external differentials. Namely, that the following diagram commutes for any $p\in \Z_{> 0}$:
\begin{equation}\label{equation-diagram2}
\xymatrix{
C^{0, *}(A, \Omega^p(A))\ar[d]^-{\delta_1} \ar[r]^-{\pi_{0, p}}  & C_{-(p-1), *}(A, A)\ar[rd]^-{\delta_1}\\
C^{1, *}(A, \Omega^p(A)) \ar[r]^-{\pi_{1, p}} & C^{0, *}(A, \Omega^{p-1}(A)) \ar[r]^{\pi_{0, p-1}}& C_{-(p-2), *}(A, A).
}
\end{equation}
The commutativity of diagram (\ref{equation-diagram2}) follows since
%\begin{equation*}
%\begin{split}
%\delta_1\circ \pi_{0, p}(\overline{x_1}\otimes \cdots \overline{x_p}\otimes x_{p+1})=&\delta_1(\epsilon(\overline{x_1})\overline{x_2}\otimes \cdots\otimes \overline{x_p}\otimes x_{p+1})\\
%=&\epsilon(\overline{x_1})\otimes \overline{x_3}\otimes \cdots \otimes \overline{x_p}\otimes x_{p+1}x_2\\
%&+
%\end{split}
%\end{equation*}
%and 
\begin{eqnarray*}
\lefteqn{\pi_{0, p-1}\circ \pi_{1, p}\circ \delta_1(\overline{x_1}\otimes \cdots \otimes \overline{x_p}\otimes x_{p+1})}\\
&=&\sum_i\pi_{0, p-1}((\pi\otimes \id^{\otimes p-1})(b(e_i\epsilon(\overline{x_1})\otimes \overline{x_2}\otimes \cdots \overline{x_p}\otimes x_{p+1}f_i)))\\
&=& \sum_i \epsilon(\overline{e_ix_2})\epsilon(\overline{x_1})\otimes \overline{x_3}\otimes \cdots \otimes 
\overline{x_p}\otimes x_{p+1}f_i\\
&&\pm\sum_i \sum_{j=1}^{p-1} \epsilon(\overline{e_i})\epsilon(\overline{x_1}) \overline{x_2}\otimes \cdots \otimes \overline{x_{j+1}x_{j+2}}\otimes \cdots \otimes x_{p+1}f_i\\
&=& \sum_i \epsilon(\overline{x_1})\overline{x_3}\otimes \cdots \otimes \overline{x_p}\otimes x_{p+1}x_2\\
&&+\sum_i \sum_{j=1}^{p-1} \epsilon(\overline{x_1})\overline{x_2}\otimes \cdots \otimes \overline{x_{j+1}x_{j+2}}\otimes \cdots \otimes x_{p+1}f_i\\
&=& \delta_1\circ \pi_{0, p}(\overline{x_1}\otimes \cdots \otimes \overline{x_p}\otimes x_{p+1}),
\end{eqnarray*}
where the third identity follows from the facts $\sum_i\epsilon(\overline{e_i})f_i=1$ and $\sum_ix e_i\otimes f_i=\sum_i e_i \otimes f_i x$.
\end{remark}

\begin{lemma}\label{lemma5.3}
For $m, p \in \Z_{>0}$ we have $$\pi_{m, p}\circ \theta_{m-1, p-1}=\id.$$ For $m=0, p\in \Z_{\geq 0}$, we have  $$\pi_{0, p}\circ \iota=\id.$$
\end{lemma}

\begin{proof}
Recall that for any $f\in C^{m-1, *}(A, (\sA)^{\otimes p-1}\otimes A)$, $$\theta_{m-1, p-1}: C^{m-1, *}(A, \Omega^{p-1}(A))\rightarrow C^{m, *}(A, \Omega^p(A))$$ is defined by $$\theta_{m-1,p-1}(f)(\overline{x_1}\otimes \cdots\otimes \overline{x_m}):= (-1)^{\text{deg}(x_1) \text{deg}( f) } \overline{x_1}\otimes f(\overline{x_2}\otimes \cdots \otimes \overline{x_m}).$$Thus
\begin{eqnarray*}
\lefteqn{\pi_{m, p}\circ \theta_{m-1, p-1}(f) (\overline{x_1}\otimes\cdots \otimes  \overline{x_{m-1}})}\\
&=&\sum_i \pm e_i \blacktriangleright(\epsilon \otimes \id^{\otimes p}) (\theta_{m-1, p-1}(f)(\overline{f_i}\otimes \overline{x_1}\otimes \cdots \otimes
\overline{x_{m-1}})))\\
&=&\sum_i \pm e_i \blacktriangleright(  \epsilon(\overline{f_i}) f(\overline{x_1}\otimes \cdots \otimes
\overline{x_{m-1}})))\\
&=&f(\overline{x_1}\otimes \cdots\otimes \overline{x_{m-1}}).
\end{eqnarray*}

Similarly, let $\overline{x_1}\otimes \cdots \otimes \overline{x_{p-1}}\otimes x_p \in C_{-(p-1), *}(A, A),$ then we have 
\begin{equation*}
\begin{split}
\pi_{0, p}\circ \iota(\overline{x_1}\otimes \cdots \otimes \overline{x_{p-1}}\otimes x_p)& =\sum_i \pm \pi_{0, p} (\overline{e_i}\otimes
\overline{x_1}\otimes \cdots \otimes \overline{x_{p-1}}\otimes x_p f_i)\\
&=\sum_i \pm \epsilon(\overline{e_i}) \overline{x_1}\otimes \cdots \overline{x_{p-1}} \otimes x_{p}f_i\\
&=\overline{x_1}\otimes \cdots \otimes \overline{x_{p-1}}\otimes x_p.
\end{split}
\end{equation*}
\end{proof}

\begin{definition}
We define $\Pi: \calC^*_{\sg}(A, A)\rightarrow \calD^*(A, A)$ as follows:  For an  element $\overline{f}\in \calC^*_{\sg}(A, A)$ represented by an element $f\in C^{m, *}(A, \Omega^p(A))$, let
\begin{equation*}
\Pi(\overline{f}) := 
\begin{cases}
\pi_{m-p, 0}\circ\pi_{m-p+1, 1} \circ \cdots\circ \pi_{m, p}(f) & \mbox{if $m-p\geq 0$},\\
\pi_{0, p-m}\circ \pi_{1, m-p+1}\circ\cdots \circ \pi_{m, p}(f) & \mbox{if $m-p<0$}.
\end{cases}
\end{equation*}
From Lemma \ref{lemma5.3}, this is indeed well-defined, namely, $\Pi$ does not depend on the representative for $\overline{f}$.  Moreover, it follows from Lemma \ref{lemma5.2} and Lemma \ref{lemma5.3} that $\Pi$ is a morphism of (degree zero) chain complexes such  that $\Pi\circ \iota=\id$. 
\end{definition}

Finally, we construct the chain homotopy $h: \calC_{\sg}^*(A, A)\rightarrow s^{-1}\calC_{\sg}^*(A, A)$. For $m\in \Z_{>0}, p\in \Z_{> 0}$, we define a linear map  $$h_{m, p}: C^{m, *}(A, \Omega^p(A))\rightarrow C^{m-1, *}(A, \Omega^p(A))$$ which sends  $f\in C^{m, *}(A, (\sA)^{\otimes p}\otimes A) $ to  $$\overline{x_1}\otimes \cdots \otimes \overline{x_{m-1}} \mapsto \sum_{i} (-1)^{\text{deg}(e_i)(\text{deg}(f)+1-k)}  \overline{e_i}\otimes (\epsilon\otimes \id^{\otimes p}) (f(\overline{f_i}\otimes  \overline{x_1}\otimes \cdots \otimes \overline{x_{m-1}})).$$
We also define $h_{m, p}:=0$ for $mp=0$. Note that the total degree of the map $h_{m,p}$ is $-1$. For $m\in\Z_{>0}$ and $p\in \Z_{\geq 0}$ define $$H_{m, p}: C^{m, *}(A, \Omega^p(A))\rightarrow C^{m-1, *}(A, \Omega^p(A))$$ as the composition $$H_{m, p}:=\sum_{i=0}^{\min\{p,m\}}\theta_{m-2, p-1} \circ \cdots\circ \theta_{m-i-1, p-i} \circ h_{m-i, p-i}\circ  \pi_{m-i+1, p-i+1}\circ \cdots \circ \pi_{m, p}.$$ Otherwise, $H_{m, p}:=0$. See Figure 4. 

\begin{figure}
\centering
  \includegraphics[width=90mm ]{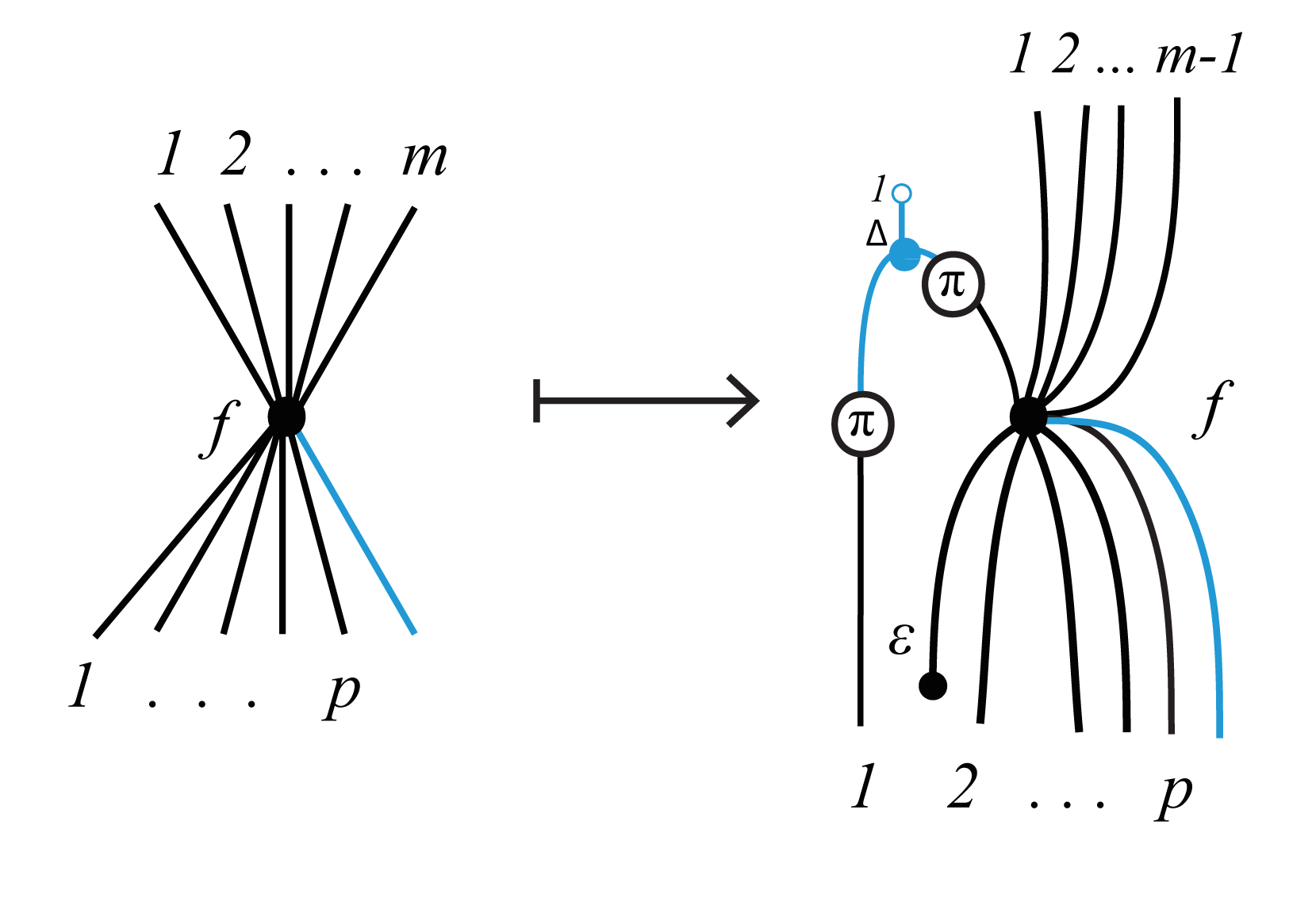}
   \caption{The above diagram represents the map $h_{m, p}: C^{m, *}(A, \Omega^p(A))\rightarrow C^{m-1, *}(A, \Omega^p(A))$. }
  \label{fig:homotopy}
\end{figure}

\begin{lemma-definition}\label{lemma-definition}
The following identity holds for $m, p\in \Z_{>0}$  $$H_{m, p}\circ \theta_{m-1, p-1}=   \theta_{m-2, p-1}\circ H_{m-1, p-1}.$$ Namely the following diagram commutes
\begin{equation*}
\xymatrix{
C^{m-1, *}(A, \Omega^{p-1}(A))\ar[d]^-{H_{m-1, p-1}}\ar[r]^-{\theta_{m-1, p-1}}  & C^{m, *}(A, \Omega^p(A))\ar[d]^-{H_{m, p}}\\
C^{m-2, *}(A, \Omega^{p-1}(A)) \ar[r]^-{\theta_{m-2, p-1}}  & C^{m-1, *}(A, \Omega^p(A)).
}
\end{equation*}
As a consequence, $H_{*, *}$ induces a well-defined morphism $$h: \calC_{\sg}^*(A, A) \rightarrow s^{-1} \calC_{\sg}^*(A, A).$$
\end{lemma-definition}

\begin{proof}
We have that
\begin{eqnarray*}
\lefteqn{H_{m, p}\circ \theta_{m-1, p-1}}\\
&=&\sum_{i=1}^{\min\{p,m\}}\theta_{m-2, p-1} \circ \cdots\circ \theta_{m-i-1, p-i} \circ h_{m-i, p-i}\circ  \pi_{m-i+1, p-i+1}\circ \cdots \circ \pi_{m, p}\circ \theta_{m-1, p-1}\\
&=&\sum_{i=1}^{\min\{p,m\}}\theta_{m-2, p-1} \circ \cdots\circ \theta_{m-i-1, p-i} \circ h_{m-i, p-i}\circ  \pi_{m-i+1, p-i+1}\circ \cdots \circ \pi_{m-1, p-1}\\
&=&\theta_{m-2, p-1}\circ H_{m-1, p-1}
\end{eqnarray*}
where the second identity follows from Lemma \ref{lemma5.3}.
\end{proof}

\begin{proposition}
The map $h$ is a chain homotopy between $\id$ and $\iota\circ \Pi$. Namely, $$\id-\iota\circ \Pi =\delta \circ h +h\circ \delta.$$
\end{proposition}

\begin{proof}
Let us first prove the following identity for $m\in\Z_{\geq 0}$ and $p\in \Z_{>0}$
\begin{equation}\label{equation5.6}
\id-\theta_{m-1, p-1}\circ \pi_{m,p}=\delta\circ h_{m, p}+h_{m+1, p} \circ\delta.
\end{equation}
We observe that $h_{*, *}$ are compatible with the internal differentials (but not with the external differentials), so it is sufficient to prove that we have the following homotopy diagram, 
\begin{equation*}
\xymatrix{
0\ar[r] & C^{0,*}(A, \Omega^p(A)) \ar[r] \ar[d]_-{\id-\iota\pi}& C^{1, *}(A, \Omega^p(A)) \ar[dl]_-{h_{1, p}}\ar[d]^-{\id-\theta\pi}\ar[r]& C^{2, *}A,\Omega^p(A))\ar[r] \ar[dl]_-{h_{2, p}}\ar[d]^-{\id-\theta\pi}& \cdots \\
0\ar[r] & C^{0,*}(A, \Omega^p(A)) \ar[r] & C^{1,*}(A, \Omega^p(A))\ar[r] & C^{2,*}(A,\Omega^p(A))\ar[r] & \cdots \\
}
\end{equation*}
Take an element $x:=\overline{x_1}\otimes\cdots \otimes \overline{x_p}\otimes x_{p+1}\in C^{0,*}(A, (\sA)^{\otimes p}\otimes A),$ then
\begin{equation*}
\begin{split}
(\id-\iota\circ \Pi)(x)=&x-\sum_{i} \pm \overline{e_i}\otimes\epsilon(\overline{x_1})\overline{x_2}\otimes \cdots\otimes \overline{x_p}\otimes x_{p+1}f_i
\end{split}
\end{equation*}
and 
\begin{equation*}
\begin{split}
h_{1, p}\circ \delta(x)&=\sum_i \pm \overline{e_i}\otimes (\epsilon\otimes \id^{\otimes p})(\delta(x)(\overline{f_i}))\\
&=\sum_i \pm \overline{e_i}\otimes (\epsilon\otimes \id^{\otimes p}) (\overline{f_i}\blacktriangleright (x))+\overline{e_i}\otimes 
(\epsilon\otimes \id^{\otimes p})x\overline{f_i}\\
&=x-\sum_{i} \pm \overline{e_i}\otimes\epsilon(\overline{x_1})\overline{x_2}\otimes \cdots\otimes \overline{x_p}\otimes x_{p+1}f_i
\\
&=x-\iota\circ \Pi(x).
\end{split}
\end{equation*}
Similarly, for $m>0$ we have  
\begin{equation*}
\begin{split}
&(\id-\theta_{m-1,p-1}\circ \pi_{m, p})(f)(\overline{x_1}\otimes \cdots \otimes \overline{x_m})\\
=&f(\overline{x_1}\otimes \cdots \otimes \overline{x_m})-\sum_i \pm \overline{x_1}\otimes e_i\blacktriangleright
(\epsilon\otimes \id )(f(\overline{f_i}\otimes\overline{x_2}\otimes \cdots \otimes \overline{x_m})).\\
\end{split}
\end{equation*}
On the other hand, we have 
\begin{equation*}
\begin{split}
&(\delta\circ h_{m,p}+h_{m+1, p}\circ \delta)(f)(\overline{x_1}\otimes \cdots \otimes \overline{x_m})\\
=&\sum_i \pm \overline{x_1}\blacktriangleright (\overline{e_i}\otimes (\epsilon\otimes \id^{\otimes p} )(f(\overline{f_i}\otimes \overline{x_2}\otimes \cdots \otimes \overline{x_m})))\\
&+\sum_j\sum_{i=1}^{m-1} \pm \overline{e_j}\otimes (\epsilon\otimes \id^{\otimes p} )(f(\overline{f_i}\otimes\overline{x_1}\otimes \cdots \otimes \overline{x_ix_{i+1}}\otimes \cdots \otimes \overline{x_m}))\\
&+\sum_i \pm \overline{e_i}\otimes (\epsilon\otimes \id^{\otimes p} )(f(\overline{f_i}\otimes \overline{x_1}\otimes \cdots \otimes \overline{x_{m-1}})\overline{x_m})\\
&+\sum_i \pm \overline{e_i}\otimes (\epsilon\otimes \id^{\otimes p})(\overline{f_i}\blacktriangleright f( \overline{x_1}\otimes \cdots \otimes
\overline{x_{m}}))\\
&+\sum_i \pm \overline{e_i}\otimes (\epsilon\otimes \id^{\otimes p})(f(\overline{f_ix_1}\otimes \overline{x_2}\otimes \cdots \otimes \overline{x_{m}}))\\
&+\sum_{j=1}^{m-1}\sum_i \pm \overline{e_i}\otimes (\epsilon\otimes \id^{\otimes p})(f(\overline{f_i}\otimes \overline{x_1}\otimes \cdots \otimes\overline{x_jx_{j+1}}\otimes \cdots \otimes \overline{x_{m}}))\\
&+\sum_i \pm \overline{e_i}\otimes (\epsilon\otimes \id^{\otimes p})(f(\overline{f_i}\otimes \overline{x_1}\otimes \cdots \otimes \overline{x_{m-1}})\overline{x_m})\\
=&\sum_i \pm \overline{x_1}\blacktriangleright (\overline{e_i}\otimes (\epsilon\otimes \id^{\otimes p} )(f(\overline{f_i}\otimes \overline{x_2}\otimes \cdots \otimes \overline{x_m})))\\
&+\sum_i \pm \overline{e_i}\otimes (\epsilon\otimes \id ^{\otimes p})(\overline{f_i}\blacktriangleright f( \overline{x_1}\otimes \cdots \otimes
\overline{x_{m}}))\\
&+\sum_i \pm \overline{e_i}\otimes (\epsilon\otimes \id ^{\otimes p})(f(\overline{f_ix_1}\otimes \overline{x_2}\otimes \cdots \otimes \overline{x_{m}}))\\
=&f(\overline{x_1}\otimes \cdots \otimes \overline{x_m})-\sum_i \pm \overline{x_1}\otimes e_i\blacktriangleright
(\epsilon\otimes \id^{\otimes p})(f(\overline{f_i}\otimes\overline{x_2}\otimes \cdots \otimes \overline{x_m})\\
=&(\id-\theta_{m-1,p-1}\circ \pi_{m, p})(f)(\overline{x_1}\otimes \cdots \otimes \overline{x_m})\\
\end{split}
\end{equation*} 
verifying identity (\ref{equation5.6}). By induction we may conclude that $ \id-\iota\circ \Pi =\delta \circ h +h\circ \delta.$
\end{proof}

\begin{corollary}\label{corollary3.26}
Let $A$ be a dg symmetric Frobenius algebra over a field $\mathbb{K}$. Then the morphism (cf. Proposition \ref{prop-quasi-isomorphism}) $$\chi: H^*(\calC^*_{\sg}(A, A)) \rightarrow \HH_{\sg}^*(A, A)$$ is an isomorphism.
\end{corollary}

\begin{proof}
We claim that  the following diagram commutes
\begin{equation}\label{equation-diagram}
\xymatrix{
H^*(\calC_{\sg}^*(A, A))  \ar[rr]^-{\chi} & &\HH_{\sg}^*(A, A) \\
&  H^*(\calD^*(A, A))\ar[ru]^-{\cong}_-{\chi'}\ar[lu]^-{H^*(\iota)}_-{\cong}
}
\end{equation}
where the isomorphism $\chi': H^*(\calD^*(A, A)) \rightarrow \HH_{\sg}^*(A, A)$ is given in Proposition \ref{prop-tate-isomorphism}. Indeed, recall that for any $m\in \Z$,  $\chi'$ can be written as the composition of the following morphisms $$H^m(\calD^*(A, A)) \rightarrow \Hom_{\KK_{\ac}}(cone(\tau), s^mcone(\tau))\xrightarrow{S^{-1}} \Hom_{\DD_{\sg}(A\otimes A^{\op})}(A, s^mA)$$ where for simplicity $\KK_{\ac}$ denotes $\KK_{\ac}(\mbox{$A\otimes A^{\op}$-$\Modu_{\inj}$}). $  If  $\alpha\in H^m(\calD^*(A, A))$ is any homogeneous element, then $H^*(\iota)(\alpha)$ is represented by an element $\alpha'\in \HH^{m+p}(A, \Omega^p(A))$ for large enough $p>0$.  Denote  by $\beta$ the image of $\alpha'$ via the following composition $$\HH^{m+p}(A, \Omega^p(A)) \rightarrow  \Hom_{\DD_{\sg}(A\otimes A^{\op})}(A, s^{m+p}\Omega^p(A)) \xrightarrow{S}\Hom_{\KK_{\ac}}(S(A), s^{m+p}S(\Omega^p(A))).$$ Note that for any $i\in \Z$, $S^i(\Omega^p(A))\cong s^{i-p}S(A)\cong s^{i-p}cone(\tau)$ in $\KK_{\ac}(\mbox{$A\otimes A^{\op}$-$\Modu_{\inj}$}),$ hence we have the following isomorphism $$\Hom_{\KK_{\ac}}(S(A), s^{m+p}S(\Omega^p(A))) \cong \Hom_{\KK_{\ac}}(cone (\tau), s^mcone(\tau)).$$ Let $\beta' \in \Hom_{\KK_{\ac}}(cone (\tau), s^mcone(\tau))$ denote the image of $\beta$ via the isomorphism above. It is easy to check that  $\beta'$ is also the image of $\alpha$ via the isomorphism $S\circ \chi'$, so the triangle diagram commutes. On the other hand, from the chain homotopy retraction described above it follows that $H^*(\iota)$ is an isomorphism, thus $\chi$ is an isomorphism as well. 
 \end{proof}

\section{DGA and DGLA structures on the Singular Hochschild complex $\calC_{\sg}^*(A,A)$}
In this section we recall briefly natural dga and dgla structures on the singular Hochschild complex $\calC_{\sg}^*(A, A)$ for a differential graded algebra $A$ over a commutative ring $\mathbb{K}$. All the constructions for dg algebras in this section are  the dg generalization of the ones  for (ordinary) associative algebras in \cite{Wan}. For more details, we refer to \cite{Wan}. Throughout this section, we assume that $A$ is  a differential graded associative algebra (not necessary symmetric Frobenius) over a commutative ring $\mathbb{K}$.
\subsection{DGA structure on $\calC_{\sg}^*(A, A)$}\label{subsection-dga}

Let $f\in C^{m, *}(A, \Omega^{p}(A))$ and $g\in C^{n, *}(A,\Omega^q(A))$, we define the cup product $f\cup g \in C^{m+n, *}(A, \Omega^{p+q}(A))$ by
$$
f\cup g:=(\id^{\otimes p+q}\otimes \mu) (\id^{\otimes q}\otimes f\otimes \id)  (\id^{\otimes m}\otimes g);
$$
where we identify $\Omega^p(A)$ with $s\overline{A}^{\otimes p}\otimes A$ as in Lemma \ref{lemma-bimodule} and denote by $\mu: A \otimes A \to A$ the multiplication in the algebra $A$. Here the formulae respect the Koszul sign rule when acting on elements.  In particular, when $p=q=0$ the cup product coincides with the usual cup product on Hochschild cochain complex $\calC^*(A, A)$.  

\begin{remark}\label{remark4.1}
We observe that the structure maps $\tilde{\theta}_p:  C^*(A, \Omega^p(A)) \to C^*(A, \Omega^{p+1}(A))$ defined in the previous section are given by cup product with the cocycle (of degree zero) $d: \overline{a} \mapsto \overline{a} \otimes 1$ in $C^{1, *}(A, \Omega^1(A))$. That is, we have $\widetilde{\theta}_p=d\cup -$.  Therefore the cup product $\cup$ is clearly compatible with the structure maps, and thus it  induces a well-defined product $$\cup: \calC_{\sg}^*(A, A) \otimes \calC_{\sg}^*(A, A) \to \calC_{\sg}^*(A, A).$$
\end{remark}

\begin{proposition}\label{proposition-dga}
The cup product $\cup$ defined above gives a (unital) dg algebra structure on the singular Hochschild complex $\calC_{\sg}^*(A, A)$. Moreover, this cup product $\cup$ is compatible with the Yoneda product of $\HH_{\sg}^*(A,A)$ via the canonical  morphism $\chi: H^*(\calC_{\sg}^*(A, A))\rightarrow \HH_{\sg}^*(A, A)$ (cf. Proposition \ref{prop-quasi-isomorphism}).
\end{proposition}

\begin{proof}
It is straightforward to verify that the cup product is associative and compatible with the differential. Note that since $\Omega^p(A)$ is an $A$-$A$-bimodule for all $p \in \mathbb{Z}_{\geq 0}$ we have the Yoneda product $$\cup': \HH^m(A, \Omega^p(A))\otimes \HH^n(A, \Omega^q(A))\rightarrow \HH^{m+n} (A, \Omega^p(A) \otimes_A \Omega^q(A)) \cong \HH^{m+n}(A, \Omega^{p+q}(A))$$ defined through the classical Hochschild cup product construction. More  precisely, take elements $\overline{f}\in \HH^m(A, \Omega^p(A))$ and  $\overline{g} \in \HH^n(A, \Omega^q(A))$, which are represented by $f\in C^{m', *}(A, \Omega^p(A))$ and $g \in C^{n', *}(A, \Omega^q(A))$ respectively, then $\overline{f} \cup' \overline{g}$ is represented by $$f\cup'g(\overline{a_1}\otimes \cdots \otimes \overline{a_{m'+n'}}):= f(\overline{a_1}\otimes \cdots \otimes \overline{a_{m'}})\otimes_A g(\overline{a_{m'+1}}\otimes \cdots \otimes \overline{a_{m'+n'}}).$$ From a direct computation, it follows that $\cup$ equals $\cup'$ up to homotopy, hence we have  $$\cup'=\cup: \HH^m(A, \Omega^p(A))\otimes \HH^n(A, \Omega^q(A))\rightarrow \HH^{m+n}(A, \Omega^{p+q}(A))$$ for any $m, n\in \Z$ and $p, q\in \Z_{\geq 0}.$   Therefore, the product $\cup'$ defines a product on the colimit $\lim\limits_{\substack{\longrightarrow\\p \in \Z_{\geq 0}}} \HH^{*}(A, \Omega^{p}(A))$, which corresponds to the cup product $\cup$ on $H^*(\calC_{\sg}^*(A, A))$ under the canonical isomorphism $$H^*(\calC_{\sg}^*(A, A))\cong \lim_{\substack{\longrightarrow\\p \in \Z_{\geq 0}}} \HH^{*}(A, \Omega^{p}(A)).$$ On the other hand, it is clear that  the Yoneda products are compatible with the morphism $\chi: H^*(\calC^*_{\sg}(A, A)) \rightarrow \HH_{\sg}^*(A, A)$ since $\chi$ is induced from the triangulated functor $\DD^b(A\otimes A^{\op})\rightarrow \DD_{\sg}(A\otimes A^{\op})$. 
 \end{proof}

\begin{figure}
\centering
  \includegraphics[width=125mm ]{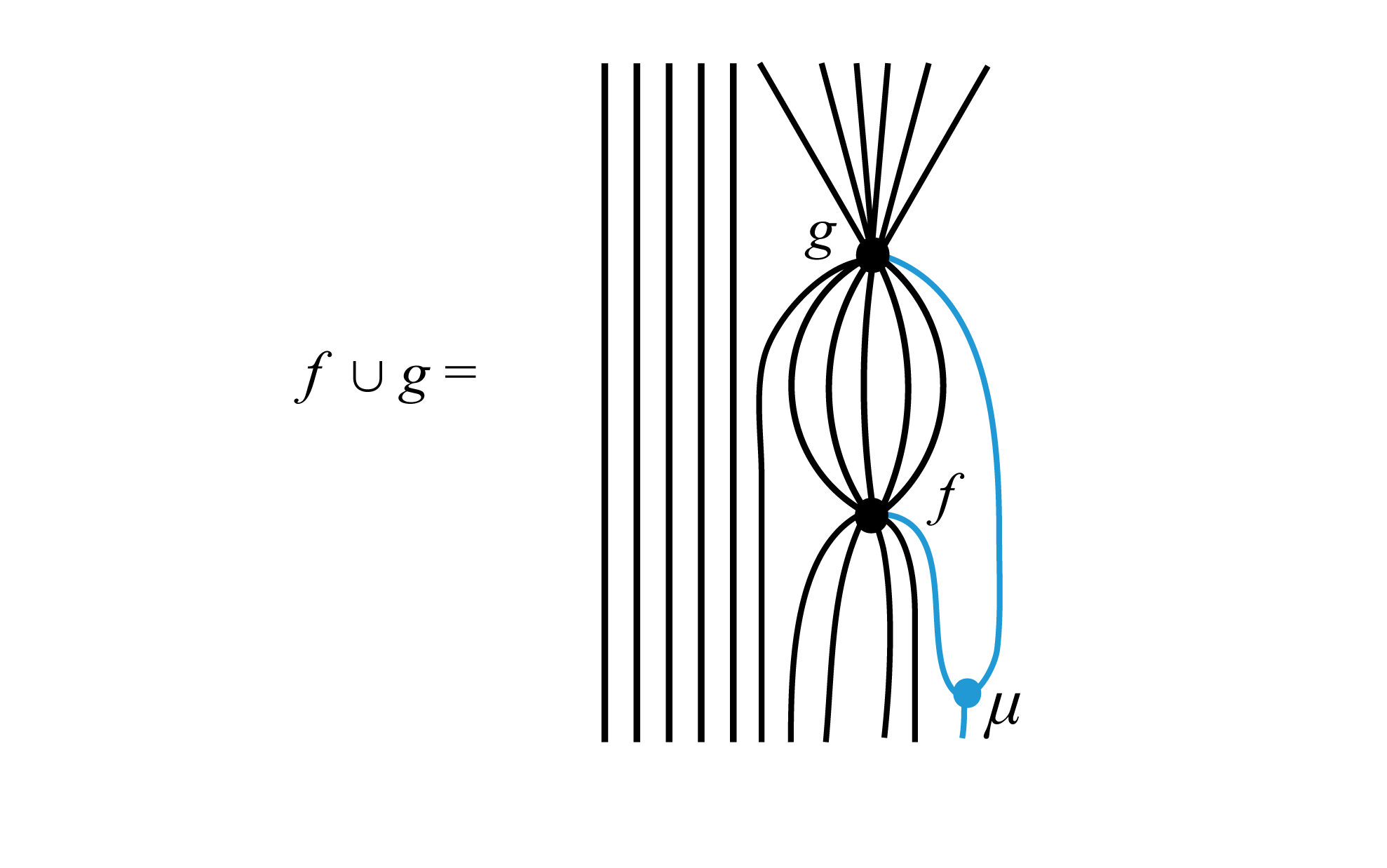}
  \caption{The diagram above represents the cup product  $\cup: \calC_{\sg}^*(A,A) \otimes \calC_{\sg}^*(A,A) \to \calC_{\sg}^*(A,A)$. The blue edges denote elements in $A$ while black edges denote elements in $s\overline{A}$. We have represented composition by stacking corollas, where $f$ and $g$ are represented as corollas as in Figure 1.}
  \label{fig:cup}
\end{figure}

A tree diagram for the cup product $\cup$ is given in Figure 5.

\subsection{DGLA structure on $\calC_{\sg}^*(A, A)$}\label{subsection-dgla}
Let $f\in C^{m, *}(A, \Omega^p(A))$ and $g\in C^{n, *}(A, \Omega^q(A))$, we define a bracket $\{f, g\}\in C^{m+n-1, *}(A, \Omega^{p+q}(A))$ as 
$$
\{f, g\}=f\bullet g-(-1)^{(\text{deg}(f)+1)(\text{deg}(g) +1)} g\bullet f
$$
where we denote  $$f\bullet g=\sum_{i=1}^m f\bullet_i g + \sum_{i=1}^p f\bullet_{-i} g$$ and $f\bullet_i g $ is defined as 
\begin{equation*}
f\bullet_i g:=
\begin{cases}
(\id^{\otimes q}\otimes f) (\id^{i-1} \otimes (\id^{\otimes q}\otimes \pi) g \otimes \id^{\otimes m-i}) & \mbox{for $i\geq 1$},\\
(\id^{\otimes p+i} \otimes (\id^{\otimes q}\otimes \pi) g \otimes \id^{\otimes -i}) (\id^{\otimes n-1}\otimes f)  & \mbox{for $i\leq -1$},
\end{cases}
\end{equation*}
where $\pi: A\twoheadrightarrow \sA$ is the natural projection map of degree $-1$, and we identify 
$\Omega^p(A)$ with $(s\overline{A})^{\otimes p}\otimes A$ as in Lemma \ref{lemma-bimodule}. Here the formulae respect the Koszul sign rule when acting on elements. In particular, when $p=q=0$ we recover the classical Gerstenhaber bracket on 
$C^*(A, A)$. It follows from a direct calculation that the bracket is compatible with the colimit construction, thus the bracket is well-defined on $\calC_{\sg}^*(A, A)$. 

\begin{proposition}
The singular Hochschild complex $\calC_{\sg}^*(A, A)$, equipped with the Lie bracket $\{\cdot, \cdot\}$ is a DGLA. 
\end{proposition}

\begin{proof}
The proof is analogous to the one of \cite[Proposition 4.6]{Wan}.
\end{proof}

\begin{remark}
The cup product $\cup$ on $\calC_{\sg}^*(A, A)$ is graded commutative up to homotopy, namely, for $f\in \calC_{\sg}^m(A, A)$ and $g\in \calC_{\sg}^n(A, A)$, $$f\cup g-(-1)^{mn} g\cup f=\delta(f)\bullet g \pm \delta(f\bullet g) \pm f\bullet \delta(g).$$ Hence, $\cup$ defines a graded commutative associative algebra structure on the cohomology $H^*(\calC_{\sg}^*(A, A))$. Moreover, we have the following result.
\end{remark}

\begin{theorem}\label{theorem-ger}
$(H^*(\calC_{\sg}^*(A, A)), \cup, \{\cdot, \cdot\})$ is a Gerstenhaber algebra.
\end{theorem}

\begin{proof}
The proof is analogous to the one of \cite[Proposition 4.9]{Wan}.
\end{proof}

\section{Products and BV operator on the Tate-Hochschild complex $\calD^*(A,A)$}\label{section5}
We now give explicit formulae for three product structures, $\star$, $\bullet$, and $[ \cdot, \cdot ]$, and an operator $\widetilde{\Delta}$ on the Tate-Hochschild complex $\calD^*(A, A)$. These operations extend some of the algebraic string operations described in [Abb] and [TrZe].  In Section 6, we relate these structures to the dga and dgla structures of $\calC_{\sg}^*(A,A)$ defined in Section 4. Diagrams representing the formulae defining $\star$ may be found in the Appendix. 

\subsection{$\star$-product on $\calD^*(A, A)$} 
We will define a product of degree zero $$ \star: \calD^*(A, A)\otimes \calD^*(A, A)\rightarrow \calD^*(A, A) $$ for a dg symmetric Frobenius algebra $A$. This $\star$-product extends the Hochschild cup product $\cup$ in $C^*(A, A)$, and the cap product $\cap$ between $C^*(A, A)$ and $C_*(A, A)$, to a product on the chain complex $(\calD^*(A, A), \delta)$ which is compatible with $\delta$, namely, $\delta$ is a derivation of $\star$. Recall that we denote $\Hom_{k}((\sA)^{\otimes m}, A)^{m+p}$ by $C^{m, p}(A, A)$ and $(A\otimes (\sA)^{\otimes m})^{-m+p}$ by $C_{-m, p}(A, A)$. Define the $\star$-product on $\calD^*(A,A)$ by the following formulae.
\begin{enumerate}
\item For any $f\in C^{m, *}(A, A)$ and $g\in C^{n, *}(A, A)$,
$$f\star g:= f\cup g \in C^{m+n,*}(A,A)$$
where $$f \cup g (\overline{a_1} \otimes \cdots  \otimes \overline{a_{m+n}})= (-1)^{\text{deg}(g)\epsilon_m}f(\overline{a_1} \otimes \cdots  \otimes \overline{a_m})g(\overline{a_{m+1}} \otimes \cdots  \otimes \overline{a_{m+n}}),$$ and $\epsilon_m=\sum_{i=1}^m \text{deg}(a_i)-m$. Namely,  this is the usual cup product in $C^*(A, A)$. 
\item For any $f\in C^{m, *}(A, A)$ and $\alpha=\overline{a_1} \otimes \cdots  \otimes \overline{a_n} \otimes a_{n+1}  \in C_{-n, *}(A, A)$
\begin{enumerate}
\item if $m-n>0$, we define $\alpha \star f , f \star \alpha \in C^{m-n-1,*}(A,A)$ as follows: For  any $\overline {b_1}\otimes \cdots\otimes \overline{b_{m-n-1}}\in (\sA)^{\otimes m-n-1}$, 
\begin{equation*}
\begin{split}
  \alpha\star f( \overline {b_1}\otimes \cdots\otimes \overline{b_{m-n-1}})&:=\sum_{i} (-1)^{\kappa_i} e_if(\overline{a_1}\otimes \cdots \otimes \overline{a_n}\otimes \overline{a_{n+1}f_i} \otimes \overline {b_1}\otimes \cdots\otimes \overline{b_{m-n-1}})\\
  f\star \alpha( \overline {b_1}\otimes \cdots\otimes \overline{b_{m-n-1}})&:= \sum_i (-1)^{\lambda_i} f(\overline{b_1}\otimes \cdots 
\otimes \overline{b_{m-n-1}}\otimes \overline{e_i} \otimes\overline{a_1}\otimes \cdots \otimes \overline{a_n})a_{n+1}f_i,
\end{split}
\end{equation*}
where the signs are given by 
$$\kappa_i= \text{deg}(f_i)\text{deg}(\alpha)+ \text{deg}(e_i)+ \text{deg}(\alpha)- \text{deg}(a_{n+1}) + (\text{deg}(\alpha)+\text{deg}(f_i)-1) \text{deg}(f),$$ and
$$\lambda_i= \text{deg}(\alpha) \text{deg}(f_i) + (\text{deg}(\alpha)+k-1)(\sum_{j=1}^{m-n-1}\text{deg}(b_j)-m+n+1)$$.
\item if $m-n\leq 0$ define $\alpha \star f, f \star \alpha \in C_{-m-n,*}(A,A)$ by
\begin{equation*}
\begin{split} 
 \alpha\star f&:=(-1)^{\epsilon_m(\text{deg}(\alpha)-\epsilon_m +\text{deg}(f))}  \overline{a_{m+1}}\otimes \cdots \otimes \overline{a_n}\otimes a_{n+1}f(\overline{a_1}\otimes \cdots \otimes \overline{a_m})\\
f\star\alpha& :=(-1)^{\epsilon_{n-m} \text{deg}(f) } \overline{a_1}\otimes \cdots \otimes \overline{a_{n-m}}\otimes f(\overline{a_{n-m+1}}\otimes \cdots \otimes \overline{a_{n}})a_{n+1},
\end{split}
\end{equation*}
where $\epsilon_l= \sum_{j=1}^l \text{deg}(a_j)- l$ as before. 
\end{enumerate}
\item For any $\alpha= \overline{a_1} \otimes \cdots  \otimes \overline{a_n} \otimes a_{n+1} \in C_{-n, *}(A, A)$ and $\beta = \overline{b_1} \otimes \cdots  \otimes \overline{b_m} \otimes b_{m+1} \in C_{-m,*}(A, A)$ define $\beta \star \alpha \in C_{-m-n-1,*}(A,A)$ by
$$\beta\star \alpha:=\sum_i (-1)^{\gamma_i} \overline{a_1}\otimes \cdots \otimes \overline{a_{n+1}e_i}\otimes \overline{b_1}\otimes \cdots \otimes \overline{b_m} \otimes b_{m+1}f_i,$$
\end{enumerate}
where $\gamma_i=\text{deg}(\beta) \text{deg}(\alpha)+(\text{deg}(\alpha) + \text{deg}(\beta))\text{deg}(f_i) + \text{deg}(\alpha)\text{deg}(e_i)+ \text{deg}(\alpha) - \text{deg}(a_{n+1})$.

\begin{remark} When $A$ is commutative the formula given in (3) for $\beta \star \alpha$ may be written as
\begin{eqnarray}\label{equation-abb}
\sum_{i} \pm \overline{a_1} \otimes \cdots  \otimes \overline{a_n} \otimes \overline{(a_{n+1}b_{m+1})'} \otimes \overline{b_1} \otimes \cdots  \otimes \overline{b_m} \otimes (a_{n+1}b_{m+1})'',
\end{eqnarray}
where we have written $\Delta: A \to A \otimes A$ as $\Delta(x)= \sum x' \otimes x''$. Formula (\ref{equation-abb}) agrees with a product of degree $k-1$ described in \cite{Abb} and \cite{Kla}. This operation does \textit{not} define a chain map on the Hochschild chain complex. In fact, this formula yields a chain map between two products $*_0$ and $*_1$, as described in \cite{Abb}. When $A$ is commutative, the product induces a chain map on the subcomplex $\bigoplus_{i\geq 1}^{\infty} (s\overline{A})^{\otimes i} \otimes A$ of $C_{*,*}(A,A)$ as proposed by \cite{Abb}. Another way to obtain a chain map in the commutative case is to modify $\star$ by defining a new product $\alpha \tilde{\star} \beta := (\alpha - p(\beta))  \star (\beta - p(\beta))$, where $p: C_{*,*}(A,A) \to \mathbb{K} \otimes A$ is the natural projection map, as proposed by \cite{Kla}. However, we do not assume commutativity in our setting. 
\end{remark}

\begin{remark}
The associativity relation holds strictly for three elements in $C_{*, *}(A, A)$ or three elements in $C^{*, *}(A, A)$. However, in general, for three mixed elements associativity holds up to homotopy, as we will see later. Furthermore, we will show that there is an $A_{\infty}$-algebra structure on $\calD^*(A,A)$ extending the differential $\delta$ and the product $\star$ (cf. Theorem \ref{theorem5.11} below). 
\end{remark}

We have a non-degenerated pairing $\langle\cdot, \cdot\rangle$ between $C^{m,*}(A, A)$ and $C_{-m, *}(A, A)$ for any $m\in \Z_{\geq 0}$. Note that when $m=0$, it is exactly  defined by the inner product of $A$. For $m>0$, $\alpha= \overline{a_1}\otimes \cdots \otimes \overline{a_m} \otimes a_{m+1} \in C_{-m, *}(A, A)$ and $f\in C^{m, *}(A, A)$, define $$\langle f, \alpha\rangle := \langle f(\overline{a_1} \otimes \cdots  \otimes \overline{a_m}), a_{m+1} \rangle$$ and
$$\langle \alpha, f \rangle:= (-1)^{\epsilon_m(\deg(a_{m+1})+\deg (f))}\langle a_{m+1},f(\overline{a_1} \otimes \cdots  \otimes \overline{a_m}) \rangle$$ where $\epsilon_m= \sum_{i=1}^m \deg(a_i)-m$. Note that with these definitions we have $\langle f, \alpha \rangle = (-1)^{\deg (\alpha) \deg (f) } \langle \alpha, f \rangle$. We may extend this pairing to $\calD^*(A,A)$ by defining $\langle f,g \rangle =0=\langle \alpha,\beta \rangle$ for any $f,g \in C^{*,*}(A,A)$ and $\alpha, \beta \in C_{*,*}(A,A)$ and  $\langle \alpha, h \rangle= 0 = \langle h, \alpha \rangle$ for any $\alpha \in C_{-m,*}(A,A)$ and $h \in C^{n,*}(A,A)$ with $m \neq n$. 

\begin{lemma}\label{lemma4.2}
The pairing is compatible with the differential in $\calD^*(A, A)$, namely, we have  $\langle \delta x, y\rangle =\langle x, \delta y \rangle$ for any $x, y \in \calD^*(A, A)$.
\end{lemma}

\begin{proof} 
This is a straightforward computation.
\end{proof}

\begin{lemma}\label{lemma-pairing}
The $\star$-product is compatible with the pairing $\langle\cdot, \cdot\rangle$. Namely, for any $x, y, z \in \calD^*(A, A)$ we have 
$$\langle x\star y, z\rangle =\langle x, y\star z\rangle.$$
%\begin{enumerate}
%\item For $f\in C^{*, *}(A, A)$ and $\alpha, \beta\in C_{*, *}(A, A)$, then 
%$$\langle \alpha\star f, \beta\rangle =\langle f, \alpha \star \beta\rangle$$
%\end{enumerate}
\end{lemma}

\begin{proof} 
Let $f\in C^{m, *}(A, A)$, $g\in C^{n, *}(A, A)$ and $\alpha \in C_{-m-n, *}(A, A)$,  then
%\begin{equation*}
%\begin{split}
%\langle f\star g, \alpha\rangle&=\langle a_0, f(a_1\otimes \cdots \otimes a_m)g(a_{m+1}\otimes \cdots a_{m+1})\rangle.
%\end{split}
%\end{equation*}
%On the other hand, 
\begin{equation*}
\begin{split}
\langle \alpha \star f, g \rangle&= \pm \langle \overline{a_{m+1}} \otimes\cdots  \otimes a_{m+1+n} f(\overline{a_1} \otimes \cdot \otimes \overline{a_m}), g \rangle \\
&= \pm \langle a_{m+1+n}f(\overline{a_1} \otimes \cdots \otimes \overline{a_m}), g(\overline{a_{m+1}} \otimes \cdots\otimes \overline{a_{m+1}} ) \rangle \\
&= \pm \langle a_{m+1+n}, f(\overline{a_1} \otimes \cdots \otimes \overline{a_m})g(\overline{a_{m+1}} \otimes \cdots \otimes \overline{a_{m+1}} ) \rangle \\
&=\pm \langle a_{m+1+n}, f \star g(\overline{a_1} \otimes \cdots \otimes \overline{a_{m+n}} \rangle\\
&= \langle \alpha, f \star g \rangle.
\end{split}
\end{equation*}
A similar calculation yields $\langle f \star \alpha, g \rangle= \langle f, \alpha \star g \rangle$. \\

Now suppose $\alpha = \overline{a_1} \otimes \cdots\otimes \overline{a_m} \otimes a_{m+1} \in C_{-m, *}(A, A)$, $\beta= \overline{b_1} \otimes \cdots \otimes \overline{b_n} \otimes b_{n+1} \in C_{-n, *}(A, A)$ and $f\in C^{m+n+1, *}(A, A)$. Then
\begin{equation*}
\begin{split}
\langle f \star \alpha, \beta \rangle&= \langle f \star\alpha (\overline{b_1} \otimes \cdots \otimes \overline{b_n}),b_{n+1} \rangle\\
&= \sum_i  \pm \langle f(\overline{b_1} \otimes\cdots \otimes \overline{b_n} \otimes \overline{e_i} \otimes \overline{a_1} \otimes \cdots  \otimes \overline{a_m})a_{m+1}f_i,b_{n+1} \rangle\\
&= \sum_i  \pm \langle f(\overline{b_1} \otimes\cdots \otimes \overline{b_n} \otimes \overline{e_i} \otimes \overline{a_1} \otimes \cdots  \otimes \overline{a_m}), a_{m+1}f_ib_{n+1} \rangle\\
&= \sum_i  \pm \langle f(\overline{b_1} \otimes\cdots  \otimes \overline{b_n} \otimes \overline{b_{n+1} e_i} \otimes \overline{a_1} \otimes \cdots  \otimes \overline{a_m}), a_{m+1}f_i \rangle\\
&= \langle f, \alpha \star \beta \rangle.
\end{split}
\end{equation*}
A similar computation yields $\langle \alpha \star f, \beta \rangle= \langle \alpha, f \star \beta \rangle$. 

\end{proof}

\subsection{$\bullet$-product and $[ \cdot, \cdot]$-bracket on $\calD^*(A, A)$}
We will define a product $$\bullet: \calD^*(A, A)\otimes \calD^*(A, A)\rightarrow \calD^*(A, A)$$ of degree $-1$ for any dg symmetric Frobenius algebra $A$. This product generalizes the Gerstenhaber $\circ$-product in $C^{*, *}(A, A)$ and  an analogous product constructed in $C_{*, *}(A, A)$ (cf. \cite{Abb, Wan}). We then define a bracket $$ [\cdot,\cdot]: \calD^*(A, A)\otimes \calD^*(A, A)\rightarrow \calD^*(A, A)$$ of degree $-1$ as the commutator of the $\bullet$-product. 
\\         
 
Define the $\bullet$-product case by case:
\begin{enumerate}
\item For $f\in C^{m, *}(A, A), g\in C^{n, *}(A, A)$,  we define $ f \bullet g \in C^{m+n-1,*}(A,A)$ by
\begin{equation*}
\begin{split}
&f\bullet g(\overline{a_1}\otimes \cdots \otimes \overline{a_{m+n-1}})\\
=& \sum_{i=1}^m (-1)^{(\deg(g)+1) \epsilon_{i-1} + \deg(f)} f(\overline{a_1} \otimes \cdots\otimes \overline{a_{i-1}} \otimes \overline{ g(\overline{a_i}\otimes
 \cdots \otimes \overline{a_{i+n-1}})} \otimes \overline{a_{i+n}} \otimes \cdots \otimes \overline{a_{m+n-1}}),
 \end{split}
 \end{equation*}
 where $\epsilon_l=\sum_{i=1}^l \deg(a_i)-l$. We will use the same definition for $\epsilon_l$ below.
 
 \item For $\alpha =\overline{a_1}\otimes \cdots \otimes \overline{a_r}\otimes a_{r+1} \in C_{-r, *}(A, A)$ and $\beta =\overline{b_1}\otimes\cdots \otimes \overline{b_s}\otimes b_{s+1} \in C_{-s, *}(A, A)$ define $\alpha \bullet \beta \in C_{-r-s-2,*}*(A,A)$ by
 $$\alpha\bullet \beta=\sum_i\sum_{j=1}^{r+1} (-1)^{\gamma_{ij}} \overline{a_1}\otimes \cdots \otimes\overline{ a_{j-1}}\otimes \overline{e_i}\otimes\overline{b_1}\otimes\cdots \otimes \overline{b_{s+1} f_i}\otimes \overline{a_j}\otimes\cdots \otimes a_{r+1},$$
 where $\gamma_{ij}= \deg(\beta)(\deg(\alpha) - \epsilon_{j-1} ) + (\deg(f_i)-1)(\epsilon_{j-1}+\deg(\beta))+ \epsilon_{j-1}\deg(e_i) -1 + \deg(\beta)-\deg(b_{s+1})$.
 
 \item Let $f\in C^{m, *}(A, A)$ and $\alpha\in C_{-r, *}(A, A)$. 
\begin{enumerate}
\item If $m\geq r+2$, 
define $f\bullet\alpha\in C^{m-r-2, *}(A, A)$ by
\begin{equation*}
\begin{split}
f\bullet \alpha:=\sum_{i, j}
\sum_{l=0}^r (-1)^{\gamma_{ijl}} \langle f(\overline{a_1}\otimes \cdots \otimes \overline{a_l}\otimes\overline{e_i}\otimes \id^{\otimes m-r-2} \otimes \overline{e_j}\otimes \overline{a_{l+1}}\otimes \cdots \otimes \overline{a_r}), a_{r+1}\rangle f_if_j,
\end{split}
\end{equation*}
where $\gamma_{ijl}=k+ (\deg(f)+\deg(\alpha))\deg(f_j)+(\deg(e_j)-1+\deg(f)+\deg(\alpha))\deg(f_i)+  (\deg(f)+ \epsilon_l) (\deg(e_i)+\deg(e_j) -2)$ and we interpret the symbol $\id^{\otimes m-r-2}$ denotes an empty slot where we plug in an element of $(s\overline{A})^{\otimes m-r-2}$. 
%$$\langle f\bullet \alpha, \beta\rangle:=\langle f, \alpha\bullet \beta\rangle$$
Similarly, define
%$$\langle \alpha\bullet f, \beta\rangle:=\langle f, \beta\bullet \alpha\rangle$$
\begin{equation*}
\begin{split}
\alpha\bullet f:=
\sum_i\sum_{j=1}^{m-r-1} (-1)^{\eta_i} f(\id^{\otimes j-1}\otimes \overline{e_i}\otimes \overline{a_1}\otimes \cdots \otimes \overline{a_{r+1}f_i}\otimes 
\id^{\otimes m-r-j-1})
\end{split}
\end{equation*}
where $\eta_{i}= \deg(f)(k-2+\deg(\alpha))+(\deg(f_i)-1)\deg(\alpha)+\deg(f) - 1 +\deg(\alpha)-\deg(a_{r+1})$. 
\item If $m< r+2$, 
define $f\bullet \alpha\in C_{-(r-m+1), *}(A, A)$ by
%$$\langle  f\bullet \alpha, g\rangle:=\langle \alpha, g \bullet f\rangle$$
\begin{equation*}
\begin{split}
f\bullet \alpha:=
\sum_{i=0}^{r-m} (-1)^{ \epsilon_i(\deg(f)+1)}   \overline{a_1}\otimes \cdots \otimes \overline{f(\overline{a_{i+1}}\otimes \cdots \otimes \overline{a_{i+m}})}\otimes \overline{a_{i+m+1}}\otimes \cdots \otimes \overline{a_r}\otimes a_{r+1},
\end{split}
\end{equation*}
and similarly 
%$$\langle \alpha\bullet f, g \rangle:=\langle \alpha, f\bullet g \rangle.$$
\begin{equation*}
\begin{split}
&\alpha\bullet f:=\sum_j\sum_{i=1}^m\\
&(-1)^{\rho_{ij}} \overline{a_i}\otimes \cdots \otimes \overline{a_{i+r-m}}\otimes e_j\langle f(\overline{a_1}\otimes \cdots \otimes \overline{a_{i-1}} \otimes \overline{f_j}\otimes  \overline{a_{i+r-m+1}} \otimes \cdots \otimes \overline{a_r}), a_{r+1}\rangle,
\end{split}
\end{equation*}
where $\rho_{ij}= \deg(e_j) + (\sum_{l=i}^{i+r-m} \deg(a_l) - r+ m -1)\epsilon_{i-1} + \epsilon_{i+r-m}( \deg(f_j) -1) + (\sum_{l=i}^{i+r-m} \deg(a_l) - r+ m -1)\deg(e_j) + \deg(f)(\epsilon_{i-1} + \deg(f_i)-1 + \deg(\alpha)- \epsilon_{i+r-n})$.
\end{enumerate}
\end{enumerate}

\begin{lemma}\label{lemma4.7}
For any $\alpha, \beta, \gamma\in \calD^*(A, A)$, we have $$\langle \alpha\bullet \beta, \gamma\rangle=\langle \alpha, \beta\bullet \gamma\rangle.$$
\end{lemma}

\begin{proof}
This is clear by direct computation. In fact, the definitions of  $f \bullet \alpha$ in case (3a) above and $\alpha \bullet f$ in case (3b) above are determined by the definition of $\bullet$ in the other cases, the non-degeneracy of the pairing, and the equation $\langle \alpha\bullet \beta, \gamma\rangle=\langle \alpha, \beta\bullet \gamma\rangle$. This is how the formula for $\bullet$ in these two cases was obtained. For the other cases the formula is given canonically by generalizing the Gerstenhaber's classical $\circ$-product. 
\end{proof}

\begin{definition}
The bracket of degree $-1$
$$[\cdot,\cdot]: \calD^*(A, A)\otimes \calD^*(A, A)\rightarrow \calD^*(A, A)$$
is defined as 
$$[x, y]:=x\bullet y-(-1)^{(|x|-1)(|y|-1)} y\bullet x.$$
Thus it follows from Lemma \ref{lemma4.7} that 
$$\langle [x, y], z\rangle=\langle x, [y, z]\rangle.$$
\end{definition}

\begin{lemma}
The bracket $[\cdot, \cdot]$ is compatible with the differential $\delta$ in $\calD^*(A, A)$. Namely, we have 
\begin{equation}\label{equation-jacobi}
\delta([x, y])=[\delta(x), y]+(-1)^{|x|-1} [x, \delta(y)] 
\end{equation}
for any $x, y \in \calD^*(A, A).$  As a consequence, the bracket $[\cdot, \cdot]$ is well-defined on $H^*(\calD^*(A, A))$. 
\end{lemma}

\begin{proof}
This follows from Lemma \ref{lemma4.2} and the fact that Identity (\ref{equation-jacobi}) holds for  either $x, y\in C^{\geq 0, *}(A, A)$ or $x, y \in C_{< 0, *}(A, A)$. 
\end{proof}

\begin{remark}
In general, the bracket $[\cdot, \cdot]$ on $\calD^*(A, A)$ does not satisfy the Jacobi identity. However, in Section \ref{section6} we will see that it does on cohomology $H^*(\calD^*(A, A))$.
\end{remark}

\subsection{BV operator $\widetilde{\Delta}$ on $\calD^*(A, A)$}\label{subsection-BV}
In this subsection we extend the Connes' boundary operator  $B$ in the Hochschild chain complex $C_*(A, A)$ to  the Tate-Hochschild complex $\calD^*(A, A)$.  Recall that Connes' boundary operator $B: C_{m, *}(A, A)\rightarrow C_{m+1, *}(A, A)$  sends a monomial element $x:=\overline{a_1}\otimes \cdots \otimes \overline{a_m}\otimes a_{m+1}\in C_{m, *}(A, A)$ to 
$$B(x):=\sum_{i=1}^{m+1}(-1)^{\epsilon_m + \epsilon_{i-1}\deg(\alpha)} \overline{a_i}\otimes \cdots \otimes \overline{a_m} \otimes \overline{a_{m+1}}\otimes \overline{a_1}\otimes
\cdots \otimes \overline{a_{i-1}}\otimes 1,$$
where $\epsilon_l=\sum_{i=1}^l \deg{a_i} - l$.  It is well-known that $B\circ B=0$ and  $B$ is compatible with the differential of $C_*(A, A)$ (cf. e.g. \cite{Lod}). Thus  $B$ induces a differential in Hochschild homology $\HH_*(A, A)$
$$B:\HH_{*}(A, A)\rightarrow \HH_{*+1}(A, A).$$

The pairing $\langle\cdot, \cdot\rangle$ of $A$ induces a linear isomorphism  $C_{m, *}(A, A)^{\vee} \cong C^{m, *}(A, A)$ for each $m\in \Z_{\geq 0}$, thus there is an isomorphism between  $C_*(A, A)^{\vee}$. We may use this isomorphism, to define
$$\Delta: C^*(A, A)\rightarrow C^{*-1}(A, A)$$
by
$$\langle\Delta(f), \overline{a_1}\otimes \cdots\otimes \overline{a_{m-1}}\otimes a_{m}\rangle:=
(-1)^{\deg(f)} \langle f, B(\overline{a_1}\otimes \cdots \otimes\overline{a_{m-1}}\otimes a_{m})\rangle.$$
\\

From $B\circ B=0$ and $B\circ d+d\circ B=0$,  it follows that $\Delta\circ \Delta=0$ and 
$\Delta\circ \delta+\delta\circ \Delta=0$, thus we have an induced map  $$\Delta: \HH^*(A, A)\rightarrow \HH^{*+1}(A, A)$$ with $\Delta\circ \Delta=0$. 
\\

We may combine the Connes' boundary operator  $B$ and its dual $\Delta$ to obtain an operator 
$$\widetilde{\Delta}: \calD^*(A, A)\rightarrow \calD^{*-1}(A, A)$$
defined by 
\begin{equation*}
\widetilde{\Delta}(x):=
\begin{cases}
\Delta(x) & \mbox{if $x\in C^{>0, *}(A, A)$},\\
0 & \mbox{if $x\in C^{0, *}(A, A)$},\\
B(x) & \mbox{if $x\in C_{\leq 0, *}(A, A)$}.
\end{cases}
\end{equation*}

\begin{remark}\label{remark-BV-operator}
$\widetilde{\Delta}$ is compatible with the pairing $\langle\cdot, \cdot\rangle$ on $\calD^*(A, A)$, that is to say,
$$\langle\widetilde{\Delta}(x), y\rangle=(-1)^{\deg(x)} \langle x, \widetilde{\Delta}(y)\rangle$$ 
for any $x, y \in \calD^*(A, A)$.  It is clear that  $\widetilde{\Delta}\circ \widetilde{\Delta}=0$ and $\widetilde{\Delta}\circ \delta+\delta\circ \widetilde{\Delta}=0,$ where  $\delta$ represents the differential of $\calD^*(A, A)$. So $\widetilde{\Delta}$ induces a differential in the cohomology  of $\calD^*(A, A)$
$$\widetilde{\Delta}:H^*(\calD^*(A, A))\rightarrow H^{*-1}(\calD^*(A, A)).$$ 
%In Section 7, we will prove that $H^*(\calD^*(A,A) \cong \THH^*(A, A) \cong \HH_{\sg}^*(A,A)$ together with the structures $\star, [\cdot, \cdot] $ and $\widetilde{\Delta}$ is a BV-algebra. 
\end{remark}

\begin{proposition}\label{prop-BV-identity}
We have the following identity on $H^*(\calD^*(A, A))$
$$[x, y]=\widetilde{\Delta}(x) \star y \pm x\star \widetilde{\Delta}(y)\pm \widetilde{\Delta}(x\star y),$$
for any $x, y \in  H^*( \calD^*(A, A))$. 
\end{proposition}

\begin{proof}

It is well-known that such an identity holds for  $[x], [y] \in H^*(\calD^*(A, A))$, where $[x]$ and $[y]$ are represented by elements $x, y \in \calD^{\geq 0, *}(A, A)$; this follows from the fact that $\HH^*(A, A)$ is a BV algebra proved in \cite{Kau, Men, Tra}. The identity also holds for  $[x], [y]\in H^*(\calD^{*, *}(A, A))$ where $[x]$ and $[y]$ are represented by elements $x, y\in \calD^{<0, *}(\calD^*(A, A))$. This follows from a computation in \cite{Abb}. Note that Abbaspour's proof  of the fact that the \textit{reduced} (shifted by $1-k$) Hochschild homology  $\widetilde{HH_*}(A,A)$ is a BV-algebra uses the assumption of graded commutativity of $A$. However, the homotopy term $H(x, y)$ (\cite[Page 738]{Abb}) also works here without the commutativity hypothesis if we switch $a_0'$ and $a_0''$ in the formula of $H(x, y)$. 
It remains to check the following cases.
\begin{enumerate}
\item If $x\in C^{m, *}(A, A)$ and $y\in C_{n, *}(A, A)$ such that  $m+n\geq 2$,  then, by definition, $[\cdot, \cdot]$ is determined by the following identity
$$
\langle[x, y], z\rangle =\langle x, [y, z]\rangle
$$
for any $z\in C_{*, *}(A, A)$. From the previous argument, it follows that on $H^*(\calD^*(A, A))$,
$$
[y, z]=\widetilde{\Delta}(x) \star y \pm x\star \widetilde{\Delta}(y)\pm \widetilde{\Delta}(x\star y).
$$
So we have that  for any $z\in C_{*, *}(A, A)$,
\begin{equation*}
\begin{split}
\langle[x, y], z\rangle &=\langle x, [y, z]\rangle\\
&=\langle x, \widetilde{\Delta}(y) \star z \pm y\star \widetilde{\Delta}(z)\pm \widetilde{\Delta}(y\star z)\rangle\\
&=\langle x\star \widetilde{\Delta}(y)\pm \widetilde{\Delta}(x\star y)\pm \widetilde{\Delta}(x)\star y, z\rangle
\end{split}
\end{equation*}
where the third identity follows from the compatibility of $\widetilde{\Delta}$ with the pairing $\langle\cdot, \cdot\rangle$ (cf. Remark \ref{remark-BV-operator}). It follows that $$[x, y]=\widetilde{\Delta}(x) \star y \pm x\star \widetilde{\Delta}(y)\pm \widetilde{\Delta}(x\star y).$$
\item For the remaining cases we may use the same argument as above to check the desired identity. 
\end{enumerate}

%\end{proof}
%\begin{proof}
%It is clear that $\widetilde{\Delta}\circ \widetilde{\Delta}=0$. The only non-trivial thing that we need to 
%check is that 
%\begin{enumerate}
%\item for  $x\in C^{0, *}(A, A)$, we have 
%$$(\widetilde{\Delta}\circ \delta+\delta\circ \widetilde{\Delta})(x)=0.$$
%Indeed, we have 
%\begin{equation*}
%\begin{split}
%(\widetilde{\Delta}\circ \delta+\delta\circ \widetilde{\Delta})(x)&=\widetilde{\Delta}\circ \delta(x)\\
%&=\widetilde{\Delta}(\delta_{in}(x))+\widetilde{\Delta}(\delta_{ex}(x))\\
%&=0
%\end{split}
%\end{equation*}
%\item for $x\in C_{0, *}(A, A)$, we have 
%$$(\widetilde{\Delta}\circ \delta+\delta\circ \widetilde{\Delta})(x)=0.$$
%Indeed, we have

%\end{enumerate}
\end{proof}

\section{$A_{\infty}$-algebra and $L_{\infty}$-algebra structures on $\calD^*(A, A)$}\label{section6}

\subsection{Homotopy transfer theorem}
We recall the Homotopy Transfer Theorem (cf. Theorem \ref{theorem-homotopy}) and use it to obtain $A_{\infty}$-algebra and  $L_{\infty}$-algebra structures on $\calD^*(A, A)$. Then we compare these transferred structures on $\calD^*(A, A)$ to the $\star$ and $[\cdot, \cdot]$ operations defined in Section \ref{section5} above. 
\begin{theorem}[Homotopy Transfer Theorem]\label{theorem-homotopy}
Let $(V, d_V)$ be a (strong) homotopy retract of $(W, d_W)$, namely we have 
\begin{equation}\label{equation-homotopy-transfer}
\xymatrix@C=4pc{
(V, d_V ) \ar@{^{(}->}@<2pt>[r]^-{\iota} &  (W, d_W)\ar@{->>}@<2pt>[l]^-{\Pi}\ar@(dl, dr)_-{h}
}
\end{equation} 
such that $$\id_W-\iota\circ \Pi=d_Wh+hd_W$$ and $$\Pi\circ \iota=\id_V.$$  Let $\mathscr{P}$ be a Koszul operad. Then any $\mathscr{P}_{\infty}$-algebra structure on $W$ can be transferred into a $\mathscr{P}_{\infty}$-algebra structure on $V$ such that $\iota$ extends to an $\infty$-quasi-isomorphism. 
\end{theorem}

\begin{remark}
We denote the operad encoding associative algebras by $Ass$ and the operad encoding Lie algebras  by $Lie$. It is well-known that these two operads are both Koszul (cf. e.g. \cite{GiKa, LoVa}). Recall that, in the case when $W$ has a dga structure then the transferred $A_{\infty}$-algebra structure on $V$ consists of maps $m_k: V^{\otimes k} \to V$ for $k=1,2,3,\cdots $ where each $m_k$ is given by the sum over all possible trivalent planar rooted trees with $k$ leaves and each tree is labeled by placing $\iota$ on the leaves, $\pi$ on the root, $h$ on the internal edges, and the product of $W$ in the (internal) vertices. For more details on the Homotopy Transfer Theorem, we refer to \cite{Kad, LoVa}.
\end{remark}

\begin{theorem}\label{theorem5.11}
There is an $A_{\infty}$-algebra structure $(m_1, m_2, \cdots)$  and an $L_{\infty}$-algebra structure $(l_1, l_2, \cdots)$ on $\calD^*(A, A)$ such that 
\begin{enumerate}
\item $m_1=\delta,  m_2=\star$, on $\calD^*(A,A)$
\item $l_1=\delta$, and $l_2=[\cdot, \cdot]$ on $H^*(\calD^*(A, A))$
\end{enumerate}
Furthermore,  $(\calD^*(A, A), m_1, m_2, \cdots )$ and $(\calD^*(A, A), l_1, l_2, \cdots )$ are  $A_{\infty}$-quasi-isomorphic and $L_{\infty}$-quasi-isomorphic to the dga and dgla structures on $\calC_{\sg}^*(A,A)$, respectively. 
\end{theorem}

\begin{proof}
It follows from the Homotopy Transfer Theorem that  we may transfer the dga structure on $\calC_{\sg}^*(A,A)$, along the homotopy retract discussed in Subsection \ref{section-3.4}, to an $A_{\infty}$-quasi-isomorphic $A_{\infty}$-algebra structure $(\delta, m_2, \cdots)$ on $\calD^*(A,A)$. Similarly, we may transfer the dgla structure on $\calC_{\sg}^*(A,A)$ to an $L_{\infty}$-quasi-isomorphic $L_{\infty}$-algebra structure $(\delta,  l_2, \cdots)$ on $\calD^*(A,A)$. We will verify that $m_2=\star$ at the chain level and $l_2=[\cdot, \cdot]$ on cohomology. We check this case by case. 
\begin{enumerate}
\item if $f,g \in C^{*,*}(A,A)$ then it is straightforward from the definitions that $m_2(f,g)= \Pi(\iota(f) \cup \iota (g))=f \cup g$. 
\item Let $\alpha=\overline{a_1}\otimes \cdots \overline{a_m}\otimes a_{m+1} \in C_{-m, *}(A, A)$ and   $\beta=\overline{b_1}\otimes \cdots \overline{b_n}\otimes b_{n+1} \in C_{-n, *}(A, A)$, where $m,n\in \Z_{\geq 0},$ then by the formulae for the transferred structure given in the proof of the Homotopy Transfer Theorem as recalled in Remark 6.2, we have  that 
\begin{equation*}
\begin{split}
m_2(\beta, \alpha)&=\Pi(\iota(\beta)\cup \iota(\alpha))\\
&=\pm \sum_{i, j}(\epsilon\otimes \id)(\overline{e_i}\otimes \overline{a_1}\otimes \cdots \otimes \overline{a_m} \otimes \overline{e_j}
\otimes \overline{b_1}\otimes \cdots \otimes b_{n+1}f_ja_{m+1}f_i)\\
&=\pm \sum_j\overline{a_1}\otimes \cdots \otimes \overline{a_m}\otimes \overline{a_{m+1}e_j}\otimes \overline{b_1}
\otimes \overline{b_n}\otimes b_{n+1}f_j\\
&= \beta\star \alpha.
\end{split}
\end{equation*}
\item If $\alpha\in C_{-m, *}(A, A)$ and $f \in C^{n, *}(A, A)$ where $m, n\in \Z_{\geq 0},$ then 
\begin{enumerate}
\item if $m<n$, then we have that $m_2(\alpha, f)\in C^{n-m-1, *}(A, A)$, moreover
\begin{equation*}
\begin{split}
m_2(\alpha, f)&= \Pi(\iota(\alpha)\cup \iota(f))\\
&= \sum_i \pm (\pi_{n-m-1, 0}\circ  \cdots \circ \pi_{n, m+1})(\overline{e_i}\otimes\overline{a_1}\otimes\cdots \otimes a_{m+1}f_i f(?))\\
&=\sum_i \pm (\pi_{n-m-1, 0}\circ  \cdots \circ \pi_{n-1, m})(\overline{e_ia_1}\otimes \overline{a_2}\otimes \cdots \otimes a_{m+1}f(\overline{f_i}\otimes ?))\\
&= \sum_i \pm (\pi_{n-m-1, 0}\circ  \cdots \circ \pi_{n-2, m-1})(\overline{e_ia_2}\otimes \overline{a_3}\otimes \cdots \otimes a_{m+1}f(\overline{a_1}\otimes \overline{f_i}\otimes ?))\\
&=\cdots\\
& \sum_{i} \pm e_i f(\overline{a_1}\otimes \cdots \otimes \overline{a_{m+1}f_i}\otimes ?)\\
&=\alpha\star f
\end{split}
\end{equation*}
where we wrote $\alpha:=\overline{a_1}\otimes \cdots \otimes \overline{a_m}\otimes a_{m+1}$ and the question mark $?$ just means a slot to plug in any monomial of length determined by the fact that $f\in C^{n,*}(A,A)$. Similarly, we have
\begin{equation*}
\begin{split}
m_2(f, \alpha)&=\Pi(\iota(f)\cup \iota(\alpha))\\
&= \sum_i \pm \Pi(f (?\otimes \overline{e_i}\otimes \overline{a_1}\otimes \cdots \otimes 
\overline{a_m})a_{m+1}f_i)\\
&=f \star \alpha
\end{split}
\end{equation*}
\item if $m\geq n$, then we have that $m_2(\alpha, f)\in C_{-(m-n), *}(A, A)$, thus 
\begin{equation*}
\begin{split}
m_2(\alpha, f)&=\Pi(\iota(\alpha)\cup \iota( f))\\
&= \sum_i \pm (\pi_{0, m-n+1}\circ \cdots \circ \pi_{n, m+1})(\overline{e_i}\otimes \overline{a_1}\otimes \cdots \otimes a_{m+1}f_if(? ))\\
%&=\sum_i(\pi_{0, m-n+1}\circ \cdots \circ \pi_{n-1, m})(\overline{e_i}\otimes\overline{a_1}\otimes \cdots\otimes a_{m+1}f_if(? ))\\
&= \sum_i \pm (\pi_{0, m-n+1}\circ \cdots \circ \pi_{n-1, m})(\overline{e_ia_1}\otimes \cdots\otimes a_{m+1}f(f_i\otimes ? ))\\
&=\sum_i \pm (\pi_{0, m-n+1}\circ \cdots \circ \pi_{n-2, m-1})(\overline{e_ia_2}\otimes \cdots\otimes a_{m+1}f(\overline{a_1}\otimes \overline{f_i}\otimes? ))\\
&=\cdots\\
&=\sum_i \pm \pi_{0, m-n-1}(\overline{e_ia_n}\otimes \overline{a_{n+1}}\otimes \cdots \otimes a_{m+1}f(\overline{a_1}\otimes \cdots \otimes 
\overline{a_{n-1}}\otimes \overline{f_i}))\\
&=\pm \overline{a_{n+1}}\otimes \cdots \otimes a_{m+1}f(\overline{a_1}\otimes \cdots\otimes \overline{a_n})\\
&=\alpha\star f.
\end{split}
\end{equation*}
Similarly, we have
\begin{equation*}
\begin{split}
m_2(f, \alpha)&=\Pi(\iota(f)\cup \iota( \alpha))\\
&= \sum_i \pm (\pi_{0, m-n+1}\circ \cdots \circ \pi_{n, m+1})(\overline{e_i}\otimes \overline{a_1}\otimes \cdots \otimes a_{m+1}f_if(? ))\\
&=\pm \overline{a_1}\otimes \cdots \otimes \overline{a_{m-n}}\otimes f(\overline{a_{m-n+1}}\otimes \cdots \otimes \overline{a_m})a_{m+1}\\
&=f \star \alpha.
\end{split}
\end{equation*}
\end{enumerate}
%\item The other cases are similar to the Case 2 above.
\end{enumerate}
Therefore, we have verified that $m_2=\star$. 
%Let us compute $m_3$. By the Homotopy transfer theorem, we have that 
%$$m_3(\alpha_1, \alpha_2, \alpha_3)=\pi (h(\iota(\alpha_1)\cup \iota(\alpha_2))\cup \iota(\alpha_3))-
%\pi (\iota(\alpha_1)\cup h(\iota(\alpha_2)\cup \iota(\alpha_3))).$$
%Observe that 
%$h(\iota(\alpha_1)\cup \iota(\alpha_2))=0$ if either 
%$\alpha_1\in C^{\geq 0, *}(A, A)$ or $\alpha_2\in C_{\leq 0, *}(A, A).$
%Hence $m_3(\alpha_1, \alpha_2, \alpha_3)$ vanish for almost all cases except 
%the following two ones where either $\alpha_1, \alpha_3\in C_{\leq 0, *}(A, A), \alpha_2\in C^{\geq 0, *}(A, A)$ 
%or $\alpha_1, \alpha_3\in C^{\geq 0,  *}(A, A), \alpha_2\in C_{\leq 0, *}(A, A)$. 
\\
\\
It remains to verify $l_2=[\cdot, \cdot]$ on $H^*(\calD^*(A, A))$. 
Again, it follows from the proof of the Homotopy Transfer Theorem that we have the following formulae for $l_2$.
\begin{enumerate}
\item If either $\alpha, \beta \in C_{*}(A, A)$ or $\alpha, \beta\in C^*(A, A)$, 
then 
$$l_2(\alpha, \beta)=[\alpha, \beta].$$
\item If $\alpha\in C_{-r, *}(A, A)$ and $f\in C^{m, *}(A, A)$ such that $m\geq r+2$, then 
\begin{equation}
\begin{split}
&l_2(\alpha, f)(\overline{b_1}\otimes \cdots \otimes \overline{b_n})=\\
&\sum_{i=1}^{r+1} \pm e_jf(\overline{a_{i}}\otimes \cdots \otimes \overline{a_{r+1}}\otimes \overline{a_1}\otimes \cdots \otimes \overline{a_{i-1}} \otimes   \overline{f_j}\otimes  \overline{b_1}\otimes \cdots \otimes \overline{b_n})+\\
&\sum_{i=1}^n \pm f(\overline{b_1}\otimes \cdots \otimes \overline{b_{i-1}}\otimes \overline{e_j}\otimes \overline{a_1}\otimes \cdots \otimes 
\overline{a_{r+1}f_j}\otimes \overline{b_{i}}\otimes \cdots \otimes \overline{b_n})+\\
&\sum_{i=1}^{r+1} \pm \langle f(\overline{a_i}\otimes \cdots \otimes \overline{a_r}\otimes \overline{e_j} \otimes\overline{b_1}\otimes \cdots \otimes \overline{b_n}\otimes \overline{e_k}\otimes \overline{a_1}\otimes \cdots \otimes \overline{a_{i-1}}), 1\rangle f_ja_{r+1}f_k,
\end{split}
\end{equation}
where $n=m-r-2$ and we wrote $\alpha:=\overline{a_1}\otimes \cdots \otimes \overline{a_r}\otimes a_{r+1}$.
Define $H(\alpha, f)\in C^{m-r-2, *}(A, A)$ as follows,
\begin{equation*}
\begin{split}
&H(\alpha, f)(\overline{b_1}\otimes \cdots \otimes \overline{b_{m-r-3}})=\sum_{i=1}^{r+1}\sum_{j=0}^{i-1} \sum_{k, l}\\
&\pm \langle f(\overline{a_i}\otimes \cdots \otimes \overline{a_{r+1}}\otimes 
\cdots\otimes \overline{ a_j}\otimes \overline{e_k} \otimes \overline{b_1}\otimes \cdots \otimes \overline{b_{m-r-3}}\otimes \overline{e_l} \otimes \overline{a_{j+1}}\otimes 
\cdots \otimes \overline{a_{i-1}}), 1\rangle f_kf_l.
\end{split}
\end{equation*}
Then by direct computation, 
we have the following identity,
$$
l_2(\alpha, f)-[\alpha, f]=\delta\circ H+H\circ \delta,
$$
thus, we have that on cohomology, 
$$
l_2(\alpha, f)=[\alpha, f].
$$ 
\item If $\alpha\in C_{-r, *}(A, A)$ and $f\in C^{m, *}(A, A)$ such that $m<r+2$, then 
\begin{equation}
\begin{split}
l_2(\alpha, f)
=&\sum_{i=1}^{m} \pm \langle f(\overline{a_i}\otimes \cdots \otimes \overline{a_{m-1}}\otimes \overline{e_j}\otimes \overline{a_1}\otimes\cdots \otimes \overline{a_{i-1}}), 1\rangle\overline{a_m}\otimes \cdots \otimes a_{r+1}f_j+\\
&\sum_{i=0}^{r-m} \pm \overline{a_1}\otimes \cdots \otimes \overline{a_i}\otimes \overline{ f(\overline{a_{i+1}}\otimes \cdots \otimes \overline{a_{i+m}})}\otimes\overline{a_{i+m+1}}\otimes \cdots \otimes a_{r+1}+\\
&\sum_{i=1}^m \pm \overline{a_i}\otimes\cdots\otimes \overline{a_{r-m+i}}\otimes f(\overline{a_{r-m+i+1}}\otimes \cdots \otimes \overline{a_{r+1}}\otimes \cdots \otimes \overline{a_{i-1}}).
\end{split}
\end{equation}
Define 
$H'(\alpha, f)\in C_{-(r-m), *}(A, A)$ as follows,
\begin{equation*}
\begin{split}
&H'(\alpha, f)=\sum_{i=1}^{m-1}\sum_{j=1}^i \sum_k\pm\\
& \overline{a_i}\otimes \cdots \otimes \overline{a_{i+r-m+1}}
\otimes e_k \langle f(\overline{a_j}\otimes \cdots \otimes \overline{a_{i-1}}\otimes \overline{f_k}
\otimes \overline{a_{i+r-m+2}}\otimes \cdots \otimes \overline{a_{r+1}}\otimes \overline{a_1}\otimes 
\cdots \overline{a_{j-1}}), 1\rangle.
\end{split}
\end{equation*}
Then by direct computation, we have 
$$
l_2(\alpha, f)-[\alpha, f]=\delta\circ H'+H'\circ \delta,
$$
thus we have that on cohomology
$$l_2(\alpha, f)=[\alpha, f].$$
\end{enumerate}
Therefore, we have verified $l_2=[\cdot, \cdot]$ on $H^*(\calD^*(A, A))$.

\end{proof}
% Let us check that $m_3=H$ in the following. 
%First recall that 
%$$
%m_3(\alpha, \beta, \gamma)=m_2( h\circ m_2(\iota(\alpha), \iota(\beta)), \iota(\gamma))-m_2(\iota(\alpha), h\circ m_2(\iota(\beta), \iota(\gamma))).
%$$
%Then we have the following cases:
%\begin{enumerate}
%\item If $\alpha, \beta, \gamma\in C^{*, *}(A, A)$, then it is clear  that $m_3(\alpha, \beta, \gamma)=0$ from the associativity of the cup product in $C^{*, *}(A, A)$. Analogously, if $\alpha, \beta,\gamma\in C_{*, *}(A, A)$, then $m_3(\alpha, \beta, \gamma)=0$ from the 
%associativity of the $\star$-product in $C_{*, *}(A, A)$.
%\item If $\alpha\in C^{m, *}(A, A), \beta\in C^{n, *}(A,A)$ and $\gamma\in C_{-r, *}(A, A)$,
%\begin{equation*}
%\begin{split}
%m_3(\alpha, \beta,\gamma)&= 
%\end{split}
%\end{equation*}

%\end{enumerate}

\begin{remark}
The homotopy $m_3$ for the associativity of $\star=m_2$ is determined by the following explicit formulae which may be obtained from the recipe for transferring a dga structure along a a homotopy retraction. 
\begin{enumerate}
\item If $f, g, h\in C^{*, *}(A, A)$, 
then $m_3(f, g, h)=0$.
\item If $\alpha, \beta, \gamma\in C_{*, *}(A, A)$, then $m_3(\alpha, \beta, \gamma)=0$.
\item If $\alpha, \beta \in C_{*,*}(A,A)$ and $f \in C^{*,*}(A,A)$, then $m_3(\alpha, \beta, f)=0=m_3(f, \alpha, \beta)$. 
\item If $\alpha \in C_{*,*}(A,A)$ and $f,g \in C^{*,*}(A,A)$, then $m_3(f,g,\alpha)=0=m_3(\alpha, f,g)$. 
\item For $f\in C^{m, *}(A, A), g\in C^{n, *}(A, A)$ and $\alpha=\overline{a_1}\otimes\cdots \otimes \overline{a_r}\otimes a_{r+1} \in C_{-r, *}(A, A)$  such that $r\leq m+n$, then $m_3(f,\alpha,g) \in C^{m-r+n,*}(A,A)$ is defined by
\begin{eqnarray*}
\lefteqn{
m_3(f, \alpha, g)=\sum_i\sum_{j=1}^{\min\{m, n, r\}}}\\
&&(-1)^{\kappa_{ij}} f(\id^{\otimes m-r+j }\otimes\overline{e_i}\otimes \overline{a_{j}}\otimes \cdots \otimes \overline{a_r})a_{r+1} 
g(\overline{a_1}\otimes \cdots \otimes \overline{a_{j-1}}\otimes \overline{f_i}\otimes \id^{\otimes n-j}), \\
\end{eqnarray*}
where $\kappa_{ij}= \deg(e_i)+ \epsilon_{j-1}(\deg(\alpha)-\epsilon_{j-1})+ \deg(f_i)(\deg(\alpha) + \deg(g) + \deg(f))+ \deg(f)\deg(e_i)+ \deg(g)\epsilon_{j-1}$ for $\epsilon_l=\sum_{k=1}^l \deg(a_k)-l$. 

\item For $\alpha=\overline{a_1}\otimes\cdots \otimes \overline{a_r}\otimes a_{r+1} \in C_{-r, *}(A,A), \beta=\overline{b_1}\otimes \cdots \otimes \overline{b_s}\otimes b_{s+1}\in C_{-s, *}(A, A)$ and $f\in C^{m, *}(A, A)$ such that $m-1\leq r+s$, then 
\begin{eqnarray*}
\lefteqn{
m_3(\alpha, f, \beta)=\sum_i\sum_{j=0}^s(-1)^{\lambda_{ij}}}\\
&&  \overline{b_1}\otimes \cdots \otimes \overline{b_j}\otimes \overline{e_i}\otimes \overline{
a_{w+1}}\otimes \cdots \otimes \overline{a_r} \otimes a_{r+1}f(\overline{a_1}\otimes \cdots \otimes \overline{a_w}\otimes
 \overline{f_i}\otimes \overline{b_{j+1}} \otimes \cdots \otimes \overline{b_s})b_{s+1}
\end{eqnarray*}
where the symbol $w$ means $m-s+j-1$ and $\lambda_{ij}= \deg(e_i)+(\sum_{l=1}^j\deg(b_l)-j)(\deg(\alpha) + \deg(f))+ \deg(f_i)(\sum_{l=1}^j\deg(b_l)-j + \deg(\alpha) + \deg(f))+ 
\\(\sum_{l=1}^w\deg(a_l) - w)\deg(\alpha) + \deg(e_i)(\sum_{l=1}^jb_l-j)+ \deg(f)(\sum_{l=1}^w\deg(a_l)-w)$.
\end{enumerate}
In case (5) if $r> m+n$ then $m_3(f, \alpha, g)=0$. Similarly, in case (6) if $m-1>r+s$ we have $m_3(\alpha,f, \beta)=0$. 
\end{remark}

\begin{proposition}\label{prop-infinity-algebra}
$(\calD^*(A, A), (m_1, m_2, \cdots), \langle\cdot, \cdot\rangle)$ is a strictly unital (cf. \cite[Definition 4.1]{KoSo}) cyclic $A_{\infty}$-algebra with $m_k=0$ for $k\geq 4$, namely,  
$$
\langle m_p(\alpha_0\otimes\cdots \otimes \alpha_{p-1}), \alpha_p\rangle=(-1)^{\deg(\alpha_0)(\deg(\alpha_1)+ \cdots  + \deg(\alpha_p))} \langle m_p(\alpha_1\otimes\cdots\otimes \alpha_p), \alpha_0\rangle
$$
for any $\alpha_0, \cdots, \alpha_p\in \calD^*(A, A)$, where $\langle\cdot, \cdot\rangle$ is the pairing on $\calD^*(A, A)$ (defined in Section \ref{section5}) induced by the pairing of the dg symmetric Frobenius algebra $A$. 
\end{proposition}

\begin{proof}
To prove that $m_k=0$ for $k\geq 4$ we first prove the following claim.
\begin{Claim}\label{claim}
We have that the following expressions vanish:
\begin{enumerate}
\item If $\alpha_1\in C^*(A,A)$ then $h(\iota(\alpha_1) \cup \iota(\alpha_2))=0$ for any $\alpha_2\in \calD^{*}(A, A).$
\item If $\alpha_1, \alpha_2 \in C_*(A,A)$ then $h( \iota(\alpha_1) \cup \iota(\alpha_2) )=0.$
%\item If $\alpha_1\in C^*(A,A)$ and $\alpha_2 \in C_*(A,A)$ then $h(\iota(\alpha_1) \cup \iota(\alpha_2))=0$
\item For any  $\alpha_1, \alpha_2, \alpha_3 \in \calD^*(A,A)$ we have  
        $$h(h(\iota(\alpha_1)\cup \iota(\alpha_2))\cup \iota(\alpha_3))=0$$ and 
         $$h(\iota(\alpha_1) \cup h(\iota(\alpha_2) \cup \iota(\alpha_3)))=0.$$
\end{enumerate}
\end{Claim}
\begin{proof}[Proof of Claim]
Identities (1) and (2) are easy to check. Let us check the first identity in (3), namely $h(h(\iota(\alpha_1)\cup \iota(\alpha_2))\cup \iota(\alpha_3))=0$. By (1) this identity holds if $\alpha_1 \in C^*(A, A)$ and by (2) it holds if $\alpha_1,\alpha_2 \in C_*(A,A)$. It remains to check it holds when $\alpha_1\in C_*(A, A)$ and $\alpha_2\in C^*(A, A)$. In this case,
%\begin{eqnarray*}\label{equation-17}
%\lefteqn{
%h(\iota(\alpha_1)\cup \iota(\alpha_2))=}\\
%&& \sum_{i=1}^{\min\{m, r\}} \sum_{k_1} \cdots \sum_{k_i} \pm \epsilon(\overline{e_{k_1}x_1})\cdots \epsilon(\overline{e_{k_{i-1}}x_{i-1}}) \overline{e_{k_{i}}}\otimes \overline{x_{i}}\otimes
%\cdots \otimes x_{m+1} \alpha_2(\overline{f_{k_1}}
%\otimes \cdots \otimes \overline{f_{k_{i}}}\otimes id^{\otimes {m-i}}),\nonumber
%\end{eqnarray*}
\begin{equation}\label{equation-17}
\begin{split}
&h(\iota(\alpha_1)\cup \iota(\alpha_2))= \sum_{i=1}^{\min\{m, r\}} \sum_{k_1} \cdots \sum_{k_i} \\
&\pm \epsilon(\overline{e_{k_1}x_1})\cdots \epsilon(\overline{e_{k_{i-1}}x_{i-1}}) \overline{e_{k_{i}}}\otimes \overline{x_{i}}\otimes
\cdots \otimes x_{m+1} \alpha_2(\overline{f_{k_1}}
\otimes \cdots \otimes \overline{f_{k_{i}}}\otimes \id^{\otimes {m-i}})
\end{split}
\end{equation}
where we wrote $$\alpha_1:=\overline{x_1}\otimes \cdots \otimes \overline{x_m}\otimes x_{m+1}.$$ Then for any $\alpha_3\in \calD^*(A, A)$, we have that  $$h( h(\iota(\alpha_1)\cup \iota(\alpha_2))\cup \iota(\alpha_3)))=0.$$  Indeed, we note that  the left most tensor element in each monomial in the sum given by $h(\iota(\alpha_1)\cup \iota(\alpha_2))$ is $e_{k_i}$, so  $h(\iota(\alpha_1)\cup \iota(\alpha_2))\cup \iota(\alpha_3) \in \calC_{\sg}^*(A,A)$ is a sum of terms each of which either has $e_{k_i}$ as the first output (leg) or lies in $\iota ( C_{*}(A, A) )$ otherwise. Note that the action of $h$ on those terms with  $e_{k_i}$  as the first output is zero because of the fact that $\sum \epsilon(e_{k_i}) \overline{f_{k_i}}=0$ in $\overline A$,  and the action of $h$ on $\iota( C_{*}(A, A) ) $ is zero as well because of dimension reasons. Hence, $h( h(\iota(\alpha_1)\cup \iota(\alpha_2))\cup \iota(\alpha_3)))=0.$
%\item If $\alpha_1\in C_{-r, *}(A, A)$ and $\alpha_2\in C^{m, *}(A, A)$ such that $m-r\leq 0$, then
%\begin{eqnarray*}\label{equation-17}
%\lefteqn{
%h(\iota(\alpha_1)\cup \iota(\alpha_2))=}\\
%&& \sum_{i=1}^{\min\{m, r\}} \sum_{k_1} \cdots \sum_{k_i} \pm \epsilon(\overline{e_{k_1}x_1})\cdots \epsilon(\overline{e_{k_{i-1}}x_{i-1}}) \overline{e_{k_{i}}}\otimes \overline{x_{i}}\otimes
%\cdots \otimes x_{m+1} \alpha_2(\overline{f_{k_1}}
%\otimes \cdots \otimes \overline{f_{k_{i}}}\otimes id^{\otimes {m-i}}),\nonumber
%\end{eqnarray*}
The identity $$h( h( \iota(\alpha_1)\cup \iota(\alpha_2))\cup \iota(\alpha_3)))=0$$ for any $\alpha_i \in \calD^*(A, A)$ ($i=1, 2, 3$) follows from a similar argument. 
\end{proof}
%Use Claim \ref{claim} to prove $m_k=0$ for $k>3$. 
%\begin{equation*}
%\begin{split}
%h(  \iota(\alpha_1)\cup h(\iota(\alpha_2)\cup \iota(\alpha_3)))&=0,\\
%h( h( \iota(\alpha_1)\cup \iota(\alpha_2))\cup \iota(\alpha_3)))&=0
%\end{split}
%\end{equation*}
%for any elements $\alpha_i\in \calD^*(A, A)$ ($i=1, 2, 3$). Indeed, if either  $\alpha_3 \in C_{<0, *}(A, A)$ or $ \alpha_2\in C^{>0, *}(A, A)$, then it is clear that $h(\iota(\alpha_2)\cup\iota( \alpha_3))=0$, thus $$h(  \iota(\alpha_1)\cup h(\iota(\alpha_2)\cup \iota(\alpha_3)))=0.$$ So let us assume that $\alpha_3\in C^{m, *}(A, A)$ and $\alpha_2\in C_{-r, *}(A, A)$. In this case, by the definition of $h$ (cf. Definition \ref{lemma-definition}), we have that 
From Claim \ref{claim} above and the Homotopy Transfer Theorem, we have 
$$
m_4(\alpha_1, \alpha_2, \alpha_3, \alpha_4)=\pm \Pi(h(\iota(\alpha_1)\cup \iota(\alpha_2))\cup h(\iota(\alpha_3)\cup \iota(\alpha_4))).
$$
The right hand side of the above equation also vanishes. Indeed, by Claim \ref{claim}, we may assume that $\alpha_1, \alpha_3\in C_{*}(A, A)$ and $\alpha_2,\alpha_4\in C^*(A, A)$. Then from the identity (\ref{equation-17}) above it follows that 
$h(\iota(\alpha_1)\cup \iota(\alpha_2))\cup h(\iota(\alpha_3)\cup \iota(\alpha_4))$ is the sum of terms with the first output  $e_i$, where 
$\sum_{i} e_i \otimes f_i =\Delta(1)$. Hence, $ \Pi(h(\iota(\alpha_1)\cup \iota(\alpha_2))\cup h(\iota(\alpha_3)\cup \iota(\alpha_4)))$ is 
a sum of terms containing $\sum_{i} \epsilon(e_i)\overline{f_i}=0$, and so $m_4(\alpha_1, \alpha_2,\alpha_3, \alpha_4)=0$. 

 In order to check $m_k=0$ for $k>4$, because of Claim \ref{claim} it is sufficient to check the following identity $$h(h(\iota(\alpha_1)\cup \iota(\alpha_2))\cup h(\iota(\alpha_3)\cup \iota(\alpha_4)))=0,$$ for any $\alpha_i\in \calD^*(A, A)$ ($i=1, \cdots, 5$) and the proof of this identity is completely analogous to the one of $m_4=0$ above.
 \\
 \\
We now check the cyclic compatibilities (or Calabi-Yau conditions). The cyclic compatibility for $m_2$ was verified in Lemma 5.4. We proceed to verify the identity $$\langle m_3(\alpha_0, \alpha_1, \alpha_2), \alpha_3\rangle = \pm \langle m_3(\alpha_1,  \alpha_2, \alpha_3), \alpha_0 \rangle$$ for any $\alpha_i\in \calD^*(A, A)$ ($i=0, \cdots, 3$).  Based on the formulae for $m_3$ given in Remark 6.4, it is sufficient to check the case where  $\alpha_0= \overline{a_1} \otimes \cdots  \otimes \overline{a_r} \otimes a_{r+1} \in C_{-r, *}(A, A), \alpha_2=  \overline{b_1} \otimes \cdots  \otimes \overline{b_s} \otimes b_{s+1}\in C_{-s, *}(A, A)$ and $\alpha_1\in C^{m, *}(A, A)$, $\alpha_3\in C^{r+s-m+1, *}(A,A)$, and $r+s\geq m$. We have
\begin{eqnarray*}
\lefteqn{\langle m_3(\alpha_0, \alpha_1, \alpha_2), \alpha_3\rangle=\sum_i\sum_{j=1}^{\min\{m, n, s\}}\pm}\\
&&  \langle \alpha_3( \overline{b_1}\otimes \cdots \otimes \overline{b_{j-1}}\otimes \overline{e_i}\otimes \overline{
a_{m-s+j+1}}\otimes \cdots \otimes \overline{a_r}),  a_{r+1}\alpha_1(\overline{a_1}\otimes \cdots \otimes 
 \overline{f_i}\otimes  \cdots \otimes \overline{b_s})b_{s+1}\rangle.
\end{eqnarray*}
Similarly, we also have
\begin{eqnarray*}
\lefteqn{\langle \alpha_0, m_3(\alpha_1,  \alpha_2, \alpha_3)\rangle=\sum_i\sum_{j=1}^{\min\{m, n, s\}}\pm}\\
&& \langle a_{r+1}, \alpha_1(\overline{a_1}\otimes \cdots\otimes \overline{a_{m-s+j}}\otimes\overline{e_i}\otimes \overline{b_{j}}\otimes \cdots \otimes \overline{b_r})b_{s+1} 
\alpha_3(\overline{b_1}\otimes \cdots \otimes \overline{b_{j-1}}\otimes \overline{f_i}\otimes
\cdots \otimes \overline{a_r})\rangle.\\
\end{eqnarray*}
Therefore, by the compatibility between the product and pairing of $A$ and by the symmetry of the pairing, we have
$$
\langle m_3(\alpha_0, \alpha_1, \alpha_2), \alpha_3\rangle=\pm \langle \alpha_0, m_3(\alpha_1,  \alpha_2, \alpha_3)\rangle,
$$ 
so the cyclic compatibilities hold. On the other hand, the strictly unital condition holds since we have $m_3(\alpha_1, \alpha_2, \alpha_3)=0$ 
if one of the three elements $\alpha_1, \alpha_2$ and $\alpha_3$ is $1\in C^0(A, A)$. Therefore we obtain a strictly unital cyclic $A_{\infty}$-algebra 
structure on $\calD^*(A, A)$. 
\end{proof}

%\section{An interpretation from TCFT}

%\section{BV-algebra structure on singular Hochschild cohomology }

\begin{corollary}
Let  $A$ be a dg symmetric Frobenius algebra then  $(\HH_{\sg}^*(A,A), \cup, \{\cdot, \cdot\}, \widetilde{\Delta})$ is a BV-algebra, where 
$\widetilde{\Delta}$ is defined in Section \ref{subsection-BV} above.
\end{corollary}
 
\begin{proof}
%prove $\star$ agrees with $\cup$ in cohomology. prove $\star$ is commutative in cohomology. prove BV equation in $\THH^*(A,A)$ using the fact that the bracket defined for tate hoch agrees in cohomology with the one which defines a BV algebra, etc\cdots  and finally use $\THH^*(A,A) \cong  \HH_{sg}^*(A,A)$. Actually there are several ways of doing this, for example can prove that $\THH^*(A,A), \star, [ \cdot, \cdot]$ is a gerstenhaber algebra  by comparing with the cohomology of $C_{sg}^*(A,A)$ and then show that in cohomology $[\cdot, \cdot]$ equals the failure of $\widetilde{\Delta}$ in being a derivation of $\star$. 
From Theorem \ref{theorem-ger} it follows that 
$(\HH_{\sg}^*(A, A), \cup, \{\cdot, \cdot\})$ is a Gerstenhaber algebra. It remains to verify (\ref{equation-bv-identity})
\begin{equation}\label{equation-bv-identity}
\{x, y\}=\widetilde{\Delta}(x)\cup y \pm \widetilde{\Delta}(x\cup y)\pm x\cup \widetilde{\Delta}(y).
\end{equation}
On the other hand,  from Proposition \ref{prop-BV-identity} it follows that such an identity holds for the Lie bracket $[\cdot, \cdot]$ on $H^*(\calD^*(A, A))$, namely,  $$[x, y]=\widetilde{\Delta}(x)\star y \pm \widetilde{\Delta}(x\star y)\pm x\star\widetilde{\Delta}(y).$$ Under the canonical isomorphism $H^*(\calD^*(A, A))\cong \HH_{\sg}^*(A, A)$, we have that  $[\cdot, \cdot]=\{\cdot, \cdot\}$ and $\cup =\star$ from Proposition \ref{prop-infinity-algebra}, thus Identity (\ref{equation-bv-identity}) holds.
\end{proof}
\begin{remark}
This result was obtained in \cite{EuSc} in the case where $A$ is an ordinary periodic (i.e. $A \cong \Omega^n(A)$ in $\DD_{\sg}(A\otimes A^{\op})$ for some $n\in\Z_{>0}$) symmetric Frobenius algebra
and then was generalized to any (ordinary) symmetric Frobenius algebra in \cite{Wan}.
\end{remark}

\section{An application to string topology}
Let $M$ be a simply-connected closed manifold of dimension $k$. Let $C^*(M, \mathbb{K})$ be the singular  cochain complex over a field $\mathbb{K}$, which is a differential graded associative $\mathbb{K}$-algebra.  We also denote $C^*(M, \mathbb{K})$ by $C^*(M)$ for short, if the base field $\mathbb{K}$ is fixed. By the main result in \cite{LaSt}, there is a cdga $(A, d)$ and a zig-zag of quasi-isomorphisms between   $(A, d)$ and $C^*(M)$.  $$(A, d) \xleftarrow{\cong} \cdots \xrightarrow{\cong} C^*(M)$$ such that $(A, d)$ is a simply-connected commutative differential graded Frobenius algebra of degree $k$ and, moreover, the induced isomorphism $H^*(A) \cong H^*(M;\mathbb{K})$ is an isomorphism of Frobenius algebras.  We call such $(A, d)$ a Frobenius cdga-model of $M$.  Denote by $LM$  the free loop space of $M$, namely, $LM:=\Map(S^1, M)$.  In this  section we prove an invariance result for singular Hochschild cohomology and apply it to compute the singular Hochschild cohomology $\HH_{\sg}^*(C^*(M), C^*(M))$. More precisely, we will prove the following.

\begin{theorem}\label{theorem-main}
Let $M$ be a simply-connected closed manifold of dimension $k$. Let $\mathbb{K}$ be a field of characteristic zero. Then we have
\begin{enumerate}
\item if the Euler characteristic $\chi(M)=0$, then 
\begin{equation*}
\HH_{\sg}^i(C^*(M), C^*(M))=
\begin{cases}
H_{k-i}(LM)& \mbox{if $i<k-1$,}\\
H_1(LM)\oplus H^0(LM) & \mbox{if $i=k-1$,}\\
H_0(LM)\oplus H^1(LM) & \mbox{if $i=k$,}\\
H^{i-k+1}(LM)  & \mbox{if $i>k$;}
\end{cases}
\end{equation*}
\item if  the Euler characteristic $\chi(M)\neq 0$, then 
\begin{equation*}
\HH_{\sg}^i(C^*(M), C^*(M))=
\begin{cases}
H_{k-i}(LM)& \mbox{if $i\leq k-1$,}\\
H^{i-k+1}(LM)  & \mbox{if $i\geq k$.}
\end{cases}
\end{equation*}\end{enumerate}
\end{theorem}

\subsection{Pull-back and push-forward}
Let $(A, d_1)$ and $(B, d_2)$ be two differential (non-negatively) graded associative algebras over a field $\mathbb{K}$.  Let $f: (A, d_1)\rightarrow (B, d_2)$ be a morphism of dg $\mathbb{K}$-algebras. There is a pull-back functor $$f^*: \DD(\mbox{$B$-$\Modu$}) \rightarrow \DD(\mbox{$A$-$\Modu$})$$ induced from the forgetful functor (i.e. considering a left dg $B$-module as a left dg $A$-module via $f$) and a push-forward functor  $$f_*: \DD(\mbox{$A$-$\Modu$}) \rightarrow \DD(\mbox{$B$-$\Modu$})
$$
given by $ f_*(M):=B\otimes^{\LL}_A M$.
\begin{proposition}\label{proposition6.3}
If $f: (A, d_1)\rightarrow (B, d_2)$ is a quasi-isomorphism of dg algebras, then the functors  $f^*$ and ${f_*}$ are inverse quasi-equivalences between  $\DD(\mbox{$A$-$\Modu$})$ and $\DD( \mbox{$B$-$\Modu$})$. In particular, we obtain an equivalence $$\overline{f^*}: \DD_{\sg}(B)\xrightarrow{\cong} \DD_{\sg}(A).$$
\end{proposition}

\begin{proof}
Note that for $X\in \DD( \mbox{$A$-$\Modu$})$, we have 
\begin{equation*}
\begin{split}
f^*\circ {f_*}(X)&\cong B\otimes^{\LL}_A X,\\
%&\cong B(A, A, A)\otimes^{\LL}_A B(B, A, A)\otimes^{\LL}_A X\\
&\cong A\otimes_A B\otimes_A X\\
&\cong X.
\end{split}
\end{equation*}
Similarly, for $Y\in  \DD( \mbox{$B$-$\Modu$})$
\begin{equation*}
\begin{split}
{f_*}\circ f^*(Y)&\cong B\otimes^{\LL}_A Y\\
& \cong Y.
\end{split}
\end{equation*}
It follows that the functors $f^*$ and ${f_*}$ are inverse quasi-equivalences. Thus $f^*$ restricts to an equivalence between the subcategories of compact objects and also induces an equivalence between $\DD^b(\mbox{$A$-$\modu$})$ and $\DD^b(\mbox{$B$-$\modu$})$, the bounded derived category of finite generated  $B$-modules,   so we have an induced equivalence of Verdier quotients $$\overline{f^*}: \DD_{\sg}(B)\xrightarrow{\cong} \DD_{\sg}(A).$$
\end{proof}

\begin{remark}\label{remark6.4}
A morphism $f: (A, d_1)\rightarrow (B, d_2)$ of dg algebras induces a morphism of dg algebras $f\otimes f: A\otimes A^{\op}\rightarrow B\otimes B^{\op}$.  If $f$ is a quasi-isomorphism, then so is $f\otimes f$. From Proposition \ref{proposition6.3} it follows that  $$(\overline{f\otimes f})^*_{\proj}: \DD_{\sg}(B\otimes B^{\op})\rightarrow \DD_{\sg}(A\otimes A^{\op})$$ is an equivalence and sends $B$ to $A$. Therefore, we have an isomorphism $$(\overline{f\otimes f})^*_{\proj}: \HH_{\sg}^*(B, B)\rightarrow \HH_{\sg}^*(A, A).$$
\end{remark}

\subsection{Quasi-isomorphisms between singular Hochschild cochain complexes}
Let $f: (A, d_1)\rightarrow (B, d_2)$ be a morphism of dg algebras. Recall that $\Omega^p(A)$ and $\Omega^p(B)$ are the non-commutative $p$-differential forms of $A$ and $B$, respectively. As before, we use the identification provided by Lemma 3.12. Then $f$ induces a morphism $\Omega^p(f): \Omega^p(A)\rightarrow \Omega^p(B)$ for $p\in \Z_{\geq 0}$ given by 
$$
\Omega^p(f)(\overline{a_1}\otimes \cdots \otimes \overline{a_p}\otimes a_{p+1}):=
\overline{f(a_1)}\otimes \cdots \otimes \overline{f(a_{p})}\otimes f(a_{p+1})
$$
where we use Lemma \ref{lemma-bimodule} to identity $\Omega^p(A)$ with $(\sA)^{\otimes p}\otimes A$. It is clear that  $\Omega^p(f)$ is a morphism of $A$-$A$-bimodules. Moreover, if $f$ is a quasi-isomorphism, then so is $\Omega^p(f)$. 
\\

Now let us construct a singular Hochschild cochain complex $\calC_{\sg}^*(A, B)$ with coefficients in $B$. Consider the Hochschild cochain complex $C^*(A, \Omega^p(B))$ with coefficients in the $A$-$A$-bimodule $\Omega^p(B)$. We define a morphism of cochain complexes $$\theta^p_{A, B}: C^*(A, \Omega^p(B))\rightarrow C^*(A, \Omega^{p+1}(B))$$ which sends $\alpha\in C^{m, *}(A, \Omega^p(B))$ to the element  
$$
\overline{a_1}\otimes \cdots \otimes \overline{a_{m+1}}\otimes a_{m+2}\mapsto \overline{f(a_1)}\otimes \alpha(\overline{a_2}\otimes
\cdots \overline{a_{m+1}}\otimes a_{m+2}).
$$
Then we define the singular Hochschild cochain complex of $A$ with coefficients in $B$ as 
$$
\calC_{\sg}^*(A, B):= \lim_{\substack{\longrightarrow\\p \in \Z_{\geq 0}}} C^{*}(A, \Omega^{p}(B)).
$$
with the induced differential. 
\\

Observe that, for each $p\in \Z_{\geq 0}$, there is a zig-zag of morphisms of cochain complexes  
\begin{equation}\label{equation-zig-zag}
\xymatrix@C=4pc{
C^*(A, \Omega^p(A)) \ar[r] ^-{C^*(A,\Omega^p( f))}& C^*(A, \Omega^p(B)) & C^*(B, \Omega^p(B))\ar[l]_-{C^*(f, \Omega^p(B))}.
}
\end{equation}
These zig-zags are compatible with the inductive systems, thus we obtain a zig-zag of morphisms  between singular Hochschild cochain complexes
$$
\xymatrix@C=4pc{
\calC_{\sg}^*(A, A) \ar[r] ^-{\calC_{\sg}^*(A, f)}& \calC_{\sg}^*(A, B) & \calC^*_{\sg}(B, B)\ar[l]_-{\calC_{\sg}^*(f, B)}.
}
$$

\begin{proposition}\label{proposition-7.4}
Let $f: (A, d_1)\rightarrow (B, d_2)$ be a quasi-isomorphism of dg algebras. Then $\calC_{\sg}^*(A, f)$ and $\calC_{\sg}^*(f, B)$ are both quasi-isomorphisms, namely, the zig-zag above is one of quasi-isomorphisms. 
\end{proposition}

\begin{proof}
Note that all three complexes in the zig-zag (\ref{equation-zig-zag}) have complete decreasing filtrations with  the associated quotients 
$$
\xymatrix{
\Hom(\overline{A}^{\otimes m}, \Omega^p(A))\ar[r] & \Hom(\overline{A}^{\otimes m}, \Omega^p(B)) & 
\Hom(\overline{B}^{\otimes m}, \Omega^p(B))\ar[l]
}
$$
It is clear that these associated quotients are quasi-isomorphic, and thus by the usual spectral sequence argument we may conclude that the morphisms in the zig-zag in (\ref{equation-zig-zag}) are quasi-isomorphisms for each $p\in \Z_{\geq 0}$. Therefore, it follows that the morphisms $\calC_{\sg}^*(A, f)$ and $\calC_{\sg}^*(f, B)$ are quasi-isomorphisms.
\end{proof}
\begin{remark}
The zig-zag of chain complex quasi-isomorphisms $$
\xymatrix@C=4pc{
\calC_{\sg}^*(A, A) \ar[r] ^-{\calC_{\sg}^*(A, f)}& \calC_{\sg}^*(A, B) & \calC^*_{\sg}(B, B)\ar[l]_-{\calC_{\sg}^*(f, B)}.
}
$$
induces an isomorphism in cohomology
$$
\xymatrix@C=4pc{
\HH_{\sg}^*(f): H^*(\calC_{\sg}^*(A, A)) \ar[r]^-{H^*(\calC_{\sg}^*(A, f))}  & H^*(\calC_{\sg}^*(A, B)) \ar[r]^-{H^*(\calC_{\sg}^*(f, B))^{-1}}  & H^*(\calC_{\sg}^*(B, B))
}
$$
which are, in fact, isomorphisms of Gerstenhaber algebras. The algebra structure in the middle object is induced by the composition
\begin{eqnarray*}
\HH^m(A, \Omega^p(B))\otimes \HH^n(A, \Omega^q(B))\rightarrow \HH^{m+n} (A, \Omega^p(B) \otimes_A \Omega^q(B))
\\ \cong \HH^{m+n}(A, \Omega^p(A) \otimes_A \Omega^q(A) ) \cong  \HH^{m+n}(A, \Omega^{p+q} (A) ) \cong \HH^{m+n}(A, \Omega^{p+q}(B))
\end{eqnarray*}
where the first map is given by the classical Hochschild cup product construction using the $A$-$A$-bimodule structure on $\Omega^i(B)$ for $i=p,q$ via $f: A \to B$, and the first and last isomorphisms are induced by the fact that $f: A \to B$ is a quasi-isomorphism.  The algebra structure on $H^*(\calC_{\sg}^*(A, A))$ and $H^*(\calC_{\sg}^*(B, B))$ is the cup product $\cup$ defined in Section \ref{subsection-dga}, which agrees with the classical Yoneda product $\cup'$ as remarked in the proof of Proposition \ref{proposition-dga}. On the other hand, it follows from \cite{Wan1} that $\HH_{\sg}^*(f)$ is an isomorphism of  Lie algebras since the derived tensor functor $B\otimes^{\mathbb{L}}_A-$ induces an equivalence between $\DD(A)$ and $\DD(B)$ as triangulated categories. 
\end{remark}

\subsection{The Frobenius cdga-model}
As we mentioned before, by a result of \cite{LaSt}, for any simply-connected closed manifold $M$ of dimension $k$ there is a dg symmetric Frobenius algebra $(A,d)$ of degree $k$ which is simply-connected (i.e. $A^0=\mathbb{K}$ and $A^1=0$) (cf. Definition \ref{definition-dg-symmetric}) together with a zig-zag of dga quasi-isomorphisms 
\begin{equation}\label{equation-zig-zag1}
(A, d) \xleftarrow{\cong} \cdots \xrightarrow{\cong} C^*(M)
\end{equation}
such that the induced isomorphism $H^*(A,d) \cong H^*(M)$ is one of Forbenius algebras.
\\

When $A$ is simply-connected, the  Tate-Hochschild complex $\calD^{*,*}(A, A)$ becomes quite simple to analyze, and thus we may compute its cohomology in terms of the  Hochschild homology and cohomology of $A$. 
\begin{lemma}\label{lemma6.7}
Let $A$ be a simply-connected dg symmetric Frobenius algebra of degree $k$ over a field $\mathbb{K}$. Then we have 
\begin{enumerate}
\item If the Euler characteristic $\chi(A):= \mu \circ   \Delta(1)=0$, then 
\begin{equation*}
\HH_{\sg}^i(A, A)=
\begin{cases}
\HH^i(A, A)& \mbox{if $i<k-1$,}\\
\HH^{k-1}(A, A)\oplus \HH_0(A, A) & \mbox{if $i=k-1$,}\\
\HH_1(A, A)\oplus \HH^k(A, A) & \mbox{if $i=k$,}\\
\HH_{i-k+1}(A, A)  & \mbox{if $i>k$;}
\end{cases}
\end{equation*}
\item if  the Euler characteristic $\chi(A)\neq 0$, then 
\begin{equation*}
\HH_{\sg}^i(A, A)=
\begin{cases}
\HH^i(A, A)& \mbox{if $i\leq k-1$,}\\
\HH_{i-k+1}(A, A)  & \mbox{if $i\geq k$.}
\end{cases}
\end{equation*}\end{enumerate}

\end{lemma}
\begin{proof}
This is an immediate observation from the shape of the complex $\calD^{*,*}(A, A)$:
\begin{equation*}
\xymatrix@R=1.2pc{
& &0& \\
&& A^k \ar[r] \ar[u]  &  0\\
\ar@{.}[rd]&\vdots& \vdots  \ar[u] & \vdots \ar[u]  \\
0\ar[r]&\overline{A}^2\otimes A^0\ar[r] \ar@{.}[rd]\ar[u]& A^2  \ar[r]\ar[u] & 0\ar[u] \\
&0\ar[u]\ar[r]& A^1     \ar[u] \ar[r]    \ar@{.}[rd]             &  0\ar[u]  \\
&0\ar[r] \ar[u]& A^0\ar[u] \ar[r]^-{\chi}     &  A^k \ar[u]\ar@{.}[rd] \ar[r] & 0  \\
&& 0 \ar[u]\ar[r] & A^{k-1} \ar[u] \ar[r] & 0\ar[u] \ar@{.}[rd] &\\
& &0 \ar[u]\ar[r] & A^{k-2} \ar[u]\ar[r] & \Hom(\overline{A}^2, A^k)\ar[u] \ar[r]  &0 \\
 &&\vdots \ar[u] & \vdots\ar[u] & \vdots \ar[u] &   \\
%& 0 \ar[u] \ar[r]& A^0\ar[u]\ar[r] & \bigoplus_{i=0}^n \Hom(\overline{A}^i, A^i)\ar[r] \ar[u]& \cdots\\
%& &  0\ar[u] \ar[r] & \bigoplus^n_{i=1} \Hom(\overline{A}^{i}, A^{i-1})\ar[u]\ar[r]  & \cdots 
}
\end{equation*}
We note that the non-zero terms (except $A^k$) of $\calD^{\geq 0, *}(A, A)$ are located below  the diagonal line (dotted line)   and similarly the non-zero terms of $\calD^{<0, *}(A,A)$ are located on or above the diagonal line. The elements on the diagonal line have the total degree $k$. It is clear from the diagram above that $\HH_{\sg}^i(A, A) $ is isomorphic to $\HH^i(A, A)$ for $i<k-1$ and isomorphic to $\HH_{i-k+1}(A, A)$ for $i>k$. Let us compute $\HH_{\sg}^i(A, A)$ for $i=k-1, k$ in the following two cases.
\begin{enumerate}
\item If the map $\chi$ is an isomorphism (equivalently, the Euler characteristic $\chi(A)\neq 0$) then both $A^0$ and $A^k$ are killed in the cohomology, hence $\HH_{\sg}^k(A, A) \cong \HH_1(A, A)$ and $\HH_{\sg}^{k-1}(A, A)\cong \HH^{k-1}(A, A)$. 
\item If the map $\chi$ is not an isomorphism, then $\chi$ is zero since $A^0$ and $A^k$ are one dimensional. Then we have $\HH_{\sg}^k(A, A) \cong \HH_1(A, A)\oplus \HH^k(A, A)$ and $\HH_{\sg}^{k-1}(A, A)\cong  \HH^{k-1}(A, A)\oplus \HH_0(A, A)$.
\end{enumerate}
\end{proof}

\begin{remark}\label{remark6.8}
If $A$ is a Frobenius cdga model for $M$, then $\chi(A)=\chi(M)$. Indeed, the zig-zag (\ref{equation-zig-zag1}) of quasi-isomorphisms of dg algebras induces an isomorphism of graded algebras between $H^*(A)$ and $H^*(M)$ and the pairings are compatible. Therefore $\chi(A)=\sum_i e_i f_i=\chi(M)$. 
\end{remark}

\begin{remark}\label{remark6.9}
From the zig-zag (\ref{equation-zig-zag1}) it follows that $$\HH_*(A, A)\cong \HH_*(C^*(M),C^*(M))$$ and $$\HH^*(A, A)\cong \HH^*(C^*(M), C^*(M)).$$ Recall that as shown in \cite{Jon}, there exists  linear isomorphisms$$H_*(LM)\cong \HH^{k-*}(C^*(M), C^*(M))$$ and $$H^*(LM) \cong \HH_*(C^*(M), C^*(M)).$$ Now let us prove the main Theorem \ref{theorem-main} of this section. 
\end{remark}

\begin{proof}[Proof of Theorem \ref{theorem-main}]
From Remark \ref{remark6.4},   the zig-zag (\ref{equation-zig-zag1}) above implies an  isomorphism $$\HH_{\sg}^*(C^*(M), C^*(M))  \cong \HH_{\sg}^*(A, A).$$ Then the result follows from Lemma \ref{lemma6.7},  Remark \ref{remark6.8} and \ref{remark6.9}. 
\end{proof}

\begin{remark}
For a differential graded algebra $(B, d)$, we do not know whether there is an isomorphism between $H^*(\calC_{\sg}(B, B))$ and $\HH_{\sg}^*(B, B)$. But for $C^*(M)$, it holds by using the Frobenius cdga-model. Namely, we have the following result.
\end{remark}

\begin{proposition}
If $M$ is a simply-connected closed manifold, then $$H^*(\calC_{\sg}(C^*(M), C^*(M))\cong \HH_{\sg}^*(C^*(M), C^*(M)).$$
\end{proposition}
\begin{proof}
Let $(A, d)$ be a Frobenius cdga-model for $M$. From Proposition \ref{proposition-7.4} it follows that $$H^*(\calC_{\sg}^*(C^*(M), C^*(M)))\cong H^*(\calC_{\sg}^*(A, A)).$$ From Remark \ref{remark6.4} we have $\HH_{\sg}^*(C^*(M), C^*(M))\cong \HH_{\sg}^*(A, A)$. So the isomorphism $$H^*(\calC_{\sg}^*(C^*(M), C^*(M)))\cong \HH_{\sg}^*(C^*(M), C^*(M))$$ holds since $\HH_{\sg}^*(A, A)\cong H^*(\calC_{\sg}^*(A, A))$ from Corollary \ref{corollary3.26}.
\end{proof}

\begin{example}
Let $S^n$ be the $n$-dimensional sphere for $n >1$. It is well-known that $S^n$ is a {\it formal} manifold (cf. \cite{DGMS}), namely, the singular cochain complex $C^*(S^n;\mathbb{K})$ and the singular cohomology $H^*(S^n;\mathbb{K})$ are quasi-isomorphic as dg algebras, where the latter is equipped with the trivial differential and $\mathbb{K}$ is a field of characteristic zero.  Therefore, there is an isomorphism between $\HH_{\sg}^*(C^*(S^n), C^*(S^n))$ and $\HH_{\sg}^*(H^*(S^n), H^*(S^n))$ from Remark \ref{remark6.4}. As a dg algebra, $H^*(S^n)\cong \mathbb{K}[\epsilon_n]/(\epsilon_n^2),$ the graded ring of  dual numbers  with the generator $\epsilon_n$ in degree $n$. It is clear that $A_n:=\mathbb{K}[\epsilon_n]/(\epsilon_n^2)$ is a dg symmetric Frobenius algebra (with the trivial differential). Now let us compute the singular Hochschild cohomology $\HH_{\sg}^*(A_n, A_n)$ via the Tate-Hochschild complex. We have the following two cases.

Let $n$ be odd. The double complex $\calD^{*,*}(A_n, A_n)$ associated to the Tate-Hochschild complex is as follows:
\begin{equation*}
\xymatrix@C=1pc{
\mbox{$3n$-row:}&0\ar[r] & \epsilon_n^{\otimes 2}\otimes \mathbb{K}\ar[r] &  \epsilon_n\otimes \epsilon_n\ar[r]  &  0\\
\mbox{$2n$-row:}&&0\ar[r]&\epsilon_n\otimes \mathbb{K}\ar[r] &\epsilon_n \ar[r] &      0                &  &    \\
\mbox{$n$-row:}&&&0\ar[r] &       \mathbb{K}  \ar[r]^-{\chi}               &\epsilon_n\ar[r] & 0\\
\mbox{$0$-row:}&&&&         0 \ar[r]              &\mathbb{K} \ar[r]               & \Hom(\epsilon_n, \epsilon_n)\ar[r] &0\\
\mbox{$-n$-row:}&&& &                       &                 0\ar[r]        & \Hom(\epsilon_n, \mathbb{K}) \ar[r]               &  \Hom(\epsilon_n^{\otimes 2}, \epsilon_n)  \ar[r]& 0,
}
\end{equation*}
where $\chi=0$ since the Euler characteristic $\chi(S^n)$ is zero. Notice that the vector space in each position is of dimension $1$. We claim that all the horizontal differentials vanish.  Indeed, for the positive row-degrees, the differentials are given by the Hochschild chains differential. Thus we have
$$
d(\epsilon_n^{\otimes p}\otimes \lambda)=\epsilon_n^{\otimes p-1}\otimes \epsilon_n\lambda-(-1)^{(n-1)(n-1)(p-1)} \epsilon_n^{\otimes p-1}\otimes
\lambda\epsilon_n=0.
$$
for $d: \epsilon_n^{\otimes p}\otimes \mathbb{K}\rightarrow \epsilon_n^{\otimes p}\otimes \epsilon_n$.  Similarly, we have that  the differentials for non-positive row degrees vanish.  Therefore,  we have the  isomorphism of graded algebras $\HH_{\sg}^*(A_n, A_n)\cong \mathbb{K}[x,x^{-1}] \otimes \Lambda[t],$ where $|x|=1-n$ and  $|t|=n$. So we have that  $$\HH_{\sg}^*(C^*(S^n), C^*(S^n))\cong  \mathbb{K}[x,x^{-1}] \otimes \Lambda[t]$$ where $|x|=1-n$, $|x|=n-1$, $|t|=n$, and $\Lambda[t]$ denotes the exterior algebra on the generator $t$. Moreover,  the BV-algebra structure is determined by
\begin{equation*}
\begin{split}
\widetilde{\Delta}(x^p\otimes t)&=p(x^{p-1}\otimes 1),\\
\widetilde{\Delta}(x^p\otimes 1)&=0.
\end{split}
\end{equation*}

Let $n$ be even. We have the analogous double complex $\calD^{*, *}(A_n, A_n)$, but now  $\chi$ is an isomorphism since the Euler characteristic $\chi(S^n)=2$ is non-zero in $\mathbb{K}$. We claim that for the $(2p+1)n$-row ($p\in \Z$), the (non-trivial) horizontal differential is an isomorphism. Indeed, we have 
$$
d(\epsilon_n^{2p}\otimes \lambda)=\epsilon_n^{\otimes 2p}\otimes \epsilon_n\lambda-(-1)^{(n-1)(n-1)(2p-1)} \epsilon_n^{\otimes 2p}\otimes
\lambda\epsilon_n=2 \epsilon_n^{\otimes 2p}\otimes \lambda \epsilon_n,
$$
thus the (non-trivial) differential is an isomorphism for $p>0$ and by the same argument we have the analogous result for $p< 0$. Similarly, for  the $2pn$-row degree ($p\in \Z$), the horizontal differentials vanish. Therefore,  we have the isomorphism of graded algebras $\HH_{\sg}^*(A_n, A_n)\cong \mathbb{K}[x,x^{-1}] \otimes \Lambda[t]$ with $|x|=2(n-1)$ and $|t|=1$. Therefore,  $$\HH_{\sg}^*(C^*(S^n), C^*(S^n))\cong \mathbb{K}[x,x^{-1}] \otimes \Lambda[t]$$ with $|x|=2(n-1)$ and $|t|=1$.     The BV-algebra structure is determined by
\begin{equation*}
\begin{split}
\widetilde{\Delta}(x^p\otimes t)&=(2p+1)x^p\otimes 1,\\
\widetilde{\Delta}(x^p\otimes 1)&=0.
\end{split}
\end{equation*}
\begin{remark}
Our computation of the BV-algebra $$\HH_{\sg}^{*<n}(C^*(S^n), C^*(S^n))\cong H_{n-*}(LS^n)$$ agrees with the corresponding result \cite{Men2} for $n>1$. A BV-algebra structure was constructed in \cite{GoHi} on $H^*(LM,M)$ for any closed manifold $M$. One should be able to lift this BV-algebra structure to $H^{>0}(LM)$ and we conjecture that it coincides with  $\HH_{\sg}^{\geq n}(C^*(M), C^*(M)) \cong H^{>0}(LM)$ in the case $M$ is simply-connected. 
\end{remark}
\end{example}
\newpage

\section*{Appendix}
The following diagrams describe the product $\star: \calD^*(A,A) \otimes \calD^*(A,A) \to \calD^*(A,A)$ for all possible cases. Blue colored output legs always represent elements of $A$, while black colored output legs represent elements of $s\overline{A}$. The degree $0$ product of $A$ is denoted by $\mu: A \otimes A \to A$, while the degree $k$ coproduct is denoted by $\Delta: A \to A\otimes A$. A blue circle with white interior denotes the unit of the algebra $A$ and the map $\pi: A \to s\overline{A}$ is the natural projection. A solid black circle on the top of a single input leg in a corolla means that such an input leg is ``blocked'', namely, that it cannot received any elements and thus may be ignored and the corolla may be interpreted as an element of $(s\overline{A})^{\otimes m} \otimes A$ for some $m$.
\begin{figure}[h]
\centering
  \includegraphics[width=70mm,height=50mm]{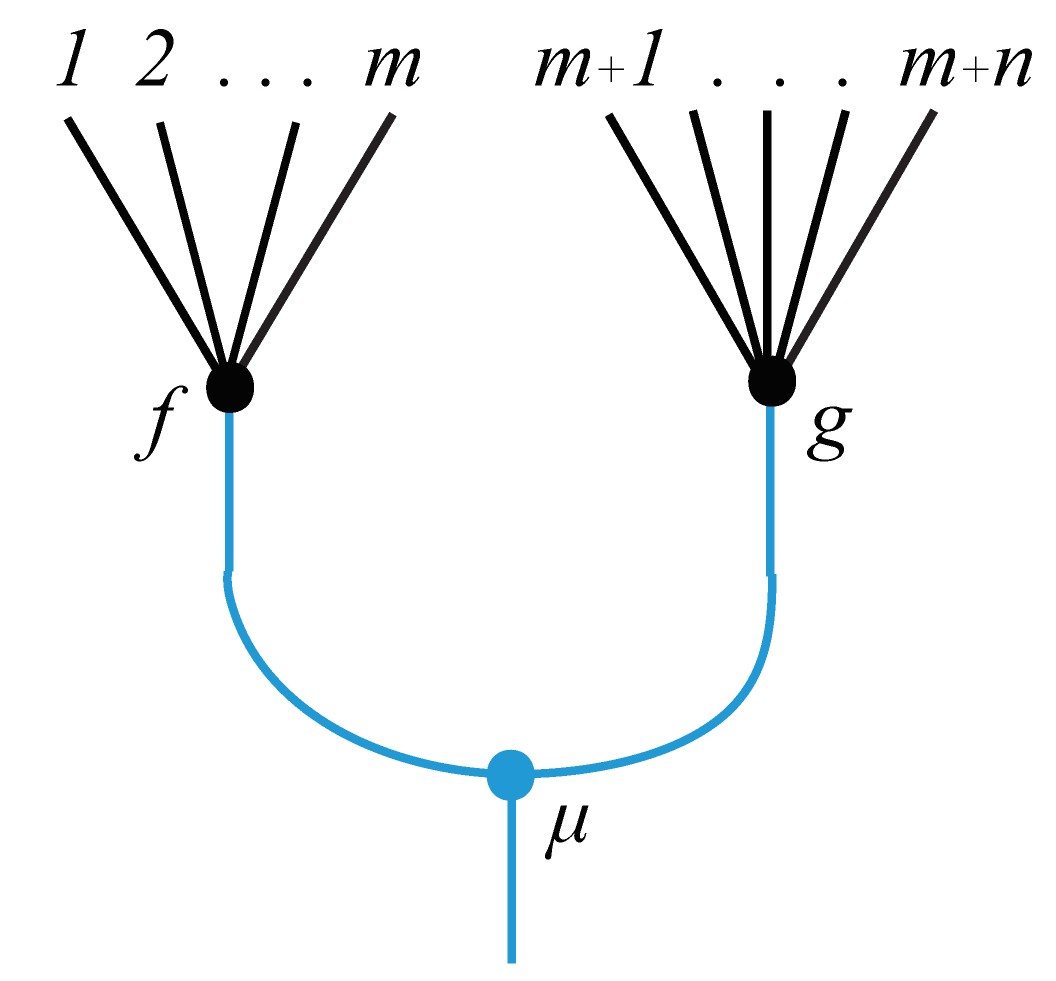}
  \caption{ For a pair of Hochschild cochains $f,g \in C^{*,*}(A,A)$ we have that $f \star g$ is defined to be the classical cup product $f \cup g$. In the above particular example of the above picture $m=4$,  $n=5$, and $f$ and $g$ are each represented by corollas with input legs colored black and with one output leg colored blue indicating that these outputs are elements in $A$.}
  \label{fig:hochcup}
\end{figure}

\begin{figure}[h]
\centering
  \includegraphics[width=70mm, height=55mm]{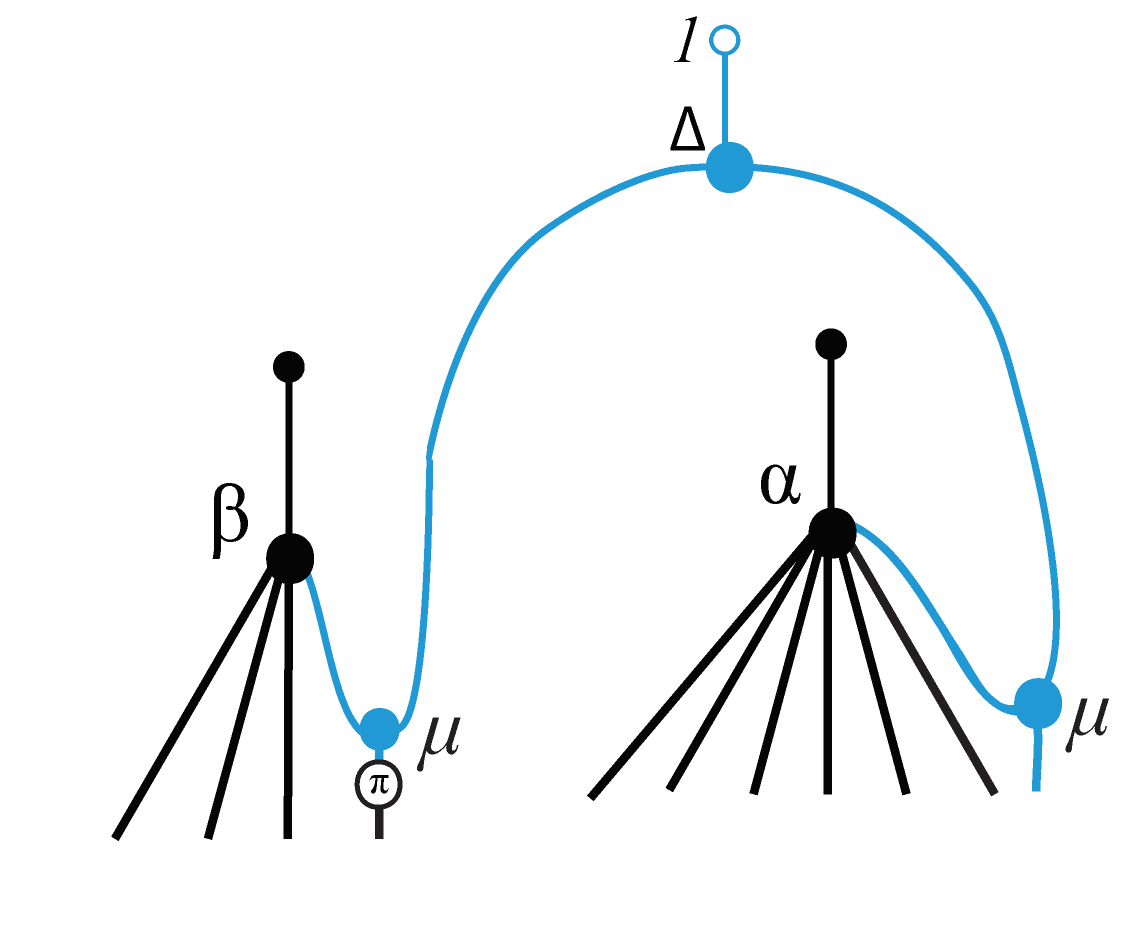}
  \caption{For a pair of Hochschild chains $\alpha, \beta \in C_{*,*}(A,A)$ we have that $\alpha \star \beta$ is defined by the diagram above. In the particular example of the above picture $m=6$,  $n=3$, and $\alpha \in (s\overline{A})^{\otimes m} \otimes A $ and $\beta \in (s\overline{A})^{\otimes n} \otimes A$ are represented by corollas with $m+1$ and $n+1$ output legs respectively and no input legs. In the literature, Hochschild chains are also represented by vertices in a circle and one of these vertices is marked; the marking corresponds to the color blue in our diagram.}
  \label{fig:starproduct}
\end{figure}

\begin{figure}[h]
\centering
  \includegraphics[width=60mm,height=40mm]{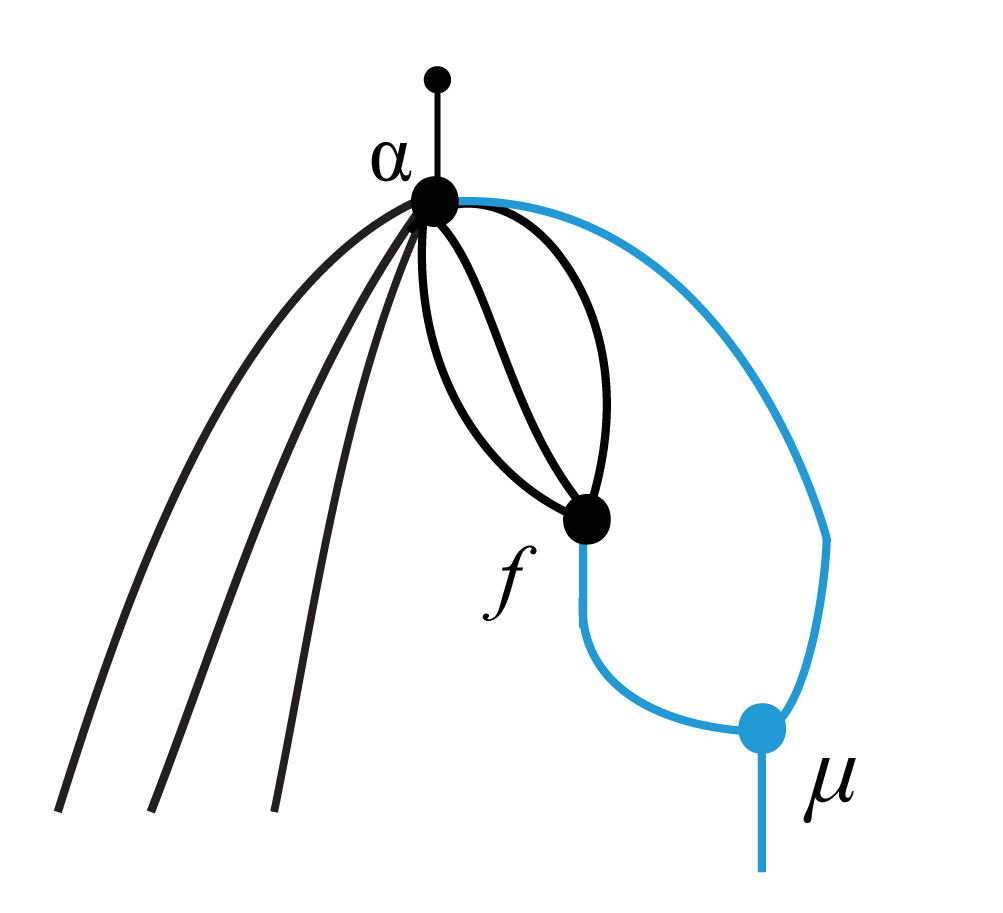}
  \caption{For $\alpha \in C_{-m,*}(A,A)$ and $f \in C^{n,*}(A,A)$ such that $m - n \leq 0$ we have that $ f \star \alpha$ is defined to be $ f\cap \alpha$, the classical cap product of a Hochschild cochain and a Hochschild chain. The above diagram describes such an operation. In the particular example of the above picture $m=6$ and $n=3$. The picture for $\alpha \star f$ is similar.}
  \label{fig:classicalcap}
\end{figure}

\begin{figure}[h]
\centering
  \includegraphics[width=70mm,height=60mm]{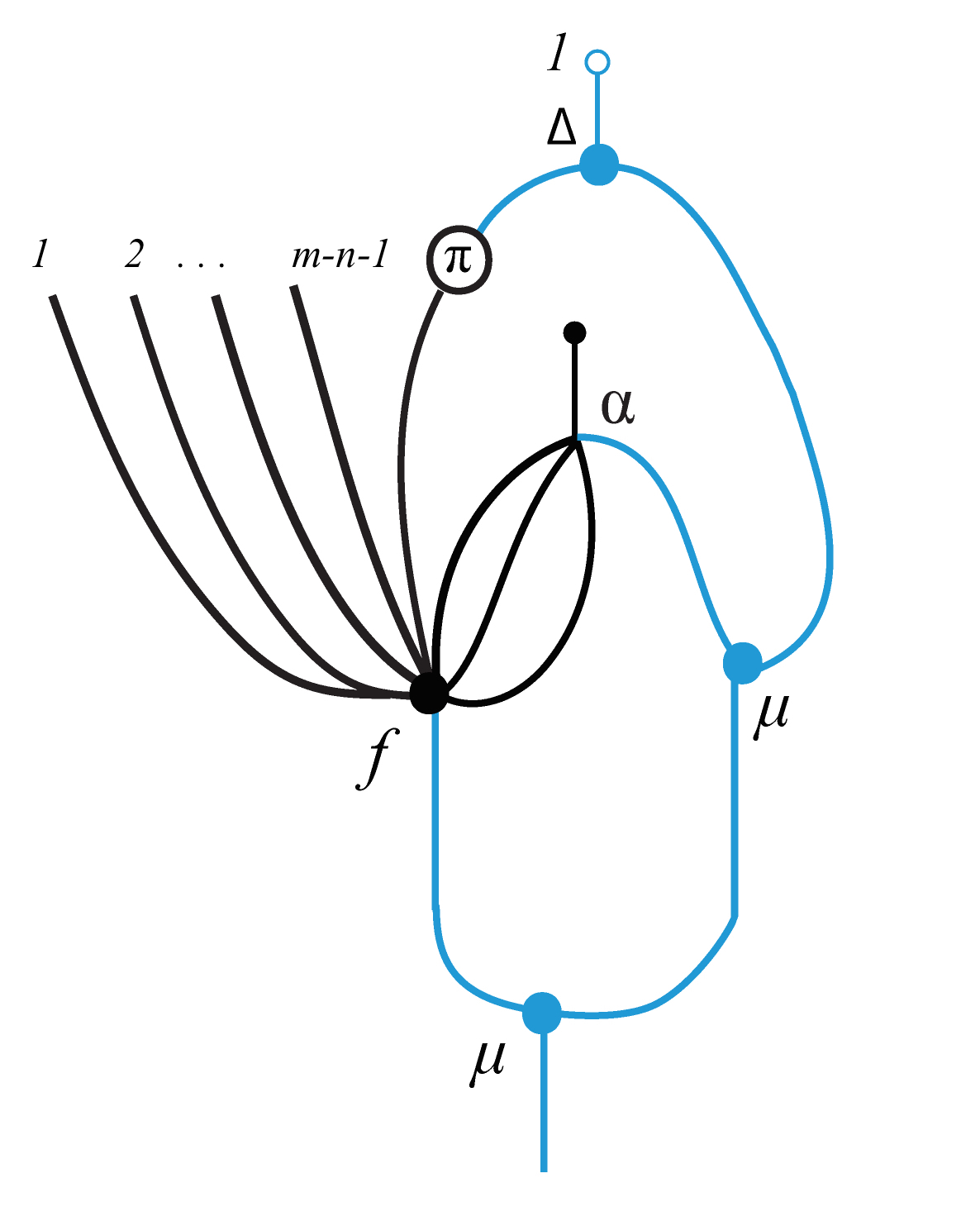}
  \caption{For $\alpha \in C_{-m,*}(A,A)$ and $f \in C^{n,*}(A,A)$ such that $m - n >0$ we have that $ f \star \alpha$ is defined by the diagram above. In the particular example of the above picture $m=3$ and $n=8$. Notice that, in this case, the product $\star$ uses the coproduct $\Delta: A \to A\otimes A$ of degree $k$ while the classical cap product does not use the data of the Frobenius structure of $A$. The diagram for $\alpha \star f$ is similar.}
  \label{fig:fcapa}
\end{figure}

\newpage

 \bibliographystyle{plain}

\end{document}